\documentclass[aop,seceqn,nameyear,MSNbibl,dvips]{arximspdf}

\doi{10.1214/10-AOP595}
\volume{39}
\issue{6}
\pubyear{2011}
\firstpage{2318}
\lastpage{2384}

\makeatletter

\renewcommand{\bar}{\overline}

\newtheorem{lemma}{Lemma}
\newtheorem{theorem}{Theorem}
\newtheorem{proposition}{Proposition}
\newtheorem{corollary}{Corollary}

\newproclaim{assumption}{Assumption}
\newproclaim{assm}{Assumption}
\newproclaim{assE}{Assumption}
\newproclaim{assK}{Assumption}
\newproclaim{assW}{Assumption}
\newproclaim{assL}{Assumption}

\newproclaim{remark}{Remark}

\newcommand{\rd}{{d}}
\newcommand{\ve}{\epsilon}
\newcommand{\e}{\varepsilon}
\newcommand{\cA}{\mathcal{A}}
\newcommand{\cE}{\mathcal{E}}
\newcommand{\cF}{\mathcal{F}}
\newcommand{\cH}{\mathcal{H}}
\newcommand{\cI}{\mathcal{I}}
\newcommand{\cJ}{\mathcal{J}}
\newcommand{\cK}{\mathcal{K}}
\newcommand{\cL}{L}
\newcommand{\cT}{\mathcal{T}}
\newcommand{\cV}{\mathcal{V}}
\newcommand{\cW}{\mathcal{W}}
\newcommand{\cX}{\mathcal{X}}
\newcommand{\cZ}{Z}

\newcommand{\bB}{\mathbb B}
\newcommand{\bE}{\mathbb E}
\newcommand{\bH}{\mathbb H}
\newcommand{\bL}{{\mathbb L}}
\newcommand{\bN}{{\mathbb N}}
\newcommand{\bP}{{\mathbb P}}
\newcommand{\bR}{{\mathbb R}}
\newcommand{\bZ}{{\mathbb Z}}
\newcommand{\mA}{\mathfrak{A}}
\newcommand{\mH}{\mathfrak{H}}
\newcommand{\mS}{\mathfrak{S}}
\newcommand{\mL}{\mathfrak{L}}
\newcommand{\ms}{\mathfrak{s}}
\newcommand{\mh}{\mathfrak{h}}
\newcommand{\mT}{\mathfrak{T}}
\newcommand{\mX}{\mathfrak{X}}
\newcommand{\mZ}{\mathfrak{Z}}

\newcommand{\sign}{\operatorname{sign}}

\makeatother

\begin{document}
\begin{frontmatter}

\title{Uniform bounds for norms of sums of independent random functions}
\runtitle{Uniform bounds for norms}

\begin{aug}
\author[A]{\fnms{Alexander} \snm{Goldenshluger}\thanksref{t2}\ead[label=e1]{goldensh@stat.haifa.ac.il}} and
\author[B]{\fnms{Oleg} \snm{Lepski}\corref{}\thanksref{t3}\ead[label=e2]{lepski@cmi.univ-mrs.fr}}
\runauthor{A. Goldenshluger and O. Lepski}
\affiliation{University of Haifa and Universit\'e Aix--Marseille I}
\address[A]{Department of Statistics\\
University of Haifa\\
Haifa 31905\\
Israel\\
\printead{e1}}
\address[B]{Laboratoire d'Analyse, Topologie\\
\quad and Probabilit\'es\\
Universit\'e Aix-Marseille 1\\
39, rue F. Joliot-Curie \\
13453 Marseille\\
France\\
\printead{e2}}
\end{aug}

\thankstext{t2}{Supported by the ISF Grant 389/07.}
\thankstext{t3}{Supported by the Grant ANR-07-BLAN-0234.}

\received{\smonth{4} \syear{2009}}
\revised{\smonth{2} \syear{2010}}

%
\begin{abstract}
In this paper, we develop a general machinery
for finding explicit uniform probability and moment bounds on
sub-additive positive functionals of random processes.
Using the developed general technique, we derive
uniform bounds
on the $\bL_s$-norms
of empirical and regression-type processes.
Usefulness of the obtained results is illustrated
by application
to the processes
appearing in kernel density estimation and in nonparametric
estimation
of regression functions.
\end{abstract}

%

\end{frontmatter}

\section{Introduction}
\label{sec:introduction}
\subsection{General setting}\label{subsec:general}
Let $\mS$ and $\mH$ be linear topological spaces, $(\Omega,\mA,\mathrm
{P})$ be a complete probability space,
and let $\xi_\theta\dvtx\Omega\to\mS$, $\theta\in\mH$
be a family of random mappings.
In the sequel, $\xi_\bullet(\omega)$ is assumed
linear and continuous
on~$\mH$ for any $\omega\in\Omega$.
Let $\Psi\dvtx\mS\to\bR_+$ be a given {\sl sub-additive} functional.
Suppose that
there exist functions $A\dvtx\mH\to\bR_+$, $B\dvtx\mH\to\bR_+$ and
$U\dvtx\mH\to\bR_+$
such that\looseness=-1
%
%
\begin{equation}\label{eq:individual-prob}
\mathrm{P}\{\Psi(\xi_\theta)-U(\theta)\geq z\}\leq g\biggl(\frac
{z^2}{A^{2}(\theta)+B(\theta)z}\biggr)\qquad\forall\theta\in\mH,
\end{equation}\looseness=0
where $g\dvtx\bR_+\to\bR_+$ is a monotone decreasing to zero function.

Let $\Theta$ be a fixed subset of $\mH$.
In this paper, under rather general assumptions on $U,A,B$ and $\Theta$,
we establish uniform probability and moment bounds
of the following type:
for any $\ve\in(0,1)$, $y>0$ and some $q\geq1$
%
%
\begin{eqnarray}
\label{eq:1}
\mathrm{P}\Bigl\{\sup_{\theta\in\Theta}\bigl[\Psi(\xi_\theta)-
u_\ve
\bigl(1+\sqrt{y}\lambda_A+y\lambda_B\bigr)U(\theta)\bigr]\geq0\Bigr\}
&\leq& P_{\ve,g}(y),
\\
\label{eq:2}
\mathrm{E} \sup_{\theta\in\Theta}\bigl[\Psi(\xi_\theta)-u_\ve
\bigl(1+\sqrt{y}\lambda_A+y\lambda_B\bigr)U(\theta)\bigr]_+^q
&\leq& E_{\ve,g}(y).
\end{eqnarray}
Here $\lambda_A$ and $\lambda_B$ are the quantities completely
determined by $U,A,\Theta$ and $U,B,\Theta$, respectively, and the inequalities
(\ref{eq:1}) and (\ref{eq:2}) hold if these quantities are finite;
$P_{\ve,g}(\cdot)$ and $E_{\ve,g}(\cdot)$ are continuous
decreasing to zero functions
completely determined by $\ve$ and $g$; and
the factor $u_\ve$ is such that $u_\ve\to1, \ve\to0$.
We present explicit expressions for all quantities appearing in~(\ref{eq:1}) and (\ref{eq:2}).

In order to derive
(\ref{eq:1}) and (\ref{eq:2}) from (\ref{eq:individual-prob}),
we assume that the set $\Theta$ is the image of a totally bounded
set in some metric space under a continuous mapping. Namely,
if $(\mZ,\mathrm{d})$ is a metric space, and
$\bZ$ is a totally bounded subset of $(\mZ,\mathrm{d})$ then we assume that
there exists a \textit{continuous} mapping $\phi$ from $\mZ$ to $\mH$
such that
%
%
\begin{equation}\label{eq:param}
\Theta=\{\theta\in\mH\dvtx\theta=\phi[\zeta], \zeta\in\bZ\}.
\end{equation}
Let $N_{\bZ,\mathrm{d}}(\delta)$, $\delta>0$
be the minimal number of balls of radius $\delta$ in
the metric $\mathrm{d}$ needed to
cover $\bZ$.
The inequalities (\ref{eq:1}) and (\ref{eq:2}) are proved under some
condition that relates $N_{\bZ,\mathrm{d}}(\cdot)$ and $g(\cdot)$.
It is worth mentioning that in particular examples the parametrization
$\Theta=\phi[\bZ]$ is often natural, while the metric $\mathrm{d}$ may have
a rather unusual form.

Inequalities
(\ref{eq:1}) and (\ref{eq:2}) can be considered as a
refinement of usual bounds
on the tail distribution of suprema
of random functions.
In particular, probability and moment bounds for
$\sup_{\theta\in\Theta}\Psi(\xi_\theta)$
can be easily derived from (\ref{eq:1}) and (\ref{eq:2}).
The well-known concentration results deal with deviation
of the supremum of a random process from the expectation
of this supremum, and estimation of the expectation is a separate
rather difficult problem. In contrast, in this paper we develop
\textit{explicit} uniform bounds on the whole trajectory $\{\Psi(\xi
_\theta
), \theta\in\Theta\}$.
The inequality in (\ref{eq:individual-prob})
provides the basic step in the development
of such uniform probability bounds.
The usual technique is based on the \textit{chaining argument}
that repeatedly applies inequality in (\ref{eq:individual-prob}) to
increments of the considered random process
[see, e.g., \citet{Ledoux-Tal} and
\citet{wellner}, Section 2.2].

The most interesting phenomena can be observed when a sequence of
random mappings $\{\xi^{(n)}_\theta, \theta\in\mH\}$, $n\in\bN^{*}$ is
considered.
There exists
a class of problems where
the quantities $\lambda_{A}$ and $\lambda_{B}$ depend on $n$, and
$\lambda_{A}\to0$,
$\lambda_{B}\to0$ as $n\to\infty$. Under these circumstances, one can
choose $y=y_n\to\infty$ and $\ve=\ve_n\to0$ such that
%
%
\begin{equation}
\label{eq1:introduction-new}
P_{\ve,g}(y_n)\to0,\qquad E_{\ve,g}(y_n)\to0,\qquad n\to\infty
\end{equation}
and, at the same time,
%
%
\begin{equation}
\label{eq2:introduction-new}
u_{\ve_n}\bigl(1+\sqrt{y_n}\lambda_A+y_n\lambda_B\bigr)U(\cdot)\to
U(\cdot),\qquad n\to
\infty.
\end{equation}
The relation in (\ref{eq1:introduction-new}) means that
$u_{\ve_n}(1+\sqrt{y_n}\lambda_A+y_n\lambda_B)U(\cdot)$ is indeed a~uniform
upper bound for $\Psi(\xi^{(n)}_\theta)$ on $\Theta$, while
(\ref{eq2:introduction-new})
indicates that for large $n$
this uniform bound is nearly as good as a nonuniform bound
$U(\cdot)$ given in~(\ref{eq:individual-prob}).
Typically\vadjust{\goodbreak} for a fixed $y>0$, we have
$P_{\ve,g}(y)\to\infty$ and $E_{\ve,g}(y)\to\infty$ as $\ve\to0$;
therefore, in order to get (\ref{eq1:introduction-new}) and
(\ref{eq2:introduction-new}), $\ve_n\to0$ and $y_n\to\infty$ should be
calibrated in an appropriate way.

The general setting outlined above includes important specific
problems that are in the focus of the present paper.
We consider
sequences of random mappings that are sums of real-valued
random functions defined on some measurable space (here the parameter
$n\in\bN^{*}$ is the number of summands).
We are interested in uniform bounds on the norms of such random functions;
thus the sub-additive functional of interest $\Psi$
is the $\bL_s$-norm, $s\geq1$.
First, the nonuniform bound (\ref{eq:individual-prob}) is established,
and then the inequalities of
the type (\ref{eq:1}) and (\ref{eq:2}) are derived.
It is shown that (\ref{eq1:introduction-new}) and (\ref{eq2:introduction-new})
hold under
mild assumptions on the parametric set $\Theta$.
We also discuss sharpness of the nonuniform inequality in
(\ref{eq:individual-prob}).
%
\subsection{Norms of sums of independent random functions}
Let $(\cT,\mT,\tau)$ and $(\cX,\mX,\nu)$
be $\sigma$-finite spaces, and let
$\cX$ be a separable Banach space.
Consider an $\cX$-valued random element $X$ defined on
the complete probability space $(\Omega,\mA, \mathrm{P})$
and having the density $f$ with respect to the measure $\nu$.
Let $\e$ be a real random variable defined on the same probability space,
independent of $X$ and having a symmetric
distribution.

For any $(\mT\times\mX)$-measurable function $w$ on $\cT\times\cX$ and
for any $t\in\cT$, $n\in\bN^*$, define the random functions
%
%
\begin{equation}\label{eq:processes}
\xi_w(t):=\sum_{i=1}^{n}[w(t,X_i)-\bE w(t,X)],\qquad
\eta_w(t):=\sum_{i=1}^{n}w(t,X_i)\e_i,
\end{equation}
where $(X_i,\e_i)$, $i\,{=}\,1,\ldots, n$, are independent copies of $(X,\e)$.
Put for \mbox{$1\,{\leq}\,s\,{<}\,\infty$}
\[
\|\xi_w\|_{s,\tau}=\biggl[\int|\xi_w(t)|^s
\tau(\rd t)\biggr]^{1/s},\qquad
\|\eta_w\|_{s,\tau}=\biggl[\int|\eta_w(t)|^s\tau(\rd t)
\biggr]^{1/s}.
\]

We are interested in uniform bounds of the type
(\ref{eq:1}) and (\ref{eq:2}) for $\|\xi_w\|_{s,\tau}$
and $\|\eta_w\|_{s,\tau}$ when $w\in\cW$, where $\cW$ is a given set of
$(\mT\times\mX)$-measurable functions. This setup
is a specific case of the general framework with
$\Psi(\cdot)=\mbox{$\|\cdot\|$}_{s,\tau}$, $\theta=w$ and $\Theta=\cW$.
More precisely,
if $\psi_w$ denotes either $\xi_w$ or
$\eta_w$, and
if
$\bP$ is the probability
law of $X_1,\ldots, X_n$ (when $\xi_w$ is studied) or of $(X_1,\e
_1),\ldots,(X_n,\e_n)$
(when $\eta_w$ is studied) then
we want to find a functional $U(\psi_w)=U_\psi(w,f)$
such that (\ref{eq:individual-prob}) holds and
%
%
\begin{eqnarray}\label{eq:11}\qquad
\bP\Bigl\{\sup_{w\in\cW}\bigl[\|\psi_w\|_{s,\tau} -
u_\ve\bigl(1+\sqrt{y}\lambda_A+y\lambda_B\bigr)U_\psi(w,f)\bigr]\geq
0\Bigr\} &\leq& P_{\ve,g}(y),
\\
\label{eq:22}
\bE\sup_{w\in\cW}
\bigl[\|\psi_w\|_{s,\tau}-u_\ve\bigl(1+\sqrt{y}\lambda_A+
y\lambda_B\bigr)U_\psi(w,f)\bigr]_+^q
&\leq& E_{\ve,g}(y),\nonumber\\[-8pt]\\[-8pt]
&&\eqntext{q\geq1.}
\end{eqnarray}

Note that $\{\xi_w, w\in\cW\}$ is the \textit{empirical process}.
In the sequel, we refer to
$\{\eta_w, w\in\cW\}$ as the
\textit{regression-type process} as it\vadjust{\goodbreak} naturally appears in
nonparametric estimation of regression functions.
In the regression context,~$X_i$ are the design variables,
$\e_i$ are the random noise variables.

Uniform probability and moment bounds
for empirical processes are a~subject of vast literature; see, for example,
\citet{alexander}, \citet{Talagrand}, \citet{wellner},
\citet{massart}, \citet{bousquet},
\citet{gine-kolt} among many others.
Such bounds play an important role in establishing
the laws of iterated logarithm and central limit
theorems [see, e.g., \citet{alexander} and
\citet{gine-zinn}].
However,
we are not aware of works studying
uniform bounds of the type~(\ref{eq:11}) and~(\ref{eq:22}) satisfying~(\ref{eq1:introduction-new})
and~(\ref{eq2:introduction-new})
for the
$\bL_s$-norms of such processes.

Apart from the pure probabilistic interest,
development of uniform bounds on the $\bL_s$-norms of processes
$\{\xi_w, w\in\cW\}$ and $\{\eta_w, w\in\cW\}$
is motivated by
problems of adaptive estimation arising in nonparametric statistics.
In particular, the processes $\{\xi_w, w\in\cW\}$ and $\{\eta_w, w\in
\cW\}$
represent stochastic errors of linear estimators with the weight $w$
in the density estimation and nonparametric regression models,
respectively.
Uniform bounds on the error process are key technical tools
in development of virtually all adaptive estimation procedures
[see, e.g., \citet{Barron-Birge}, \citet{devroye}
\citet{golubev}, \citet{GL1} and \citet{spok-golubev}].

The \textit{kernel density estimator process} is
a particular case of the empirical process
$\{\xi_w, w\in\cW\}$ that was frequently studied
in the probabilistic literature. It is
associated with the weight function $w$ given by
%
%
\begin{equation}\label{eq:w}
w(t, x)= \frac{1}{n \prod_{i=1}^d h_i} K\biggl(\frac{t-x}{h}\biggr
),\qquad x\in\cX=\bR
^{d}, t\in\cT=\bR^d,
\end{equation}
where $K\dvtx\bR^d\to\bR$ is a kernel, $h=(h_1, \ldots, h_d)$ is the
bandwidth vector,
and $u/v$ denotes the coordinate-wise division for $u, v\in\bR^d$.
Limit laws
for the
$\bL_s$-norms of the kernel density estimators were derived in
\citet{beir-mason}; \citet{dumbgen}
study exact asymptotics for the large/moderate deviation probabilities.
\citet{gine-mason-zaitsev} investigate weak convergence of
the $\bL_1$-norm kernel density estimator process
indexed by a class of kernels under entropy conditions.
For other closely related work, see
\citet{mason-ein},
\citet{gine-kolt-zinn}, \citet{gine-nikl}
and references therein.
We remark that the kernel density estimator process is
naturally parametrized by $\cW=\cK\times\cH$, where $\cH$ is a set
of bandwidths and $\cK$ is a family of kernels.
The \textit{convolution kernel density estimator process}
will be also studied in Section \ref{sec_examples}.

The inequalities (\ref{eq:11}) and (\ref{eq:22}) are useful for constructing
statistical procedures provided that the following requirements are met.
\begin{enumerate}[(iii)]
\item[(i)] \textit{Explicit expression for $U_\psi(w,f)$.}
Typically, the bound $U_\psi(w,f)$ is directly involved in the
construction of statistical procedures; thus, it should be explicitly given.
\item[(ii)] \textit{Minimal assumptions on $\cW$.} This condition is
dictated by a
variety of problems where the inequalities (\ref{eq:11}) and (\ref{eq:22})
can be applied. In particular, the sets $\cW$ may have a complicated
structure
(see, e.g., examples in Section \ref{sec_examples}).
\item[(iii)] \textit{Minimal assumptions on $f$}.
The probability measure $\bP$ (and the expectation~$\bE$) as well as
the right-hand sides of (\ref{eq:11}) and (\ref{eq:22})
are determined by the density $f$. Therefore, we want to establish
(\ref{eq:11}) and (\ref{eq:22}) under weak assumptions on $f$.
In particular, we would like to emphasize
that all our results are established
for the set of all probability densities
uniformly bounded by a given constant. No regularity conditions are supposed.
\item[(iv)] \textit{Minimal assumptions on the distribution of $\e$.}
If the process $\{\eta_w, w\in\cW\}$ is considered, then the
probability measure
$\bP$ (and the expectation~$\bE$) is also determined
by the distribution of $\e$. Therefore, we would like to have
(\ref{eq:11}) and (\ref{eq:22}) under mild assumptions on this distribution.
We will see that the function $g$ given in (\ref{eq:individual-prob})
depends on the distribution tail of $\e$.
\end{enumerate}

Let us briefly discuss some consequences of requirement
(i) for the process $\{\xi_w, w\in\cW\}$.
Using the Talagrand concentration inequality, we prove that
(\ref{eq:individual-prob})
holds with
$U_\xi(w,f)=\bE\|\xi_w\|_{s,\tau}$,
on the space of functions $\mH=\{w\dvtx\sup_{x\in\cX}\|w(\cdot,x)\|
_{s,\tau
}<\infty\}$.
However, this bound cannot be used in statistical problems
at least for two reasons.

First, it is implicit and a reasonably sharp explicit upper bound
$\bar{U}_\xi(w,f)$ on $U_\xi(w,f)$
should be used instead. Sometimes if the class $\cW$ is not so complex
(e.g., $\cW=\cK\times\cH$) one can find a constant
$c$ independent of $w$, $f$ and $n$
such that
\[
c\bar{U}_\xi(w,f) \leq U_{\xi}(w,f) \leq\bar{U}_\xi(w,f).
\]
In such cases, $\bar{U}_\xi(w,f)$ can be regarded as a sharp bound on
$U_\xi(w,f)$.
We note, however, that establishing the above inequalities
requires additional assumptions on $\cW$ and $f$ and
nontrivial technical work. It seems that for more complex classes $\cW$
the problem of finding an ``optimal'' upper estimate
for $U_\xi(w,f)$ cannot be solved in the framework of probability theory.
Contrary to that, theory of adaptive
nonparametric estimation is equipped with
the optimality criterion, and an upper bound $\bar{U}_\xi(w,f)$ can be regarded
as sharp if it
leads to the optimal statistical procedure.
Thus, sharpness of $\bar{U}_\xi(w,f)$
can be assessed through accuracy analysis
of the resulting statistical procedure.

Second, $U_\xi(w,f)$ [and presumably its sharp upper bound $\bar
{U}_\xi(w,f)$]
depends on $f$.
In the density estimation context where the process
$\{\xi_w, w\in\cW\}$ appears, $f$ is the parameter to be estimated.
Therefore, bounds depending on $f$ cannot be used in construction of
estimation procedures.
A natural idea is to replace $U_\xi(w,f)$ by its
empirical counterpart $\hat{U}_\xi(w)$
whose construction is based only on the observations $X_1,\ldots, X_n$.
We adopt this strategy and establish
the corresponding inequality
%
%
\begin{equation}\label{eq:3}
\bE
\sup_{w\in\cW}\bigl[\|\xi_w\|_{s,\tau}-v_\ve
\bigl(1+\sqrt{y}\lambda_A+y\lambda_B\bigr)\hat{U}_\xi(w)\bigr]^q_+
\leq\tilde{E}_{\ve,g}(y),\qquad
q\geq1,\hspace*{-37pt}
\end{equation}
where $\tilde{E}_{\ve,g}(\cdot)$ differs
from $E_{\ve,g}(\cdot)$ in (\ref{eq:22}) only
by some absolute multiplicative factor, and,
therefore, satisfies (\ref{eq1:introduction-new})
if (\ref{eq2:introduction-new}) holds for $\bar{U}_\xi(w,f)$.
Here $v_\ve$
is bounded by some absolute constant and completely determined by $\ve$
and~$\cW$.
We provide an explicit expression for $v_\ve$.

Thus, requirement (i)
leads to a new type of uniform bounds that are~ran\-dom.
A natural question about sharpness of
these bounds arises.
In order to give an answer to this question,
we prove that
under mild assumptions on the class of weights
$\cW$ one can choose
$\ve\,{=}\,\ve_n\,{\to}\,0$ and $y_n\,{\to}\,\infty$ as $n\,{\to}\,\infty$~so~that
\[
\lim_{n\to\infty}
v_{\ve_n}\bigl(1+\sqrt{y_n}\lambda_A+y_n\lambda_B\bigr)=1,\qquad
\lim_{n\to\infty}
\tilde{E}_{\ve_n,g}(y_n)=0,
\]
and there exists $\tilde{\ve}_n\to0$, $n\to\infty$ such that for any
subset $\cW_0\subseteq\cW$ and any $q\geq1$
\[
\bE\Bigl[\sup_{w\in\cW_0}\hat{U}_\xi(w)\Bigr]^{q} \leq\Bigl[(1+\tilde
{\ve}_n)
\sup_{w\in\cW_0}
\bar{U}_\xi(w,f)\Bigr]^{q}+
R_n(\cW_0),
\]
where the remainder term $R_n(\cW_0)$ is asymptotically negligible
in the sense that
for any $\ell>0$
one has $\limsup_{n\to\infty}\sup_{f\in\cF}\sup_{\cW_0\subseteq\cW}
[ n^{\ell} R_n(\cW_0)]=0$. Here $\cF$ denotes the set of all
probability densities uniformly bounded by a given constant
[see (\ref{eq:cF})]. These results show
that in asymptotic terms the random uniform bound is almost as good as
the nonrandom
one, and there is no loss of sharpness due to the use of the random uniform
bound.
%
%
\subsection{Summary of results and organization of the paper}
In this paper, we develop a general machinery
for finding uniform upper bounds on sub-additive positive
functionals of sums of independent random functions.
We start with the general setting as outlined in Section \ref{subsec:general}
above, and establish
inequalities of the type
(\ref{eq:1}) and (\ref{eq:2}) (see Proposition \ref{l_uniform-2}).
Proofs of these results are based
on the chaining and
slicing/peeling techniques.
The distinctive feature of
our approach is that
$\Theta$ is assumed to be an image of
a subset $\bZ$, of a metric space under some continuous mapping $\phi$,
that is, $\Theta=\phi(\bZ)$ as in (\ref{eq:param}).
Then chaining on $\Theta$
is performed according to the distance induced on $\Theta$
by the mapping $\phi$.

Section \ref{sec:process-xi} is devoted to a
systematical study of the $\bL_s$-norm of
the empirical process
$\{\xi_w, w\in\cW\}$.
First, we derive an inequality
on the tail probability of $\|\xi_w\|_{s,\tau}$ for an individual
function $w\in\cW$
(see Theorem \ref{fixed_w1} in Section~\ref{sec:fixed}).
Here we use the Bernstein inequality
for empirical processes proved by \citet{bousquet} and inequalities
for norms of integral operators.
Then in Section~\ref{sec:uniform} we proceed with establishing
uniform bounds.
In Theorem \ref{t_uniform}
of Section \ref{sec:uniform-non-random},
we derive
uniform nonrandom bounds
for\vadjust{\goodbreak} $\|\xi_w\|_{s,\tau}$, $w\in\cW$ that hold for all $s\geq1$.
In the case $s>2$, the nonrandom bound depends on the density $f$; therefore,
for $s>2$ we construct a random bound and present
the corresponding result in Theorem \ref{t_random}.
Theorems \ref{t_uniform} and \ref{t_random} hold for classes of weights
$\cW$ satisfying rather general conditions.
In Section \ref{sec:proc-xi_assW},
we specialize results of Theorems \ref{t_uniform} and \ref{t_random}
to the classes $\cW$ of weights depending on the difference
of their arguments. This allows us to derive explicit both nonrandom
and random uniform bounds
on $\|\xi_w\|_{s,\tau}$ under conditions on the weights which can be easily
interpreted. The corresponding results are given in
Theorems \ref{t:consec1_new} and \ref{t_random_case}.
We also present some asymptotic corollaries which demonstrate
sharpness of the derived uniform bounds.
Section \ref{sec_examples} applies the results
of Theorems \ref{t:consec1_new} and \ref{t_random_case}
to special examples of the set $\cW$. In particular,
we consider the kernel density estimator process given by (\ref{eq:w}),
and the convolution kernel density estimator processes.
It turns out that corresponding results can be formulated
in a unified way, and they are presented in Theorem \ref{t_example_1}.

In Section \ref{sec:proc-eta}, we study $\bL_s$-norm
of the regression-type processes
$\{\eta_w, w\in\cW\}$ given in (\ref{eq:processes}).
First, we
derive an inequality
on the tail probability of $\|\eta_w\|_{s,\tau}$ for an individual
function $w\in\cW$
(Theorem \ref{th:eta_fixed_w} in Section \ref{sec:proc-eta-fixed}).
This theorem is proved under two different types of conditions
on the tail
probability of the random variable $\e$.
In Section \ref{sec:proc-eta_uniformW}, we present a
nonrandom uniform bound for
$\|\eta_w\|_{s,\tau}$ for all $s\geq1$ over
the class of weights depending on the difference of their
arguments. The corresponding result is given in Theorem \ref{t:proc-eta-assW},
and some asymptotic results that follow from Theorem \ref{t:proc-eta-assW}
are formulated in Corollary \ref{cor:concec-eta}.
Sections \ref{sec:proofs-prop}--\ref{sec:proofs-eta} contain
proofs of main results of this paper. Proofs of auxiliary
lemmas are given in the \hyperref[app]{Appendix}.
%
\section{Uniform bounds in general setting}\label{sec:key}
In this section, we
establish uniform probability bounds for the supremum of a general
sub-additive functional of a random process from the probability
inequality for
the individual process.

Let $\mS$ and $\mH$ be linear topological spaces, $(\Omega,\mA,\mathrm
{P})$ be a complete probabi\-lity space,
and let $\xi_\theta\dvtx\Omega\,{\to}\,\mS, \theta\,{\in}\,\mH$
be a family of random mappings
such~that:\looseness=-1
\begin{itemize}
\item
$\xi_{\bullet}(\omega)$ is linear and continuous on $\mH$
for any $\omega\in\Omega$;
\item
$\xi_{\theta}(\cdot)$ is $\mA$-measurable for any $\theta\in\mH$.
\end{itemize}
Let $\Psi\dvtx\mS\to\bR_+$ be a given {\sl sub-additive} functional, and
$\Theta$
be a fixed subset of $\mH$.
\begin{assumption}
\label{fixed_theta}
There exist functions $A\dvtx\mH\to\bR_+$, $B\dvtx\mH\to\bR_+$, $U\dvtx
\mH\to\bR_+$ and
$g\dvtx\bR_+\to\bR_+$
such that:
\begin{enumerate}[(iii)]
\item[(i)]
for any $z>0$
\[
{\mathrm{P}\{\Psi(\xi_\theta)-U(\theta)\geq z\} \leq g\biggl(\frac
{z^2}{A^{2}(\theta)+B(\theta)z}\biggr)\qquad
\forall\theta\in\mH;}\vadjust{\goodbreak}
\]
\item[(ii)] the function
$g$ is monotonically decreasing to $0$;

\item[(iii)]
$0< r:= \inf_{\theta\in\Theta}U(\theta) \leq
\sup_{\theta\in\Theta}U(\theta)=:R \leq\infty$.
\end{enumerate}
\end{assumption}

Condition (i) is a Bernstein-type probability inequality on $\Psi(\xi
_\theta)$
for a fixed \mbox{$\theta\in\mathfrak{H}$}. In particular, in examples of
Sections \ref{sec:process-xi} and \ref{sec:proc-eta} we have
$g(x)=ce^{-x^{\alpha}}$ and $g(x)=c x^{-p}$ for some $c,\alpha, p>0$.
Based on Assumption \ref{fixed_theta}, our goal is to derive
uniform probability and moment bounds of the type
(\ref{eq:1}) and (\ref{eq:2}).
For this purpose, we suppose that the set $\Theta$ is
parametrized in a special way; this assumption facilitates the use of the
standard chaining technique and leads to quite natural conditions on
the functions $U, A$ and $B$.
\begin{assumption}
\label{parameter}
Let $(\mZ,\mathrm{d})$ be a metric space, and let
$\bZ$ be a totally bounded subset of $(\mZ,\mathrm{d})$.
There exists a \textit{continuous} mapping $\phi$ from $\mZ$ to $\mH$
such that
\[
\Theta=\{\theta\in\mH\dvtx\theta=\phi[\zeta], \zeta\in\bZ\}.
\]
\end{assumption}

\begin{remark}
In statistical applications the set $\Theta$ is
parametrized in a~natural way. For instance,
if, as in the introduction section,
$\Psi(\cdot)=\|\cdot\|_{s, \tau}$ and
$\xi_\theta=\xi_w$ with
$w$
given by (\ref{eq:w}), then $\Theta$ is
parametrized by the kernel and bandwidth $(K, h)\in\cK\times\cH$.
The distance $\mathrm{d}$ on $\cK\times\cH$ may have a rather special form.
\end{remark}

Let $Z$ be a subset of $\bZ$. Define the following quantities:
%
%
\begin{eqnarray}
\label{eq2:assuption_parameter}
\varkappa_U(\cZ)&:=&\sup_{\zeta_1,\zeta_2\in\cZ}\frac{U(\phi[\zeta
_1]-\phi[\zeta_2])}{
\mathrm{d}(\zeta_1,\zeta_2)} \vee
\sup_{\zeta\in\cZ}U(\phi[\zeta]),\\
\label{eq:assumption_parameter-2}
\Lambda_A(\cZ)&:=&\sup_{\zeta_1,\zeta_2\in\cZ}\frac{A(\phi[\zeta
_1]-\phi[\zeta_2])}{
\mathrm{d}(\zeta_1,\zeta_2)} \vee
\sup_{\zeta\in\cZ}A(\phi[\zeta]),\\
\label{eq2:assuption_parameter-3}
\Lambda_B(\cZ)&:=&\sup_{\zeta_1,\zeta_2\in\cZ}\frac{B(\phi[\zeta
_1]-\phi[\zeta_2])}{
\mathrm{d}(\zeta_1,\zeta_2)} \vee
\sup_{\zeta\in\cZ}B(\phi[\zeta]).
\end{eqnarray}
Let $N_{\cZ,\mathrm{d}}(\delta)$ denote
the minimal number of balls
of radius $\delta$ in the metric
$\mathrm{d}$ needed to cover the set $\cZ$, and let
$\cE_{\cZ,\mathrm{d}}(\delta)=
\ln[N_{\cZ,\mathrm{d}}(\delta)]$ be the $\delta$-entropy
of~$Z$.
For any $y>0$ and $\ve>0$, put
\[
\cL^{(\ve)}_{g, Z}(y)=g(y
)+\sum_{k=1}^{\infty}[N_{\cZ,\mathrm{d}}
(\ve2^{-k})]^{2}
g(9y 2^{k-3} k^{-2}).
\]
\subsection*{Key propositions}
The next two statements are the main results of this section.
Define
%
%
\begin{equation}\label{eq:C-*}
C^{*}(y,Z):=\sqrt{y}\Lambda_A(Z)+y\Lambda_B(Z),\qquad y>0,
\end{equation}
where $\Lambda_A$ and $\Lambda_B$ are given in (\ref{eq:assumption_parameter-2})
and (\ref{eq2:assuption_parameter-3}).
\begin{proposition}
\label{l_uniform}
Suppose that Assumptions \ref{fixed_theta} and \ref{parameter} hold,
and let
$Z$ be a subset of $\bZ$ such that
$\sup_{\zeta,\zeta^\prime\in Z}\mathrm{d}(\zeta, \zeta^\prime)\leq\ve/4$
and $\varkappa_U(Z)<\infty$.
Then for all $y>0$ and $\ve>0$ one has
\[
\mathrm{P}\Bigl\{\sup_{\zeta\in\cZ}\Psi\bigl(\xi_{\phi[\zeta]}\bigr)\geq(1+\ve
)[\varkappa_U(Z)+C^*(y,\cZ)]\Bigr\}
\leq\cL^{(\ve)}_{g, Z}(y).
\]
\end{proposition}
\begin{remark}
\label{rem:after-prop1}
Inspection of the proof of Proposition \ref{l_uniform}
shows that
continuity of $\xi_{\bullet}$ on $\mH$ can be replaced
by the assumption that $\Psi(\xi_\bullet)$ is continuous
$\mathrm{P}$-almost surely on
$\phi[\bZ]$ in the distance $\mathrm{d}$.
The latter assumption is often easier to verify in specific problems.
\end{remark}

Define
%
%
\begin{equation}
\label{eq:psi-new}
\Psi_u^*(y, Z):=\sup_{\zeta\in Z}\bigl\{\Psi\bigl(\xi_{\phi[\zeta]}\bigr)-
u C^*(y)U(\phi[\zeta])\bigr\},\qquad y>0,
\end{equation}
where $Z\subseteq\bZ$ is a subset of $\bZ$,
$u\geq1$ is a constant,
and $C^{*}(\cdot)$ is the function defined below in
(\ref{eq:definitions-new-new}).
We derive bounds on the tail probability and $q$th moment
of the random variable $\Psi_u^*(y,\bZ)$.
Note that $\Psi_u^*(y,\bZ)$ is $\mA$-measurable for given $y$ and $u$ because
the mapping $\zeta\mapsto\xi_{\phi[\zeta]}$ is $\mathrm{P}$-almost surely
continuous, and $\bZ$ is a totally bounded set.
By the same reason
the supremum taken over any subset of $\bZ$ will be measurable as well.

With $r$ and $R$ defined in Assumption \ref{fixed_theta}(iii), for
any $a\in[r, R]$
consider the following subsets of $\bZ$:
%
%
\begin{equation}\label{eq:Z-a}
\bZ_a:=\{\zeta\in\bZ\dvtx a/2<U(\phi[\zeta])\leq a\}.
\end{equation}
In words, for given $a\in[r,R]$, $\bZ_a$ is the \textit{slice} of the parameter
values $\zeta\in\bZ$ for which the function $U(\phi[\zeta])$ takes
values between
$a/2$ and $a$.

In what follows, the quantities $\varkappa_U(\bZ_a)$, $\Lambda_A(\bZ_a)$,
$\Lambda_B(\bZ_a)$ and $L_{g, \bZ_a}^{(\epsilon)}(y)$ will be considered
as functions of $a\in[r, R]$. That is why,
with slight abuse of notation,
we will write
%
%
\begin{equation}\label{eq:definitions}
\varkappa_U(a):=\varkappa_U(\bZ_a),\qquad
L_{g}^{(\epsilon)}(y, a):=L_{g, \bZ_a}^{(\epsilon)}(y).
\end{equation}
Put also
%
%
\begin{equation}
\label{eq:definitions-new}
\Lambda_A:=\sup_{a\in[r, R]}a^{-1}\Lambda_A(\bZ_a);\qquad
\Lambda_B:=\sup_{a\in[r, R]}a^{-1}\Lambda_B(\bZ_a),
\end{equation}
and let the function $C^{*}(\cdot)$ in (\ref{eq:psi-new}) be defined as
%
%
\begin{equation}
\label{eq:definitions-new-new}
C^{*}(y):= 1+2\sqrt{y}\Lambda_A+2y\Lambda_B,\qquad y>0.
\end{equation}

\begin{proposition}\label{l_uniform-2}
Suppose that Assumptions \ref{fixed_theta} and \ref{parameter} hold, and
let\break \mbox{$\varkappa_U(\bZ)<\infty$}.
If
%
%
\begin{equation}\label{eq:<a}
\varkappa_U(a)\leq a\qquad \forall a\in[r,R],
\end{equation}
and if $u_{\ve}=2^{\ve}(1+\ve)$ then
for any $\epsilon\in(0,1]$, $y>0$ and any $q\geq1$
one has
%
%
\begin{eqnarray}\qquad
\label{eq:P}
\mathrm{P}\{\Psi_{u_{\ve}}^*(y, \bZ)\geq0\}&\leq& N_{\bZ,
\mathrm{d}}(\epsilon/8) \sum_{j=0}^{[\ve^{-1}\log_2(R/r)-1]_+}\cL
^{(\ve)}_{g}\bigl(y,r2^{\ve(j+1)}\bigr),
\\
\label{eq:E}
\mathrm{E}[\Psi_{u_{\ve}}^*(y, \bZ)]^q_+
&\leq& N_{\bZ,\mathrm{d}}(\epsilon/8) \bigl[u_{\ve}
C^*(y)\bigr]^q \nonumber\\[-8pt]\\[-8pt]
&&{}\times\sum_{j=0}^{[\ve^{-1}\log_2(R/r)-1]_+}
\bigl[r2^{\ve(j+1)}\bigr]^q
J^{(\ve)}_{g}\bigl(y,r2^{\ve(j+1)}\bigr),\nonumber
\end{eqnarray}
where $J^{(\ve)}_{g}(z,a):=q\int_1^\infty(x-1)^{q-1}\cL^{(\ve
)}_{g}(zx, a) \,\rd x$.\vspace*{-3pt}
\end{proposition}
\begin{remark}
\begin{enumerate}
\item
Proposition \ref{l_uniform}
establishes an upper bound on the tail probability of the supremum of
$\Psi(\xi_{\phi[\zeta]})$ over an arbitrary subset of $\bZ$
contained in a ball of radius $\ve/8$ in the metric $\mathrm{d}$.
The proof of Proposition \ref{l_uniform-2}
uses this bound for balls $Z$ of the radius $\ve/8$ that form
a covering of $\bZ$. Each ball $Z$ is divided on
\textit{slices} on which the value of $U(\phi[\zeta])$ is roughly
the same. Then the supremum
over $\bZ$ is bounded by the sum of suprema over the
\textit{slices}. This simple
technique is often used in the literature on
empirical processes
where it is referred to as \textit{peeling} or \textit{slicing}
[see, e.g.,
\citet{van-de-Geer}, Section 5.3, and
\citet{gine-kolt}].
\item Note that Proposition \ref{l_uniform-2}
holds for any distance $\mathrm{d}$ on $\mathfrak{Z}$. Therefore, if~$\varkappa_U(a)$ is proportional to $a$, condition
(\ref{eq:<a}) can be enforced by rescaling the distance~$\mathrm{d}$.\vspace*{-3pt}
\end{enumerate}
\end{remark}

We now present a useful bound that can be easily
derived from (\ref{eq:P}) and~(\ref{eq:E}).
Let
%
%
\begin{equation}\label{eq:L-g}
L^{(\ve)}_{g} := \sum_{k=1}^{\infty}[N_{\bZ,\mathrm{d}}(\ve2^{-k})]^{2}
\sqrt{g(9\cdot2^{k-3} k^{-2})}.
\end{equation}
Note that for all $\cZ\subseteq\bZ$ and $y\geq1$
\[
L^{(\ve)}_{g,\cZ}(y)\leq g(y)+L^{(\ve)}_{g}\sqrt{g(y)},
\]
because $\inf_{k\geq1}2^{k}(k)^{-2}=8/9$ and $g$ is monotone decreasing.
Therefore,
we arrive to the following corollary to Proposition \ref{l_uniform-2}.\vspace*{-3pt}
\begin{corollary}
\label{cor:prop2_new}
If the assumptions of Proposition \ref{l_uniform-2} hold,
and $L^{(\ve)}_g<\infty$ then for all $y\geq1$ and $\ve\in(0,1]$
\begin{eqnarray*}
\mathrm{P}\{\Psi_{u_{\ve}}^*(y, \bZ)\geq0\}
&\leq&N_{\bZ,\mathrm{d}}(\ve/8)
[1\vee\ve^{-1}\log_2(R/r)]\bigl[g(y)+L^{(\ve)}_g\sqrt{g(y)} \bigr],
\\
\mathrm{E}[\Psi_{u_{\ve}}^*(y, \bZ)]^q_+
&\leq& N_{\bZ,\mathrm{d}}(\ve/8) [2^{2\ve}R(1+\ve)
C^*(y)]^q [2^{q\ve}-1]^{-1}J^{(\ve)}_{g}(y),
\end{eqnarray*}
where\vadjust{\goodbreak} $J^{(\ve)}_{g}(z)=q\int_{1}^\infty(x-1)^{q-1}[g(zx)+L^{(\ve
)}_g\sqrt{g(zx)} ]\,\rd x$.
\end{corollary}

\section{Uniform bounds for norms of empirical processes} 
\label{sec:process-xi}
Based on the results obtained in Propositions \ref{l_uniform} and \ref
{l_uniform-2},
in this section we develop
uniform bounds for the family
$\{\|\xi_w\|_{s,\tau}, w\in\cW\}$, where
$\xi_w$ is defined in (\ref{eq:processes}).
The first step here is to check Assumption \ref{fixed_theta}. For this purpose,
we establish an exponential inequality for $\|\xi_w\|_{s,\tau}$ when
the function $w\in\cW$ is fixed. Next, using
Corollary \ref{cor:prop2_new} we derive
a \textit{nonrandom} uniform bound and establish
corresponding inequalities of the type (\ref{eq:11}) and (\ref{eq:22})
satisfying requirements (i)--(iv) of the \hyperref
[sec:introduction]{Introduction}.
We develop also a \textit{random} uniform bound
based on $X_1,\ldots, X_n$
and derive an inequality of the type (\ref{eq:3}).

To proceed, we need the following assumption.

\renewcommand{\theassm}{(A\arabic{assm})}

\begin{assm}\label{assumptionA1}
Let $\bar{\cX}$ be a countable dense subset of
$\cX$.
For any $\e>0$ and
any $x\in\cX$, there exists $\bar{x}\in\bar{\cX}$ such that
\[
\|w(\cdot,x)-w(\cdot, \bar{x})\|_{s,\tau}\leq\e.
\]
\end{assm}

In the sequel, we
consider only the sets $\cW$ of $(\mT\times\mX)$-measurable functions
satisfying Assumption \ref{assumptionA1}.
Let
\[
\nu^\prime(\rd x)=f(x)\nu(\rd x),
\]
and
for any $s\in[1,\infty]$ define
%
%
\begin{eqnarray}\label{eq:M-s}
\Sigma_s(w,f) :\!&=&
\biggl[\int\|w(t, \cdot)\|^s_{2,\nu^\prime} \tau(\rd t)\biggr]^{1/s}\nonumber\\
&=&
\biggl[\int\biggl(\int|w(t, x)|^2 f(x)\nu(\rd x)\biggr)^{s/2}
\tau(\rd t)\biggr]^{1/s},
\nonumber\\[-8pt]\\[-8pt]
M_{s,\tau,\nu^\prime}(w) :\!&=&
\sup_{x\in\cX} \|w(\cdot, x)\|_{s,\tau} \vee
\sup_{t\in\cT} \|w(t,\cdot)\|_{s,\nu^\prime},\nonumber\\
M_s(w):\!&=& M_{s,\tau,\nu}(w).\nonumber
\end{eqnarray}
Let $c_1(s):=15s/\ln s$, $s>2$, $c_2(s)$ be the constant
appearing below in inequality~(\ref{eq:folland-2}) of Lemma \ref
{folland}, and define
%
%
\begin{eqnarray}\label{eq:c-*}
c_3(s)&:=&c_1(s)\vee c_2\bigl(s/(s-1)\bigr)\qquad \forall s>2,\nonumber\\[-8pt]\\[-8pt]
c_*(s)&:=&\cases{
0, &\quad$1\leq s<2$,\cr
1, &\quad$s=2$,\cr
c_3(s), &\quad$s>2$.}\nonumber
\end{eqnarray}
It is worth mentioning that
$c_1(s)$ is the best known constant in the Rosenthal inequality [see
\citet{Johnson}],
and in many particular examples $c_2(s)=1$ (see Lemma \ref{folland}
below).
Although $c_1(s)$ is defined for $s>2$ only, it will be convenient
to set $c_1(s)=1$ if $s\in[1,2]$. We use this convention
in what follows without further\vadjust{\goodbreak} mention.
%
\subsection{Probability bounds for fixed weight function}
\label{sec:fixed}
For any $w\in\cW$, we define
%
%
\begin{eqnarray}\label{eq:rho}
\rho_s(w,f)&:=&\cases{
\bigl[\sqrt{n}
\Sigma_s(w,f)\bigr]\wedge[4n^{1/s}M_s(w)], &\quad
$s<2$,\vspace*{2pt}\cr
\sqrt{n} M_2(w), &\quad
$s=2$,\vspace*{2pt}\cr
c_1(s)\bigl[ \sqrt{n} \Sigma_s(w,f)
+ 2n^{1/s}M_s(w)\bigr], &\quad
$s>2$,}
\nonumber\\[-8pt]\\[-8pt]
\omega^2_s(w,f)&:=&\cases{
M^2_{s}(w)[14n +96n^{1/s}], &\quad
$s<2$,
\vspace*{2pt}\cr
6 nM^2_{1,\tau,\nu^\prime}(w)+24\sqrt{n}M^2_2(w),
&\quad$s=2$,}
\nonumber
\end{eqnarray}
and if $s> 2$ then we set
%
%
\begin{eqnarray}\label{eq:omega-s>2}
\omega^2_s(w,f)&:=& 6c_3(s)\bigl[n M^2_{{2s}/({s+2}),\tau,\nu^\prime}(w)\nonumber\\[-8pt]\\[-8pt]
&&\hspace*{30.6pt}{}+
4 \sqrt{n}\Sigma_s(w,f)M_s(w)+ 8n^{1/s}M^2_s(w)\bigr].\nonumber
\end{eqnarray}

\begin{theorem}
\label{fixed_w1}
Let $s\in[1,\infty)$ be fixed, and suppose that Assumption \ref
{assumptionA1} holds.
If $M_s(w)<\infty$, then for any $z>0$
%
%
\begin{eqnarray}\label{eq:11111}
&&\bP\{\|\xi_w\|_{s,\tau}\geq\rho_s(w,f)+z \}\nonumber\\[-8pt]\\[-8pt]
&&\qquad\leq\exp\biggl\{-\frac{z^2}{
({1}/{3})\omega_s^2(w, f) + ({4}/{3}) c_*(s) M_s(w)
z}\biggr\},\nonumber
\end{eqnarray}
where $c_*(s)$ is given in (\ref{eq:c-*}).
\end{theorem}
\begin{remark}
Because $c_*(s)=0$ for $s\in[1,2)$,
the distribution of the random variable
$\|\xi_w\|_{s, \tau}$ has a sub-Gaussian tail.
In this case,
similar bounds can be obtained from the inequalities given in
\citet{pinelis}, Theorem 2.1, \citet{pinelis-94}, Theorems
3.3--3.5, and
\citet{Ledoux-Tal}, Section 6.3.
In particular,
Theorem 1.2 of \citet{pinelis}
gives the upper bound
$\exp\{-z^2/2nM_s^2(w)\}$ which is
better by a constant factor
than our
upper bound in (\ref{eq:11111}) whenever $s\in[1,2)$.
However,
if $s\geq2$ then the cited results are not accurate enough
in the sense
that the corresponding bounds do not satisfy relations
(\ref{eq1:introduction-new}) and (\ref{eq2:introduction-new}) of
the \hyperref[sec:introduction]{Introduction}.
It seems that only concentration principle
leads to tight upper bounds; that is why we use this unified method
in our derivation.
\end{remark}

It is obvious that
the upper bound
of Theorem \ref{fixed_w1} remains valid
if we replace $\rho_s(w,f)$, $\omega^2_s(w,f)$
and $M_s(w)$ by their upper bounds.
The next result can be derived from
Theorem \ref{fixed_w1} in the case $s\in[1,2)$.
\begin{corollary}
\label{cor1:fixed_w}
Let $s\in[1,2)$ be fixed, and suppose that
Assumption \ref{assumptionA1} holds. If $M_s(w)<\infty$
then for every $z>0$ and for all $n\geq1$
\[
\bP\{\|\xi_w\|_{s,\tau} \geq4n^{1/s}M_s(w)+z \}
\leq\exp\biggl\{-\frac{z^2}{37nM^2_{s}(w)}\biggr\}.
\]
\end{corollary}

The result of the corollary is valid without any conditions on the
density~$f$. Moreover, neither the bound for $\|\xi_w\|_{s,\tau}$, nor
the right-hand side of the inequality depend on $f$. It is important to
realize that the probability inequality of Corollary~\ref{cor1:fixed_w}
is sharp in some cases. In particular, it is not too difficult to
construct a density $f$ such that $\Sigma_s(w,f)=+\infty$ for any
function $w$ satisfying rather general assumptions. In this case, the
established inequality seems to be sharp. On the other hand, for any
density $f$ satisfying a moment condition $\sqrt{n}\Sigma_s(w,f)$ can
be bounded from above, up to a numerical constant, by $\sqrt{n}M_2(w)$
which is typically much smaller than $n^{1/s}M_s(w)$.

Several useful bounds can be derived from Theorem \ref{fixed_w1}.
In particular, it is shown at the end of the proof of
Theorem \ref{fixed_w1}
that for all $s\geq2$ and $p\geq1$
%
%
\begin{equation}\label{eq:sigg}\qquad
\Sigma_s(w,f) \leq M_2(w)\bigl\|\sqrt{f}\bigr\|_{s,\nu},\qquad
M_{p,\tau,\nu^\prime}(w)\leq
[1\vee\|f\|_\infty]^{1/p} M_{p}(w).\vadjust{\goodbreak}
\end{equation}
Using these inequalities,
we arrive to the following result.
%
%
\begin{corollary}
\label{cor1:fixed_w3}
Let $s>2$ be fixed, and suppose that Assumption \ref{assumptionA1}
holds. If
$M_s(w)<\infty$, then for every $z>0$ and for all $n\geq1$
\begin{eqnarray*}
&&\bP\{\|\xi_w\|_{s,\tau}\geq\tilde{\rho}_s(w,f)+z \}\\
&&\qquad\leq\exp\biggl\{-\frac{z^2}{({1}/{3})\tilde{\omega}^2_s(w,f)+
({4}/{3})c_3(s)M_s(w)z}\biggr\},
\end{eqnarray*}
where $\tilde{\rho}_s(w,f):=c_1(s)
[ \sqrt{n} M_2(w)\|\sqrt{f}\|_{s,\nu}
+ 2n^{1/s}M_s(w)]$ and
\begin{eqnarray*}
\tilde{\omega}^2_s(w,f)&:=&6c_3(s)\bigl\{n[1\vee\|f\|_\infty]^{({s+2})/{s}}
M^2_{{2s}/({s+2})}(w)\\
&&\hspace*{30.9pt}{}+ 4 \sqrt{n}M_2(w)M_s(w)\bigl\|\sqrt{f}\bigr\|_{s,\nu} +
8n^{1/s}M^2_s(w)\bigr\}.
\end{eqnarray*}
\end{corollary}

\subsection{Uniform bounds}
\label{sec:uniform}
Theorem \ref{fixed_w1} together with Corollaries \ref{cor1:fixed_w}
and \ref{cor1:fixed_w3} ensures that Assumption \ref{fixed_theta} is
fulfilled for
$\|\xi_w\|_{s,\tau}$.
In this section, we use Proposition \ref{l_uniform-2} together with
Theorem \ref{fixed_w1}
in order to derive a uniform over $\cW$ bounds on
$\|\xi_w\|_{s,\tau}$.

Following the general setting of Section \ref{sec:key},
we assume that $\cW$ is a parame\-trized
set of weights, that is,
%
%
\begin{equation}\label{eq:W}
\cW=\{w: w=\phi[\zeta], \zeta\in\bZ\},
\end{equation}
where
$\bZ$ is a totally bounded subset of some metric space $(\mZ, \mathrm{d})$.
Thus, any $w\in\cW$ can be represented as $w=\phi[\zeta]$ for
some
$\zeta\in\bZ$. Recall that $N_{\bZ, \mathrm{d}}(\delta)$, $\delta>0$
stands for the minimal
number of balls of radius $\delta$ in the metric $\mathrm{d}$
needed to cover the set~$\bZ$, and $\cE_{\bZ, \mathrm{d}}(\delta)=\ln
[N_{\bZ, \mathrm{d}}(\delta)]$ is the $\delta$-entropy of $\bZ$.

The next assumption requires that the mapping
$\zeta\mapsto\phi[\zeta]=w$ be continuous in the\vadjust{\goodbreak} supremum norm.
\begin{assm}\label{assumptionA2}
For every $\e>0$, there exists $\gamma>0$ such that
for all $\zeta_1, \zeta_2\in\bZ$ satisfying
$\mathrm{d}(\zeta_1,\zeta_2)\leq\gamma$ one has
\[
{\sup_t\sup_x} |w_1(t,x)-w_2(t,x)|\leq\e,
\]
where $w_1(t,x)=\phi[\zeta_1](t,x)$ and $w_2(t,x)=\phi[\zeta_2](t,x)$.
\end{assm}

Because $\xi_w$ is linear in $w$, this assumption
along with Assumption \ref{assumptionA1}
guarantees that all the considered objects are measurable.

Let $\cF$ be the class of all probability densities uniformly bounded
by constant~$\mathrm{f}_\infty$,
%
%
\begin{equation}\label{eq:cF}
\cF:=\biggl\{p\dvtx\bR^{d}\to\bR\dvtx p\geq0, \int p=1, \|p\|_\infty\leq
\mathrm
{f}_\infty<\infty\biggr\}.\vadjust{\goodbreak}
\end{equation}
It is easily seen that
the inequalities of Theorem \ref{fixed_theta} and Corollary \ref{cor1:fixed_w3}
can be made uniform with respect to the class $\cF$.
Indeed, the bound of Corollary \ref{cor1:fixed_w3}
remains valid if one replaces
$\|f\|_\infty$ and $\|\sqrt{f}\|_{s,\nu}$ by $\mathrm{f}_\infty$ and
$\mathrm{f}_\infty^{1/2 -1/s}$, respectively.
From now on, we suppose
without loss of generality that
$\mathrm{f}_\infty\geq1$.
%
\subsubsection{Uniform nonrandom bound}
\label{sec:uniform-non-random}
Theorem \ref{fixed_w1} together with Corollaries~\ref{cor1:fixed_w}
and~\ref{cor1:fixed_w3} show
that Assumption \ref{fixed_theta} is fulfilled
for $\|\xi_w\|_{s,\tau}$ with $g(x)=e^{-x}$,
%
%
\begin{eqnarray}\label{eq:U-xi-f}
U(w)&=&U_{\xi}(w,f)\nonumber\\
:\!&=&
\cases{4n^{1/s}M_s(w), &\quad
$s\in[1, 2)$,\vspace*{2pt}\cr
\sqrt{n} M_2(w), &\quad$s=2$,\vspace*{2pt}\cr
c_1(s)\bigl[ \sqrt{n} \Sigma_s(w,f)
+ 2n^{1/s}M_s(w)\bigr], &\quad
$s>2$;}
\nonumber\\[-8pt]\\[-8pt]
A^2(w)&=&A_{\xi}^2(w)\nonumber\\
:\!&=&\cases{
37nM^2_{s}(w), &\quad
$s<2$,\vspace*{2pt}\cr
2\mathrm{f}^{2}_\infty nM^2_{1}(w)+8\sqrt{n} M^2_2(w),
&\quad
$s=2$,\vspace*{2pt}\cr
2c_3(s)\mathrm{f}^{2}_\infty\bigl[nM^2_{{2s}/({s+2})}(w)+
4 \sqrt{n}M_2(w)M_s(w)\vspace*{2pt}\cr
\hspace*{141pt}+\, 8n^{1/s}M^2_s(w)\bigr], &\quad
$s> 2$;}\hspace*{-34pt}
\nonumber
\end{eqnarray}
and
$B(w)=B_\xi(w):=\frac{4}{3}c_*(s)M_s(w)$,
where $c_*(s)$ is defined in (\ref{eq:c-*}).

Put
%
%
\begin{equation}\label{eq:r-R-xi}
r_\xi:=\inf_{w\in\cW} U_{\xi}(w,f),\qquad
R_\xi:=\sup_{w\in\cW} U_{\xi}(w,f).
\end{equation}
Let $\varkappa_{U_\xi}(\cdot)$ be given
by (\ref{eq:definitions}) with $U=U_\xi$, and
\[
C_\xi^*(y)=1+2\sqrt{y}\Lambda_{A_\xi} + 2y\Lambda_{B_\xi},\qquad y>0,
\]
where $\Lambda_A$ and $\Lambda_B$ are defined in (\ref{eq:definitions-new});
see also (\ref{eq:definitions-new-new}).
\begin{theorem} \label{t_uniform}
Let $s\geq1$ be fixed, (\ref{eq:W}) hold,
and let $f\in\cF$ if $s\geq2$.
Let Assumption \ref{assumptionA2} be fulfilled.
If $\varkappa_{U_\xi}(a)\leq a$ for all $a\in[r_\xi, R_\xi]$
then\vadjust{\goodbreak} for any $y\geq1$ and $\ve\in(0,1]$ one has\vspace*{-3pt}
\begin{eqnarray*}
&&\bP\Bigl\{
\sup_{w \in\cW}
[\|\xi_{w}\|_{s,\tau}- u_\ve C_\xi^*(y)U_\xi(w, f)]
\geq0\Bigr\} \\[-2pt]
&&\qquad\leq\frac{1}{\ve}N_{\bZ,\mathrm{d}}(\ve/8)
[1\vee\log_2(R_\xi/r_\xi)]
\bigl[1+L^{(\ve)}_{\exp}\bigr]e^{-y/2} ,
\\[-2pt]
&&\bE\sup_{w\in\cW}
[\|\xi_w\|_{s,\tau}- u_\ve C_\xi^*(y)U_\xi(w, f)]_+^q
\\[-2pt]
&&\qquad\leq
\frac{2^{q(\ve+1)}u_{\ve}^q}{2^{q\ve}-1} \Gamma(q+1)N_{\bZ,\mathrm
{d}}(\ve/8)
[R_\xi
C_\xi^*(1)]^q
\bigl[1+L^{(\ve)}_{\exp}\bigr]e^{-y/2} ,\vspace*{-3pt}
\end{eqnarray*}
where $u_\ve=2^\ve(1+\ve)$, $\Gamma(\cdot)$ is the gamma-function and\vspace*{-3pt}
%
%
\begin{equation}\label{eq:L-exp}
L^{(\ve)}_{\exp}=\sum_{k=1}^{\infty}\exp\{2\cE_{\bZ, \mathrm{d}}(\ve
2^{-k})-(9/16) 2^{k} k^{-2}\}.\vspace*{-3pt}
\end{equation}
\end{theorem}

The proof follows immediately by application
of Corollary \ref{cor:prop2_new}, and \mbox{noting} that for $g(x)=e^{-x}$
the quantity $L_g^{(\ve)}$ is given by the above formula
[cf.~(\ref{eq:L-g})], while
$J_g^{(\ve)}(\cdot)$ for $g(x)=e^{-x}$
is bounded as follows\vspace*{-3pt}
\begin{eqnarray*}
J_g^{(\ve)}(z)&=&q\int_{1}^\infty(x-1)^{q-1}\bigl[e^{-zx}+L_g^{(\ve)}
\sqrt{e^{-zx}} \bigr]
\,\rd x \\[-2pt]
&\leq&\Gamma(q+1)\bigl[1+L^{(\ve)}_{\exp}\bigr](2/z)^{q}e^{-z/2}.\vspace*{-3pt}
\end{eqnarray*}

\begin{remark}
It is instructive to compare
the results of Theorem \ref{t_uniform}
with those of Theorem \ref{fixed_w1} (and
Corollaries \ref{cor1:fixed_w} and \ref{cor1:fixed_w3}).
The uniform bound on $\|\xi_w\|_{s,\tau}$ in Theorem~\ref{t_uniform}
is determined by the individual
bound $U_\xi(w,f)$ for a fixed weight
$w\in\cW$, and by the function $C_\xi^*(\cdot)$\vspace*{1pt} which, in its turn,
is computed on the basis of $A_\xi(w)$, $B_\xi(w)$ and $U_\xi(w,f)$.
The function $C_\xi^*(\cdot)$
depends on the
parametrization (\ref{eq:W}) and on
the distance $\mathrm{d}$ on $\mathfrak{Z}$ via the quantities $\Lambda_{A}$
and~$\Lambda_B$ [see (\ref{eq:definitions-new})].
The right-hand sides of the inequalities in
Theorem \ref{t_uniform}
depend on massiveness of the set of weights $\cW$ as measured
by the entro\-py~$\cE_{\bZ,\mathrm{d}}(\cdot)$. Note also that these
bounds decrease exponentially in $y$.\vspace*{-1pt}
\end{remark}

%
\subsubsection{Uniform random bound}
\label{s_random}
$\!\!\!$The uniform nonrandom bounds on $\|\xi_w\|_{s,\tau}$
given in Theorem \ref{t_uniform}
depend on the density $f$
via $U_\xi(w, f)$.
As discussed in the \hyperref[sec:introduction]{Introduction}, this
does not allow one to use this bound in statistical problems.
Our goal is to recover the statement of Theorem \ref{t_uniform}
(up to some numerical constants)
with the unknown quantity $U_\xi(w,f)$ replaced by
its estimator $\hat{U}_\xi(w)$. Note also that $U_\xi(w, f)$
of Theorem \ref{t_uniform} depends on $f$ only if $s>2$; here
the quantity depending on $f$ is $\Sigma_s(w,f)$.

Assume that the conditions of Theorem \ref{t_uniform}
are satisfied, and let $s>2$.
For any $t\in\cT$ define\vspace*{-3pt}
%
%
\begin{eqnarray}
\label{eq:rho-hat-1}
\hat{\Sigma}_s(w)&:=&\|S_w\|_{s,\tau},\qquad
S^2_w(t) := \frac{1}{n}\sum_{i=1}^{n}w^2(t,X_i),\vadjust{\goodbreak}
\\[-2pt]
\label{eq:rho-hat-2}
\hat{U}_\xi(w) &:=&
c_1(s)\bigl[\sqrt{n} \hat{\Sigma}_s(w)+2n^{1/s}M_s(w)\bigr].
\end{eqnarray}
It is easily seen that $\hat{U}_\xi(w)$ is a reasonable estimate of
$U_\xi(w,f)$
because under mild assumptions
for any fixed $t\in\cT$ by the law of large numbers
\[
S^2_w(t)-\|w(t,\cdot)\|^2_{2,\nu^\prime}\to0,\qquad n\to\infty\qquad \mbox{in
probability}.
\]
Moreover,
\begin{eqnarray*}
|\hat{\Sigma}_s(w) - \Sigma_s(w,f)|^2
&\leq&\bigl\|S_w-\|w(\cdot,\cdot)\|_{2,\nu^\prime}\bigr\|^2_{s,\tau}\\
&\leq&\bigl\|\sqrt{\bigl|S^2_w-\|w(\cdot,\cdot)\|^2_{2,\nu^\prime}\bigr|}
\bigr\|^2_{s,\tau}\\
&=&\bigl\|S^2_w-\|w(\cdot,\cdot)\|^2_{2,\nu^\prime}
\bigr\|_{{s}/{2},\tau}\\
&=&
\Biggl\|\frac{1}{n}\sum_{i=1}^{n}[w^2(\cdot,X_i)-\bE w^2(\cdot,X)]\Biggr\|_{
{s}/{2},\tau}.
\end{eqnarray*}
Thus, for any $s>2$ we have
%
%
\begin{equation}
\label{eq:r1}
|\hat{\Sigma}_s(w)- \Sigma_s(w,f)|\leq\sqrt{\frac{\|\xi_{w^2}\|_{{s}/{2},\tau}}{n} },
\end{equation}
that is, the difference between
$\hat{\Sigma}_s(w)$ and $\Sigma_s(w,f)$
is controlled in terms of $\|\xi_{w^2}\|_{s/2,\tau}$.
The idea now is to use Theorem \ref{t_uniform} in order to find
a nonrandom upper bound
on $\|\xi_{w^2}\|_{s/2,\tau}$.
One can expect that this bound will be much smaller than $\Sigma_s(w,f)$
provided that the function $w$ is small enough.
If this is true then $\hat{\Sigma}_s(w)$ approximates well
$\Sigma_s(w,f)$, and it can be used
instead of $\Sigma_s(w,f)$
in the definition of the uniform over $\cW$
upper bound on $\|\xi_w\|_{s,\tau}$.

In order to control uniformly $\|\xi_{w^2}\|_{s/2,\tau}$ by applying
Theorem \ref{fixed_w1} and Corollary~\ref{cor:prop2_new}, we need the
following definitions. Put
\begin{eqnarray*}
\tilde{U}(w^{2})&:=&\cases{
4n^{2/s}M_{s/2}(w^{2}), &\quad$s\in(2,4)$,\vspace*{2pt}\cr
c_1(s/2)\bigl[ \mathrm{f}^{1/2}_{\infty}\sqrt{n} M_2(w^{2})
+ 2n^{2/s}M_{s/2}(w^{2})\bigr], &\quad$s\geq4$;}
\\
\tilde{A}^2(w^{2})&:=&\cases{
37n M^2_{s/2}(w^{2}),&\quad
$s\in(2,4)$,\vspace*{2pt}\cr
2c_3(s/2)\mathrm{f}^{2}_\infty\bigl[nM^2_{{2s}/({s+4})}(w^{2})\vspace*{2pt}\cr
\hspace*{55.2pt}+\,
4 \sqrt{n}M_2(w^{2})M_{s/2}(w^{2})\vspace*{2pt}\cr
\hspace*{88.8pt}+\,
8n^{2/s}M^2_{s/2}(w^{2})\bigr],&\quad
$s \geq4$;}
\nonumber
\end{eqnarray*}
and $\tilde{B}(w^{2}):=\frac{4}{3}c_*(s/2)M_{s/2}(w^{2})$,
where $c_*(\cdot)$ is given in\vadjust{\goodbreak} (\ref{eq:c-*}).

For any subset $\cZ\subseteq\bZ$, let $\varkappa_{\tilde{U}}(\cZ)$,
$\Lambda_{\tilde{A}}(\cZ)$, and $\Lambda_{\tilde{B}}(\cZ)$ be given by
\mbox{(\ref{eq2:assuption_parameter})--(\ref{eq2:assuption_parameter-3})}
with $U=\tilde{U}$, $A=\tilde{A}$ and $B=\tilde{B}$.
With $r_\xi$ and $R_\xi$ defined in (\ref{eq:r-R-xi}), let
%
%
\begin{equation}\label{eq:Z-aa}
\bZ_a=\{\zeta\in\bZ\dvtx a/2 <U_\xi(w,f)\leq a\},\qquad
a\in[r_\xi,R_\xi],\vadjust{\goodbreak}
\end{equation}
and
we set
%
%
\begin{eqnarray}
\label{eq:definitions-tilde-new}
\varkappa_{\tilde{U}}(a):\!&=&\varkappa_{\tilde{U}}(\bZ_a),\qquad
\lambda_{\tilde{A}}=\sup_{a\in[r_\xi, R_\xi]}
a^{-2}\Lambda_{\tilde{A}}(\bZ_a),\nonumber\\[-8pt]\\[-8pt]
\lambda_{\tilde{B}}&=&\sup_{a\in[r_\xi, R_\xi]}a^{-2}\Lambda_{\tilde
{B}}(\bZ_a),\nonumber
\end{eqnarray}
[cf. (\ref{eq:definitions}) and (\ref{eq:definitions-new})].
It is important to emphasize here that in the definition of $\varkappa
_{\tilde{U}}$, $\lambda_{\tilde{A}}$ and $\lambda_{\tilde{B}}$
we use the same set $\bZ_a$ as in the definition of
$\varkappa_{U_\xi}(\cdot)$, $\Lambda_{A_\xi}(\cdot)$ and $\Lambda_{B_\xi
}(\cdot)$.

The next result establishes a random uniform bound on $\|\xi_w\|_{s,\tau}$.
\begin{theorem}
\label{t_random}
Let $s> 2$ be fixed, (\ref{eq:W}) hold, Assumption \ref{assumptionA2}
be fulfilled, and
%
%
\begin{equation}\label{eq:condition-kappa}
\varkappa_{U_\xi}(a)\leq a\qquad \forall a\in[r_\xi, R_\xi].
\end{equation}
Let $\ve\in(0,1]$ be fixed, and suppose that there exists a positive
number $\gamma<[4c_1(s)(1+\ve)]^{-1} $ such that
%
%
\begin{equation}\label{eq:condition-kappa-tild}
\varkappa_{\tilde{U}}(a)\leq(\gamma a)^2\qquad
\forall a\in[r_\xi, R_\xi].
\end{equation}
If $y_{\gamma}$ denotes the root of the equation
%
%
\begin{equation}\label{eq:y-gamma}
\sqrt{y}\lambda_{\tilde{A}}+y\lambda_{\tilde{B}}=\gamma^{2},
\end{equation}
and if $y_\gamma> 1$ then:
\begin{enumerate}[(ii)]
\item[(i)]
For every $y\in[1,y_\gamma]$ one has
\[
\bE\sup_{w\in\cW}
\{\|\xi_{w}\|_{s,\tau}-
\bar{u}_\ve(\gamma) C_\xi^*(y) \hat{U}_\xi(w)\}^q_+\leq T_{1,\ve}[C_\xi
^*(y)]^{q} \exp\{-y/2\} ,
\]
where
$\bar{u}_\ve(\gamma):=u_\ve[1-4c_1(s)(1+\ve)\gamma]^{-1}$ and $u_\ve
=2^\ve(1+\ve)$.
\item[(ii)] For any subset $\cW_0\subseteq\cW$, one has
%
%
\begin{eqnarray}\label{eq:rem_after_th_random}
\bE\Bigl[\sup_{w\in\cW_0}\hat{U}_\xi(w)\Bigr]^{q}&\leq&
[1+4c_1(s)(1+\ve)\gamma]^q \sup_{w\in\cW_0}[ U_\xi(w,f)]^{q}
\nonumber\\[-8pt]\\[-8pt]
&&{} + T_{2,\ve}\Bigl[\sqrt{n} \sup_{w\in\cW_0} M_s(w)\Bigr]^{q}
\exp\{-y_{\gamma}/2\} .\nonumber
\end{eqnarray}
%
\end{enumerate}
The explicit expressions for the
constants $T_{1,\ve}$ and $T_{2,\ve}$
are given in the beginning of proof of the theorem.
\end{theorem}
%

\begin{remark}
\label{rem:after_th_random}
\begin{enumerate}
\item
Theorem \ref{t_random} requires two sets of conditions:
conditions of Theorem \ref{t_uniform}, and
conditions on behavior of the functions $\varkappa_{\tilde{U}}(\cdot)$,
$\Lambda_{\tilde{A}}(\cdot)$ and
$\Lambda_{\tilde{B}}(\cdot)$ on
the slices $\bZ_a$ defined through $U_\xi(w,f)$.
\item
The parameter $\gamma$ controls closeness of $\hat{U}_\xi(\cdot)$
to $U_\xi(\cdot,f)$: the smaller $\gamma$,
the closer the random bound $\hat{U}_\xi(\cdot)$
to the nonrandom one $U_\xi(\cdot,f)$ [see~(\ref{eq:rem_after_th_random})].
In this case,\vadjust{\goodbreak} we do not loose much if $U_\xi(\cdot,f)$ is replaced
by its empirical counterpart $\hat{U}_\xi(w)$.
Clearly,
it is possible to choose
$\gamma$ small and simultaneously to keep $y_\gamma$ large only if
$\lambda_{\tilde{A}}$ and $\lambda_{\tilde{B}}$
are small enough. Fortunately, this is the case in 
many examples.
\item
Note also that when $\gamma$ approaches $[4c_1(s)(1+\ve)]^{-1}$
the parameter $\bar{u}_\ve(\gamma)$ increases to infinity\vadjust{\goodbreak}
[clearly, we want to keep $\bar{u}_\ve(\gamma)$ as close to one as possible].
Thus, the assumption $\gamma<[4c_1(s)(1+\ve)]^{-1}$ is important;
this poses a restriction
on the parameter set $\cW$. We conjecture that
the following condition
is necessary: for given $s\geq2$ there exists a
universal constant, say, $c(s)$,
such that $\gamma< c(s)$.
\end{enumerate}
\end{remark}

%
The next corollary to Theorem \ref{t_random}
will be useful in what follows.
\begin{corollary}
\label{cor:th_random}
The statements of Theorem \ref{t_random} remain valid if $U_\xi(w,f)$ and
$\hat{U}_\xi(w)$ are redefined as
$\max\{U_\xi(w,f), \sqrt{n}M_2(w)\}$ and
$\max\{\hat{U}_\xi(w),\break \sqrt{n}M_2(w)\}$,
respectively.
\end{corollary}
%

\subsection{Unifrom bounds for classes of weights depending
on the difference of arguments}
\label{sec:proc-xi_assW}
As we have seen, the results and assumptions
in Theorems \ref{t_uniform} and \ref{t_random}
are stated in terms of the
quantities (such as $\lambda_{\tilde{A}}$,
$\lambda_{\tilde{B}}$, $y_{\gamma}$) that are given implicitly.
In particular, additional computations are still necessary in order to
apply Theorems~\ref{t_uniform} and \ref{t_random}
in specific problems.
In this section, we specialize the results
of Theorems \ref{t_uniform} and \ref{t_random}
for the classes of weights $\cW$ depending on the difference
of arguments.
Under natural and easily interpretable assumptions
on the class of such weights, we derive explicit uniform bounds
on the norms of empirical processes.

Throughout this section,
$\cX=\cT=\bR^{d}$,
$\tau=\nu=\operatorname{mes}$ is the Lebesgue measure and
we write $\|\cdot\|_s$ instead of $\|\cdot\|_{s,\tau}$.
In this section, the class of weights~$\cW$ is a set of functions from
$\bR^d\times\bR^d$ to $\bR$ of the following
form
%
%
\begin{equation}\label{eq:W-V}
\cW=\{w(t-x), w\in\cV\},
\end{equation}
where $\cV$ is a given set
of $d$-variate functions.
For the sake of
notational convenience, we will identify the weight $w\in\cW$ with the
$d$-variate function $w\in\cV$ in the definition of the process $\xi
_w$ and
the quantities such as $U_\xi$, $A_\xi$, $B_\xi$ etc.
Thus, when we write
$w\in\cW$ we mean the weight $w(\cdot-\cdot)$ while
$w\in\cV$ denotes the corresponding $d$-variate function;
this should not lead to a confusion.

Let $(\mathfrak{Z}, \mathrm{d})$ be a fixed metric space; as before,
we suppose that
$\cV$ is
parame\-trized by the parameter $\zeta\in\mathfrak{Z}$, that is,
%
%
\begin{equation}\label{eq:W-new}
\cV=\{w\dvtx w=\phi[\zeta], \zeta\in\bZ\},
\end{equation}
where
$\bZ$ is a totally bounded subset of the metric space $(\mZ, \mathrm{d})$.
Recall that
$N_{\bZ,\mathrm{d}}(\delta)$, $\delta>0$ is the number of the balls
of radius $\delta$ in the metric $\mathrm{d}$
that form a minimal covering of the set $\bZ$.
%

We need the following assumptions on the class of weights $\cW$ (the functional
set\vadjust{\goodbreak} $\cV$).
\renewcommand{\theassW}{(W)}
\begin{assW}\label{assW}
\begin{enumerate}[(W1)]
\item[(W1)]\hypertarget{assW1} The Lebesgue measure of support of all
functions from $\cV$
is finite, that is,
%
%
\begin{equation}\label{eq:mu-*}
\mu_{*}:=\sup_{w\in\cV} \operatorname{mes}\{\operatorname{supp}(w)\}
<\infty.
\end{equation}
%
%
\item[(W2)]\hypertarget{assW2}
There exist real numbers $\alpha_1\in(0,1)$ and $\alpha_2\in(0,1)$
such that
\[
\operatorname{mes}\{x\in\bR^{d}\dvtx|w(x)|\geq
\alpha_1\|w\|_\infty\}\geq\alpha_2 \operatorname{mes}
\{\operatorname{supp}(w)\}\qquad \forall w\in\cV.
\]
\item[(W3)]\hypertarget{assW3}
There exists a real number $\mu\geq1$ such that
\[
n \operatorname{mes}\{\operatorname{supp}(w)\} \geq\mu\qquad \forall w\in\cV.
\]
\item[(W4)]\hypertarget{assW4}
There exists a real number $\beta\in(0,1)$ such that
\[
\sup_{\delta\in(0,1)}\{\ln[N_{\bZ,\mathrm{d}}(\delta)] - \delta^{-\beta
}\}=:C_{\bZ}(\beta)<\infty.
\]
\end{enumerate}
\end{assW}

\begin{remark}
We will show that Assumption \hyperlink{assW2}{(W2)} is fulfilled
if $\cV$ is a~set of smooth functions.
Assumption \hyperlink{assW3}{(W3)} together with \hyperlink
{assW2}{(W2)} allows one to establish relations between
$\bL_p$-norms of functions from $\cV$; this will be extensively
used in what follows. Assumption \hyperlink{assW4}{(W4)} is a usual
entropy condition. In particular, \hyperlink{assW4}{(W4)}~ensures
that the quantity $L^{(\ve)}_{\exp}$ in (\ref{eq:L-exp}) is finite.
\end{remark}

In addition to Assumption \ref{assW}, we will need the following
assumption on the properties of the mapping $\phi$ in (\ref{eq:W-new}).
For $p\geq1$, put
%
%
\begin{equation}\label{eq:w-p}
0<\underline{\mathrm{w}}_{p}:=n^{1/p}\inf_{w\in\cV}
\|w\|_{p}\leq n^{1/p}\sup_{w\in\cV}
\|w\|_{p}=:
\bar{\mathrm{w}}_{p}<\infty
\end{equation}
and define
%
%
\begin{equation}\label{eq:Z-ss}
\bZ_{p}(b):=
\{\zeta\in\bZ\dvtx n^{1/p} \|\phi[\zeta]\|_{p} \leq b\},\qquad
b\in[\underline{\mathrm{w}}_{p},\bar{\mathrm{w}}_{p}].
\end{equation}

\renewcommand{\theassL}{(L)}

\begin{assL}\label{assL}
The mapping $\phi$ in (\ref{eq:W-new}) satisfies the following conditions:
\begin{itemize}
\item if $s\in[1,2)$ then
\[
\sup_{\zeta_1,\zeta_2\in\bZ_{s}(b)}
\frac{n^{1/s} \|\phi[\zeta_1]-\phi[\zeta_2]\|_{s}}
{\mathrm{d}(\zeta_1,\zeta_2)} \leq b \qquad \forall
b\in[\underline{\mathrm{w}}_{s}, \bar{\mathrm{w}}_{s}],
\]
\item if $s\geq2$ then
%
%
\begin{equation}\label{eq:L-s>2}
\sup_{\zeta_1,\zeta_2\in\bZ_{2}(b)}
\frac{\sqrt{n} \|\phi[\zeta_1]-\phi[\zeta_2]\|_2}
{\mathrm{d}(\zeta_1,\zeta_2)} \leq b \qquad \forall
b\in[\underline{\mathrm{w}}_{2}, \bar{\mathrm{w}}_{2}].
\end{equation}
\end{itemize}
\end{assL}
%

We note that Assumption \ref{assL}
guarantees continuity of $\|\xi_w\|_s$ on $\phi[\bZ]$ for any $s\leq2$.
The same property for $s>2$ follows from Lemma \ref{lem:before_th-random-case}.
This, in view of Remark \ref{rem:after-prop1}, replaces Assumption \ref
{assumptionA2}.

The next
statement presents the uniform moment bound on $\|\xi_w\|_s$
when $s\in[1,2]$, and $\cW$ is given by (\ref{eq:W-V}).
\begin{theorem}
\label{t:consec1_new}
$\!\!\!\!$Let the class of weights $\cW$ be defined by (\ref{eq:W-V}), and let~(\ref{eq:W-new}) and
Assumptions \textup{\hyperlink{assW1}{(W1)}}, \textup{\hyperlink{assW4}{(W4)}} and \ref
{assL} hold.
\begin{enumerate}[(ii)]
\item[(i)] If $s\in[1,2)$ then for all $n\geq1$,
$z\geq[\sqrt{37}/2] n^{1/2-1/s}$, and $\ve\in(0,1]$ one has
\[
\bE\sup_{w\in\cW}
[\|\xi_w\|_{s}- 4 u_\ve(1+z) n^{1/s}\|w\|_{s}]_+^q
\leq T_{3,\ve} n^{q/s}\exp\biggl\{-\frac{2z^2}{37}n^{(2/s)-1}\biggr\}.
\]
\item[(ii)]
If $f\in\cF$ then
for all $n\geq1$,
$z\geq\sqrt{8[\mu_*\mathrm{f}^2_\infty+ 4n^{-1/2}]}$,
and
$\ve\in(0,1]$ one has
\begin{eqnarray*}
&&\bE\sup_{w\in\cW}
\bigl[\|\xi_w\|_{2}- u_\ve(1+z+ z^{2}/12)\sqrt{n}\|
w\|_{2}\bigr]_+^q\\
&&\qquad\leq T_{4,\ve} n^{q/2} \exp\biggl\{-\frac{z^{2}}
{16[\mathrm{f}_\infty^{2}\mu_* + 4n^{-1/2}]}\biggr\}.
\end{eqnarray*}
\end{enumerate}
The explicit expressions for the constants $T_{3,\ve}$ and $T_{4,\ve}$
are given in the beginning of the proof of the theorem.
\end{theorem}

The bound of Theorem \ref{t:consec1_new}
is nonrandom because $U_\xi(w,f)$ does not depend
on $f$ whenever $s\in[1,2]$. The proof of this statement is based
on application of Theorem~\ref{t_uniform}.

Now we proceed with the case $s>2$.
Here we need some further notation. Given $p\geq2$, let $m_p\in(0,1]$
be such that
%
%
\begin{equation}
\label{eq:before-t-random-case}
\sup_{b\in[\underline{\mathrm{w}}_{2},\bar{\mathrm{w}}_{2}]}
b^{-1} \sup_{\zeta_1,\zeta_2\in\bZ_2(b)}
\frac{n^{1/p}\|\phi[\zeta_1]-\phi[\zeta_2]\|_p}
{[\mathrm{d}(\zeta_1,\zeta_2)]^{m_p}}=:C_p<\infty.
\end{equation}
Existence of $m_p\in(0,1]$ such that
(\ref{eq:before-t-random-case}) holds
is ensured by Lemma \ref{lem:before_th-random-case}
given in Section \ref{sec:proof-th-5}. In particular, it is shown there that
if Assumptions \ref{assW} and~\ref{assL} hold then $m_p$ can be taken
equal to $2/p$.
We note also that $m_2=1$ and $C_2=1$ by Assumption \ref{assL}.

Following (\ref{eq:rho-hat-1}), (\ref{eq:rho-hat-2}) and
Corollary \ref{cor:th_random}, we set
%
%
\begin{eqnarray}\label{eq:U-bar}\quad
\hat{U}_\xi(w)&=&c_1(s)\Biggl\{\sqrt{n}\Biggl\|\Biggl[\frac{1}{n}
\sum_{i=1}^n w^2(\cdot-X_i)\Biggr]^{1/2}\Biggr\|_s + 2n^{1/s}\|w\|_s\Biggr\},
\nonumber\\
\breve{U}_\xi(w):\!&=&\max\bigl\{\hat{U}_\xi(w),
\sqrt{n}\|w\|_2\bigr\},\\
\bar{U}_\xi(w):\!&=&\max\bigl\{U_\xi(w,f), \sqrt{n}\|w\|_2\bigr\}.
\nonumber
\end{eqnarray}
Put also
%
%
\begin{eqnarray}\label{eq:C-*-def}
C_\xi^*(y)&=&1+2\vartheta_0\bigl\{\sqrt{y} \bigl[\mu_*^{1/s} +n^{-1/(2s)}\bigr]
+ yn^{-1/s}\bigr\},\nonumber\\[-8pt]\\[-8pt]
m:\!&=&\cases{
1\wedge m_s, &\quad$s\in(2,4)$,\cr
1\wedge m_s\wedge m_{s/2}, &\quad$s\geq4$,}\nonumber
\end{eqnarray}
where
$\vartheta_0:=5c_1(s)[C_s\vee1] \mathrm{f}_\infty\alpha_1^{-1}\alpha
_2^{-1/2}$,
$\alpha_1$ and $\alpha_2$ are given in Assumption~\hyperlink
{assW2}{(W2)}, and $m_p$ and $C_p$
are defined
in (\ref{eq:before-t-random-case}).\vspace*{-3pt}

\begin{theorem}\label{t_random_case}
Let Assumptions \ref{assW} and \ref{assL} hold, and assume that
\mbox{$f\in\cF$}. Suppose that \textup{\hyperlink{assW3}{(W3)}} is fulfilled with
$\mu> [64c^2_1(s)]^{({s\wedge4})/({s\wedge4-2})}$, and
\textup{\hyperlink{assW4}{(W4)}}
is fulfilled with $\beta<m$.
Let $\gamma=\mu^{1/(s\wedge4)-1/2}$, and
%
%
\begin{equation}\label{eq:y-**}
y_*:=\cases{
\vartheta_1 n^{4/s-1}, &\quad$s\in(2,4)$,\cr
\vartheta_2 \mu^{-1/2}
[\mu_*^{2/s}+n^{-1/s}]^{-2}, &\quad$s\geq4$,}
\end{equation}
with constants $\vartheta_1$ and $\vartheta_2$ specified
explicitly in the proof; then for any $s>2$ and $y\in[1,y^*]$ one has
\[
\bE\sup_{w\in\cW}
\{\|\xi_{w}\|_{s}-
\bar{u}_\ve(\gamma)
C_\xi^*(y)\breve{U}_\xi(w)\}^q_+
\leq T_{5,\ve} n^{q/2}[C_\xi^*(y)]^{q}\exp\{-y/2\},
\]
where $\bar{u}_\epsilon(\cdot)$ is defined in Theorem \ref{t_random}.
In addition, if $\cW_0\subseteq\cW$ is an arbitrary subset
of $\cW$ then
\begin{eqnarray*}
\bE\Bigl[\sup_{w\in\cW_0}\breve{U}_\xi(w)\Bigr]^{q} &\leq&
\Bigl\{\bigl[1+4c_1(s)(1+\ve)\mu^{{1}/({s\wedge4})-{1}/{2}}\bigr]\sup_{w\in\cW
_0}\bar{U}_\xi(w)\Bigr\}^{q}\\
&&{} +T_{6,\ve} n^{{q(s-2)}/({2s})}
\exp\{-y_*/2\}.
\end{eqnarray*}
The explicit expressions for the constants $T_{5,\ve}$
and $T_{6,\ve}$ are given in the proof.\vspace*{-3pt}
\end{theorem}

Theorem \ref{t_random_case} establishes random uniform bounds
on the norms of empirical processes in terms of the parameters
determining the class $\cW$. In particular, the parameters
$\mu$ and $\mu_*$ play an important role.
Theorem \ref{t_random_case} leads
to a~number of
powerful asymptotic results
that demonstrate sharpness
of the proposed random bound.\vspace*{-3pt}
\begin{corollary}
\label{cor:asym-after-thW-random}
Let assumptions of Theorem \ref{t_random_case} hold, and let
$s>2$ be fixed.
There exist
positive constants
$k_{i}=k_i(s)$, $i=1,2,3$ such that if
\begin{eqnarray*}
\mu&=&\mu_n\asymp[\ln n]^{k_{1}},\qquad
\mu_*=\mu_{*,n}\asymp[\ln n]^{-k_{2}},\\
\ve&=&\ve_n\asymp[\ln n]^{-k_{3}},\qquad n\to\infty,
\end{eqnarray*}
then for all $\ell>0$ and $q\geq1$
\begin{eqnarray*}
&\displaystyle \lim_{n\to\infty}\sup_{f\in\cF} n^{\ell} \bE\sup_{w\in\cW} [\|\xi
_w\|_{s} -
(1+ 3\ve_n)\breve{U}_\xi(w)]_+^q=0,&
\\
&\displaystyle
\bE\Bigl[\sup_{w\in\cW_0}\breve{U}_\xi(w)\Bigr]^{q}\leq\Bigl[(1+\ve_n)\sup_{w\in\cW
_0}\bar{U}_\xi(w,f)\Bigr]^{q}
+ R_n(\cW_0),&
\end{eqnarray*}
%
where\vadjust{\goodbreak}
$\limsup_{n\to\infty}\sup_{f\in\cF}\sup_{\cW_0 \subseteq\cW}
[n^\ell R_n(\cW_0)]=0$.
\end{corollary}

The explicit expressions for the constants
$k_{i}>0$, $i=1,2,3$ are easily computed from
Theorem \ref{t_random_case}.
\begin{remark}
Corollary \ref{cor:asym-after-thW-random} shows that if the class of weights
$\cW$ is such that $\mu=\mu_n$ and $\mu_*=\mu_{*,n}$, and if $\ve$ is
set to be $\ve=\ve_n$,
then $(1+3\ve_n) \breve{U}_n(w)$
is a~uniform random bound on $\|\xi_w\|_s$
which is asymptotically almost as good as the nonrandom bound $\bar
{U}_\xi(w,f)$
depending on $f$.
Thus, in asymptotic terms, there is no loss in sharpness
of the random uniform bound in comparison with the nonrandom bound
that depends on $f$.
\end{remark}

\subsection{Specific problems}
\label{sec_examples}
In this section, we
consider process $\xi_w$ corresponding to
special classes of weights $\cW$ that arise in kernel density estimation.
Using results of
Theorems \ref{t:consec1_new} and \ref{t_random_case},
we derive uniform bounds on the norms of these processes.
As in Section \ref{sec:proc-xi_assW},
here $\cX=\cT=\bR^d$, and $\nu$ and~$\tau$ are both the Lebesgue measure.

Let $\cK$ be a given set of real functions defined on
$\bR^d$ and suppose that $\cK$ is a~\textit{totally bounded} set with respect to the $\bL_\infty$-norm.
Let $\cH:= \bigotimes_{i=1}^d [h^{\min}_{i}, h^{\max}_{i}]$,
where the vectors
$h^{\min}=(h^{\min}_1,\ldots, h^{\min}_d)$,
$h^{\max}=(h^{\max}_1,\ldots, h^{\max}_d)$,
$0<h^{\min}_{i}\leq h^{\max}_{i}\leq1$, $\forall i=1, \ldots, d$
are fixed.

For any $h\in\cH$ define
$V_h:=\prod_{i=1}^d h_{i}$, and
endow the set $\cH$ with the following distance:
%
%
\begin{equation}\label{eq:distance-Delta}
\Delta_\cH(h,h^\prime)
=\max_{i=1,\ldots,d} \ln\biggl(\frac{h_i\vee h_i^\prime}{h_i\wedge
h_i^\prime}\biggr).
\end{equation}
In order to
verify
that $\Delta_\cH$ is indeed a distance on
$\cH$ it suffices to note that
the function $(x,y) \mapsto\ln(x\vee y)-\ln(x\wedge y)$, $x>0, y>0$
satisfies all axioms of distance on
$\bR_+\setminus\{0\}$.

We will be interested in the following classes of weights $\cW$ and the
corresponding processes $\xi_w$.
\subsubsection*{Kernel density estimator process}
With any $K\in\cK$ and $h\in\cH$, we associate the weight function
\[
w(t-x)=n^{-1}K_h(t-x):= (nV_h)^{-1}K[(t-x)/h].
\]
As before,
$u/v$, $u, v\in\bR^d$, stands for
the coordinate-wise division $(u_1 /v_1,\ldots,\allowbreak u_d/v_d)$.

The weight $w$ is naturally parametrized by $K$ and $h$ so that
we put
%
%
\begin{equation}\label{eq:phi-1}
\bZ^{(1)}:=\cK\times\cH,\qquad \zeta=(K,h),\qquad w=\phi_1[\zeta]:=n^{-1}K_h.
\end{equation}
We equip $\bZ^{(1)}$ with
the family of distances $\{\mathrm{d}^{(1)}_\vartheta(\cdot, \cdot),
\vartheta>0\}$
defined by
\begin{eqnarray*}
\mathrm{d}_\vartheta^{(1)}(\zeta,\zeta^\prime)&=&\vartheta\max\{\|
K-K^\prime\|_\infty, \Delta_\cH(h,h^\prime)\},\\
\zeta&=&(K,h),\qquad
\zeta^\prime=(K^\prime,h^\prime),\qquad \vartheta>0.
\end{eqnarray*}
Obviously, $\bZ^{(1)}$ is a totally bounded set with\vadjust{\goodbreak} respect to
$\mathrm{d}^{(1)}_\vartheta$ for any $\vartheta>0$.

The corresponding family of random fields is
%
%
\begin{equation}\label{eq:process-ex-1}
\xi^{(1)}_{w}(t):=\xi_{\phi_1[\zeta]}(t)=\frac{1}{n} \sum_{i=1}^n \{
K_h(t-X_i) -\bE K_h(t-X)\},\qquad \zeta\in\bZ^{(1)},\hspace*{-45pt}
\end{equation}
and we are interested in bounds on the $\bL_s$-norm of this process
uniform
over the class of weights
\[
\cW^{(1)}:=\bigl\{w(\cdot-\cdot)=n^{-1}K_h(\cdot-\cdot)\dvtx(K,h)\in\bZ^{(1)}\bigr\}.
\]

We note that $\xi^{(1)}_w$ is the stochastic error
of the kernel density estimator associated with the kernel $K\in\cK$
and bandwidth $h\in\cH$.
According to Theorems \ref{t:consec1_new} and \ref{t_random_case},
for the process $\{\xi_w, w\in\cW^{(1)}\}$, the uniform
bounds on $\|\xi_w\|_s$ should be based on the following functionals. Define
\[
U_\xi^{(1)}(w):=\cases{ 4(nV_h)^{1/s-1}\|K\|_s, &\quad
$s\in[1,2)$,\vspace*{2pt}\cr
(nV_h)^{-1/2} \|K\|_2,&\quad
$s=2$.}
\]
For $s>2$, we put
\begin{eqnarray*}
U^{(1)}_\xi(w,f) &:=& c_1(s)\biggl[n^{-1/2} \biggl(\int\biggl[\int K^{2}_h(t-x)f(x)\,\rd
x\biggr]^{s/2}\,\rd t\biggr)^{1/s}\\
&&\hspace*{125.19pt}{}+ 2(nV_h)^{1/s-1}\|K\|_s\biggr],
\\
\hat{U}^{(1)}_\xi(w)&:=& c_1(s)\Biggl[n^{-1/2} \Biggl(\int\Biggl[n^{-1}\sum
_{i=1}^{n}K^{2}_h(t-X_i)\Biggr]^{s/2}\,\rd
t\Biggr)^{1/s}\\
&&\hspace*{120.04pt}{}+ 2(nV_h)^{1/s-1}\|K\|_s\Biggr],\\
\end{eqnarray*}
and finally
%
%
\begin{eqnarray}\label{eq:U-1}
\bar{U}{}^{(1)}_\xi(w,f)&:=& \max\bigl[U^{(1)}_\xi(w,f), (nV_h)^{-1/2}
\|K\|_2\bigr]
\nonumber\\[-8pt]\\[-8pt]
\breve{U}^{(1)}_\xi(w)&:=&
\max\bigl[\hat{U}^{(1)}_\xi(w), (nV_h)^{-1/2}\|K\|_2\bigr].\nonumber
\end{eqnarray}
\subsubsection*{Convolution kernel density estimator process}
For any $(K,h)\in\bZ^{(1)}$ and $(Q,\mh)\in\bZ^{(1)}$, we define
%
%
\begin{equation}\label{eq:KK-hh}
w(t-x)=n^{-1}[K_h\ast Q_{\mh}](t-x),
\end{equation}
where $\bZ^{(1)}$ is defined in (\ref{eq:phi-1}), and
$\ast$ stands for the convolution on $\bR^d$.
Put
\[
\bZ^{(2)}:=\bZ^{(1)}\times\bZ^{(1)},\qquad
z=[(K,h),(Q,\mh)],\qquad
w=\phi_2[z]=n^{-1} (K_h*Q_\mh),
\]
and define the family of distances on $\bZ^{(2)}$ as
\begin{eqnarray}
\mathrm{d}^{(2)}_\vartheta(z,z^\prime)=
\vartheta\max\{
\| K-K^\prime\|_\infty\vee\| Q-Q^\prime\|_\infty,
\Delta_\cH(h,h^\prime)
\vee\Delta_\cH(\mh,\mh^\prime)\},\nonumber\\
\eqntext{\vartheta>0,}
\end{eqnarray}
where
$z=[(K,h),(Q,\mh)]$,
$z^\prime=[(K^\prime,h^\prime),(Q^\prime,\mh^\prime)]$,
$z, z^\prime\in\bZ^{(2)}$.
Obviously, $\bZ^{(2)}$~is a totally bounded set with respect to
the distance
$\mathrm{d}_\vartheta^{(2)}$ for any $\vartheta>0$.

The corresponding family of random fields is
%
%
\begin{eqnarray}\label{eq:process-ex-2}
\xi^{(2)}_{w}(t):\!&=&\xi_{\phi_2[z]}(t)\nonumber\\
&=&
\frac{1}{n}\sum_{i=1}^n \{[K_h\ast Q_{\mh}] (t-X_i) -
\bE[K_h\ast Q_{\mh}](t-X)\}, \\
&&\eqntext{\qquad\zeta\in\bZ^{(2)},}
\end{eqnarray}
and we are interested in a uniform bound on $\|\xi^{(2)}_w\|_s$ over
\[
\cW^{(2)}:=\bigl\{w(\cdot- \cdot)=n^{-1} K_h*Q_\mh(\cdot-\cdot),
[(K,h), (Q,\mh)]\in\bZ^{(2)}\bigr\}.
\]

The random field $\xi_{w}$ with $w$ given by (\ref{eq:KK-hh})
appears
in the context of multivariate density estimation.
In particular, the uniform bounds
on $\|\xi_w\|_s$ are instrumental in
construction of a selection rule
for
the family of kernel estimators parametrized by $\cK\times\cH$
[see \citet{GL2}].
%
Theorems \ref{t:consec1_new} and \ref{t_random_case}
suggest to base the uniform bounds on the following
quantities. Define
\begin{eqnarray*}
U^{(2)}_\xi(w)&:=&
\cases{4n^{1/s-1}\|K_h\ast Q_\mh\|_s,
&\quad
$s\in[1,2)$,\cr
n^{-1/2}\|K_h\ast Q_\mh\|_2,&\quad$s=2$.}
\end{eqnarray*}
For $s>2$, we put
\begin{eqnarray*}
U^{(2)}_\xi(w,f)&:=& c_1(s)\biggl[n^{-1/2} \biggl(\int\biggl[\int[K_h\ast Q_\mh
]^{2}(t-x)f(x)\,\rd x\biggr]^{s/2}\,\rd t\biggr)^{1/s}
\\
&&\hspace*{151.2pt}{}+ 2n^{1/s-1}\|K_h\ast Q_\mh\|_s\biggr],
\\
\hat{U}^{(2)}_\xi(w)&:=& c_1(s)\Biggl[n^{-1/2} \Biggl(\int\Biggl[n^{-1}\sum
_{i=1}^{n}[K_h\ast Q_\mh]^{2}(t-X_i)
\Biggr]^{s/2}\,\rd t\Biggr)^{1/s}\\
&&\hspace*{146.7pt}{} + 2n^{1/s-1}\|K_h\ast Q_\mh\|_s\Biggr];
\end{eqnarray*}
and finally
%
%
\begin{eqnarray}\label{eq:U-2}
\bar{U}{}^{(2)}_\xi(w,f) &:=& \max\bigl[U^{(2)}_\xi(w,f),
n^{-1/2}\|K_h\ast Q_\mh\|_2\bigr],
\nonumber\\[-8pt]\\[-8pt]
\breve{U}^{(2)}_\xi(w)&:=&\max\bigl[\hat{U}^{(2)}_\xi(w), n^{-{1/2}}
\|K_h\ast Q_\mh\|_2\bigr].
\nonumber
\end{eqnarray}
Theorems \ref{t:consec1_new} and \ref{t_random_case}
can be used in order to establish upper bounds on the norms of
the processes $\xi^{(i)}_w$, $i=1,2$. For this purpose, Assumptions \ref{assW}
and~\ref{assL}
should be verified for the classes of weights $\cW^{(i)}$, $i=1,2$,
defined above. To this end,
we introduce conditions on the family of kernels $\cK$
that imply Assumptions \ref{assW} and~\ref{assL}.
These conditions are rather natural and easily verifiable;
they can be weakened in several ways, but
we do not pursue
this issue here and try to minimize cumbersome calculations to be done.
\renewcommand{\theassK}{(K)}
\begin{assK}\label{assK}
\begin{enumerate}[(K2)]
\item[(K1)]\hypertarget{assK1}
The family
$\cK$ is a subset of the
isotropic H\"older ball of functions~$\bH_d(1,\allowbreak L_\cK)$
with the exponent $1$ and the
Lipschitz constant $L_\cK$, that is,
\[
|K(x)-K(y)|\leq L_\cK|x-y|\qquad \forall x,y\in\bR^d,
\]
where \mbox{$|\cdot|$} denotes the Euclidean distance.
Moreover, any function~$K$ from $\cK$ is compactly supported and,
without loss of generality,\break
$\operatorname{supp}(K)\subseteq[-1/2,1/2]^d$ for all $K\in\cK$.
\item[(K2)]\hypertarget{assK2}
There exist real numbers $\mathrm{k}_{1}>0$ and
$\mathrm{k}_{\infty}<\infty$ such that
\[
\mathrm{k}_{1}\leq\biggl|\int K(t) \,\rd t\biggr| \leq
\|K\|_\infty\leq
\mathrm{k}_{\infty}\qquad \forall K\in\cK.
\]
Without loss of generality, we will assume that $\mathrm{k}_{\infty
}\geq1$ and $\mathrm{k}_1\leq1$.
\item[(K3)]\hypertarget{assK3}
The set
$\cK$ is a totally bounded set with respect to the $\bL_\infty$-norm,
and there exists a real number $\beta_\cK\in(0,1)$
such that
the entropy $\cE_\cK(\cdot)$ of~$\cK$ satisfies
\[
\sup_{\delta\in(0,1)}[\cE_{\cK}(\delta)- \delta^{-\beta_{\cK}}]=:C_\cK
<\infty.
\]
\end{enumerate}
\end{assK}

Several remarks on the above assumptions are in order.
First, we note that Assumptions \hyperlink{assK1}{(K1)} and \hyperlink
{assK3}{(K3)}
are not completely independent. In fact, if we suppose that $\cK\subset
\bH_d(\alpha,L_\cK)$ with
some $\alpha>d$ then Assumption \hyperlink{assK3}{(K3)} is
automatically fulfilled with $\beta_\cK=\alpha/d$.
On the other hand, all our results remain valid
if $\cK\subset\bH_d(\alpha,L_\cK)$ with some $\alpha>0$.
Observe also that the condition $ |{\int K(t)\,\rd t} |\geq\mathrm{k}_{1}$
of Assumption \hyperlink{assK2}{(K2)}
is not restrictive at all because
for kernel estimators $\int K(t)\,\rd t=1$. Therefore,
the first inequality in \hyperlink{assK2}{(K2)} is satisfied with
$\mathrm{k}_{1}=1$.
\begin{remark}
\label{r_verification_assump}
It is easy to check that Assumption \hyperlink{assK1}{(K1)} implies
Assumption~\ref{assumptionA2} in Section \ref{sec:uniform}
and Assumption \ref{parameter} in Section \ref{sec:key}.
\end{remark}

Now we apply Theorems~\ref{t:consec1_new} and~\ref{t_random_case}
to the families of random fields given
by~(\ref{eq:process-ex-1}) and~(\ref{eq:process-ex-2}).
We present the
results for the processes
$\{\xi_{\phi_1[\zeta]},\zeta\in\bZ^{(1)}\}$ and $\{\xi_{\phi_2[z]},z\in
\bZ^{(2)}\}$
in a unified way.
%
\subsubsection{\texorpdfstring{Case $s\in[1,2]$. Uniform nonrandom bounds}{Case s in [1,2]. Uniform nonrandom bounds}}
In order to derive the uniform upper bounds for $s\in[1,2]$,
we use Theorem \ref{t:consec1_new}.
Obviously,
Assumption \ref{assK} implies Assumptions \hyperlink{assW1}{(W1)} and
\hyperlink{assW4}{(W4)}. Thus, in order to apply
Theorem \ref{t:consec1_new}, we need to verify Assumption \ref{assL}.
This is done in Lemma \ref{lem:as-L} given in Section \ref{proof:theorem7}.
Thus, Theorem \ref{t:consec1_new} is directly applicable, and
nonasymptotic bounds can be straightforwardly derived from this theorem;
one needs only to recalculate the constants appearing in the
statements of the theorem.

We note that the quantity $\mu_*$ defined in (\ref{eq:mu-*})
satisfies $\mu_*\leq V_{h^{\max}}$ for the set of weights $\cW^{(1)}$
and $\mu_*\leq2^dV_{h^{\max}}$ for the set of weights $\cW^{(2)}$.
If we assume that $V_{h^{\max}}\to0$ as $n\to\infty$,
then we can establish
some asymptotic results, one of which is given in the next
theorem.\vspace*{-3pt}
\begin{theorem}
\label{t:concec-part-s-less2}
If Assumption \ref{assK} holds, then
for all $s\in[1,2)$, $\ell>0$ and $\ve\in(0,1)$
\[
\lim_{n\to\infty}n^{\ell} \sup_{f\in\cF}\bE\sup_{w\in\cW^{(i)}}
\bigl[\bigl\|\xi^{(i)}_w\bigr\|_{s}- (1+\ve)U^{(i)}_\xi(w)\bigr]_+^q=0,\qquad i=1,2.
\]
If Assumption \ref{assK} holds and $V_{h^{\max}}=o(1/\ln n)$ as $n\to
\infty$,
then
for all $\ell>0$ and $\ve\in(0,1)$
\[
\lim_{n\to\infty} n^{\ell} \sup_{f\in\cF}\bE\sup_{w\in\cW^{(i)}}
\bigl[\bigl\|\xi^{(i)}_w\bigr\|_{2}-
(1+\ve)U^{(i)}_\xi(w)\bigr]_+^q=0,\qquad i=1,2.
\]
\end{theorem}

Proof of the theorem is omitted; it is a straightforward consequence
of Theorem~\ref{t:consec1_new} and Lemma \ref{lem:as-L} given below in
Section \ref{proof:theorem7}.\vspace*{-3pt}
%
\subsubsection{Case $s>2$. Uniform random bounds}
In the case $s>2$, the uniform bounds are derived
from Theorem \ref{t_random_case}.
To state these results, we need the following notation.
Define
%
%
\begin{eqnarray}\label{eq:vartheta0-1-2}
\vartheta_{0}^{(1)}&:=& 10c_1(s)\mathrm{f}_\infty
\bigl[L_\cK\sqrt{d}/\mathrm{k}_1\bigr]^{d/2},\nonumber\\[-8pt]\\[-8pt]
\vartheta_0^{(2)}&:=&10c_1(s)\mathrm{f}_\infty\bigl[2^{d+2}\sqrt{d}
L_\cK\mathrm{k}_\infty/\mathrm{k}_1^2\bigr]^{d/2}.\nonumber
\end{eqnarray}
The next two quantities, $A_\cH$ and $B_\cH$, are
completely determined by the bandwidth
set $\cH$:
%
%
\begin{eqnarray}\label{eq:AH-BH}
A_\cH&:=&\prod_{j=1}^d \ln(h_j^{\max}/h_j^{\min}),\nonumber\\[-10pt]\\[-10pt]
B_\cH&:=& \log_2 (V_{h^{\max}}/V_{h^{\min}})=\sum_{j=1}^d
\log_2 (h_j^{\max}/h_j^{\min}).\nonumber
\end{eqnarray}
For $y>0$ put
%
%
\begin{eqnarray}\label{eq:C-*-i}\quad
C_{\xi,i}^*(y)&:=&1+2\vartheta_0^{(i)}
\bigl\{\sqrt{y}\bigl(\bigl[ 2^{d(i-1)}V_{h^{\max}}\bigr]^{1/s} + n^{-1/2s}\bigr)
+yn^{-1/s}\bigr\},\nonumber\\[-10pt]\\[-10pt]
&&\eqntext{i=1,2.}\vadjust{\goodbreak}
\end{eqnarray}
Define also
\[
y_*^{(i)}:=\cases{
\vartheta_1^{(i)} n^{4/s-1}, &\quad$s\in(2,4)$,\vspace*{2pt}\cr
\vartheta_2^{(i)}(nV_{h^{\min}})^{-1/2}
\bigl[ \bigl(2^{d(i-1)}V_{h^{\max}}\bigr)^{2/s} + n^{-1/s}\bigr]^{-2}, &\quad$s\geq4$,}
\]
where explicit expressions for the
constants $\vartheta_1^{(i)}, \vartheta_2^{(i)}$, $i=1,2$ are given in the
proof of Theorem \ref{t_example_1}.
\begin{theorem}\label{t_example_1}
$\!\!\!$Let Assumption \ref{assK} hold, $f\!\in\!\cF$, and let
\mbox{${\max_{j=1,\ldots,d}}|h^{\max}_j|\!\leq\!1$}.
For $i=1,2$ assume that
%
%
\begin{equation}
\label{eq:cond-bandwidth}
nV_{h^{\min}} > [64 c^2_1(s)]^{({s\wedge4})/({s\wedge4-2})}
\bigl[2^{d+2}\sqrt{d}L_\cK\mathrm{k}_\infty/\mathrm{k}_1^2\bigr]^{d(i-1)}.
\end{equation}
If $\gamma:=(nV_{h^{\min}})^{1/(s\wedge4) -1/2}$, then for
any $s>2$, $y\in[1, y_*^{(i)}]$ and for $i=1,2$
one has
\begin{eqnarray*}
&&\bE\sup_{w\in\cW^{(i)}}
\bigl\{\bigl\|\xi_{w}^{(i)}\bigr\|_{s}-
\bar{u}_\ve(\gamma)
C_{\xi,i}^*(y)\breve{U}^{(i)}_\xi(w)\bigr\}^q_+\\
&&\qquad\leq\tilde{T}^{(i)}_{1,\ve}(1+A_\cH)^{2i}(1+B_\cH) n^{q/2}[C_{\xi,i}^*(y)]^{q}
e^{-y/2},
\end{eqnarray*}
where $\bar{u}_\epsilon(\cdot)$ is defined in Theorem \ref{t_random},
and $\breve{U}_\xi^{(i)}(w)$ are defined in (\ref{eq:U-1}) and~(\ref{eq:U-2}).

In addition, for any subset $\cW_0\subseteq\cW^{(i)}$, any $s>2$ and for
$i=1,2$ one has
\begin{eqnarray*}
\bE\Bigl[\sup_{w\in\cW_0}\breve{U}^{(i)}_\xi(w)\Bigr]^{q}&\leq&
\bigl[1+4c_1(s)(1+\ve)(nV_{h^{\min}})^{{1}/({s\wedge4})-
{1}/{2}}\bigr]^q\sup_{w\in\cW_0}\bigl\{\bar{U}{}^{(i)}_\xi(w)\bigr\}^{q}
\\
&&{} + \tilde{T}_{2,\ve}^{(i)} (1+A_\cH)^{2i}(1+B_\cH)
n^{{q(s-2)}/({2s})}
\exp\bigl\{-y_*^{(i)}/2\bigr\}.
\end{eqnarray*}
The explicit expressions for the constants $\tilde{T}_{1,\ve}^{(i)}$
and $\tilde{T}_{2,\ve}^{(i)}$ are given in the proof.
\end{theorem}

We emphasize that the upper bounds of Theorem \ref{t_example_1}
are nonasymptotic. The constants $\vartheta_1^{(i)}$, $\vartheta_2^{(i)}$,
$\tilde{T}_{1,\ve}^{(i)}$ and $\tilde{T}_{2,\ve}^{(i)}$
are written down explicitly in the proof of the theorem; they
are completely determined through
the quantities $L_\cK$,
$\mathrm{k}_1$, $\mathrm{k}_\infty$, $C_\cK$ and $\beta_\cK$
appearing in Assumption \ref{assK}, and the constant $c_1(s)$
in the Rosenthal inequality.
\begin{remark}\label{rem:1}
Condition (\ref{eq:cond-bandwidth}) is not restrictive
because the standard assumption on the bandwidth set $\cH$
in the kernel density estimation is
that
\[
nV_{h^{\min}} \to\infty,\qquad V_{h^{\max}}\to0,\qquad n\to\infty.
\]
\end{remark}

The bounds established in Theorem \ref{t_example_1} can be used
in order to derive
asymptotic (as $n\to\infty$) results under general assumptions on the
set of bandwidths $\cH$. One of such results
is given in the next corollary.
\begin{corollary}
\label{cor:asymp-proc-xi-example}
Let $s>2$ be fixed, Assumption \ref{assK} hold, and $f\in\cF$.
There exist positive constants $k_{1,i}=k_{1,i}(s)$,
$k_{2,i}=k_{2,i}(s)$ and $k_{3,i}=k_{3,i}(s)$, $i=1,2$, such that
if
\begin{eqnarray*}
V_{h^{\max}} &\asymp&[\ln n]^{-k_{1,i}},\qquad
nV_{h^{\min}} \asymp
[\ln n]^{k_{2,i}},\\
\ve&=&\ve_n\asymp[\ln n]^{-k_{3,i}},\qquad n\to\infty,
\end{eqnarray*}
then for all
$\ell>0$, $q\geq1$
\[
\lim_{n\to\infty}\sup_{f\in\cF} n^{\ell}
\bE\sup_{w\in\cW^{(i)}} \bigl[
\bigl\|\xi^{(i)}_w\bigr\|_{s} -
(1+3\ve_n)\breve{U}^{(i)}_\xi(w)\bigr]_+^q=0.
\]
In addition, for any subset $\cW_0\in\cW^{(i)}$
one has
\[
\bE\Bigl[
\sup_{w\in\cW_0}\breve{U}^{(i)}_\xi(w)\Bigr]^{q} \leq
\Bigl[(1+\ve_n)
\sup_{w\in\cW_0}\bar{U}{}^{(i)}_\xi(w,f)\Bigr]^{q} + R_n^{(i)}(\cW_0),
\]
where
$\limsup_{n\to\infty}\sup_{f\in\cF}\sup_{\cW_0\subseteq\cW^{(i)}}
[n^{\ell}R_n^{(i)}(\cW_0)]=0$, $i=1,2$.
\end{corollary}

We remark that
explicit expressions for the constants $k_{1,i}$ and $k_{2,i}$,
$i=1,2$,
are easily derived from Theorem \ref{t_example_1}.
%
\section{Uniform bounds for norms of regression-type processes}
\label{sec:proc-eta}
In this section,
we use Proposition \ref{l_uniform-2} in order to
derive uniform bounds for the family
$\|\eta_w\|_{s,\tau}$, \mbox{$w\in\cW$}; we recall that
\[
\eta_w(t)=\sum_{i=1}^n w(t, X_i)\varepsilon_i,
\]
see (\ref{eq:processes}).
First. we verify Assumption \ref{fixed_theta}
by establishing an analogue of Theorem~\ref{fixed_w1} for
a fixed weight function $w\in\cW$
[see Theorem \ref{th:eta_fixed_w} below].
It turns out that the corresponding inequality depends heavily
on the tail probability of the random variable $\e$.
In other words, we prove that Assumption \ref{fixed_theta} is fulfilled with
function $g$ that is determined by the rate at
which the tail probability of $\e$ decreases.
Next, under Assumptions \ref{assW} and \ref{assL},
we derive uniform bounds using Corollary \ref{cor:prop2_new}; this
leads to
an analogue of Theorem \ref{t:consec1_new}
for the regression-type processes.
%
\subsection{Probability bounds for fixed weight function}
\label{sec:proc-eta-fixed}
%
We consider two types of moment conditions on the distribution
of $\e$.
\renewcommand{\theassE}{(E)}
\begin{assE}\label{assE}
The distribution of $\e$ is symmetric, and
one of the following two conditions is fulfilled:
\begin{enumerate}[(E2)]
\item[(E1)]\hypertarget{assE1} there exist
constants $\alpha>0$, $v>0$ and $b>0$ such that
\[
\bP\{|\e|\geq x\}\leq v\exp\{-b x^\alpha\}\qquad \forall x>0,
\]
\item[(E2)]\hypertarget{assE2}
there exist constants $p\geq[s\vee2]$ and $P>0$ such that
\[
\bE|\e|^p\leq P.
\]
\end{enumerate}
\end{assE}

Let
$\sigma_{\e}^2:=\bE\e^2$ and
$e_s:=(\bE|\e|^s)^{1/s}$. For any
$w\in\cW$ define
\begin{eqnarray*}
\varrho_s(w,f)&:=&\cases{
\sigma_\e\bigl\{
\sqrt{n} \Sigma_s(w,f)
\wedge4n^{1/s}M_s(w)
\bigr\}, &\quad
$s<2$,
\vspace*{2pt}\cr
\sigma_\e\sqrt{n} M_2(w), &\quad
$s=2$,
\vspace*{2pt}\cr
c_1(s)\bigl[ \sigma_\e\sqrt{n} \Sigma_s(w,f)
+ 2n^{1/s} e_sM_s(w)\bigr], &\quad$s>2$,}
\\
\varpi^2_s(w,f)&:=&\cases{
M^2_{s}(w)[(6\sigma_\e^2+8)n +96\sigma_\e n^{1/s}],
&\quad$s<2$,\vspace*{2pt}\cr
6\sigma_\e^2nM^2_{1,\tau,\nu^\prime}(w)+ 24\sigma_\e\sqrt{n}M^2_2(w),
&\quad$s=2$,}
\end{eqnarray*}
and if $ s> 2$ then we set
\begin{eqnarray*}
\varpi^2_s(w,f)&:=& 6c_3(s) \bigl[\sigma_\e^2n
M^2_{{2s}/({s+2}),\tau,\nu^\prime}(w)\\
&&\hspace*{30.42pt}{}+
4\sigma_\e\sqrt{n} \Sigma_s(w,f) M_s(w)
+ 8e_sn^{1/s}M^2_s(w)\bigr].
\end{eqnarray*}
In the above formulas, we use notation introduced in the beginning of
Section~\ref{sec:process-xi}; the formulas should be compared with
(\ref{eq:rho}) and (\ref{eq:omega-s>2}).

The next theorem is the analogue of Theorem \ref{fixed_w1}
for the regression-type processes.
\begin{theorem}
\label{th:eta_fixed_w}
\begin{enumerate}[(ii)]
\item[(i)]
Suppose that Assumption \textup{\hyperlink{assE1}{(E1)}} holds, and for $x>0$ define
the function
%
%
\begin{eqnarray}\label{eq:G}
G_1(x) &:=&
(1+nv)g_{\alpha,b}(x), \nonumber\\[-8pt]\\[-8pt]
g_{\alpha,b}(x)&:=&
\cases{
\exp\bigl\{-|x| \wedge|b^{1/\alpha} x|^{\alpha/(2+\alpha)}\bigr\},&\quad$s<2$,
\cr
\exp\bigl\{-|x| \wedge|b^{1/\alpha} x|^{\alpha/(1+\alpha)}\bigr\},
&\quad$s\geq2$.}\nonumber
\end{eqnarray}
Then for all $s\in[1,\infty)$ and $z>0$ one has
\[
\bP\{\|\eta_w\|_{s,\tau}\geq\varrho_s(w,f)+z \}\leq
G_1\biggl(
\frac{z^2}{({1}/{3})\varpi_s^2(w, f) + ({4}/{3}) c_*(s) M_s(w) z}\biggr),
\]
where $c_*(\cdot)$ is given in (\ref{eq:c-*}).
\item[(ii)] Suppose that
Assumption \textup{\hyperlink{assE2}{(E2)}} holds and for $x>0$ define the function
\[
G_2(x):=(1+nP)\times\cases{
(x^{-1}p\ln[1+p^{-1}x])^{p/2}, &\quad$s<2$,\cr
(x^{-1}p\ln[1+p^{-1}x])^{p}, &\quad$s\geq2$.}
\]
Then for all $s\in[1,\infty)$ and $z>0$ one has
\[
\bP\{\|\eta_w\|_{s,\tau}\geq\varrho_s(w)+z \}\leq
G_2\biggl(
\frac{z^2}{({1}/{3})\varpi_s^2(w, f) + ({4}/{3}) c_*(s) M_s(w) z}\biggr).
\]
\end{enumerate}
\end{theorem}

\subsection{Uniform bound}
\label{sec:proc-eta_uniformW}
Theorem \ref{th:eta_fixed_w}
guarantees
that Assumption \ref{fixed_theta} holds with function $g$ being either $G_1$
or $G_2$.
This result is the basis for derivation of uniform bounds,
and the
general machinery presented in the previous sections can be fully
applied here.
In this section, we restrict ourselves only with
uniform bounds over the classes of weights depending on the difference
of arguments. In other words, under
Assumptions \ref{assW}, \ref{assL} and \hyperlink{assE1}{(E1)} we prove
an analogue of
Theorem \ref{t:consec1_new}
for the regression-type processes.

A natural assumption in the regression model where the process
$\{\eta_w, w\in\cW\}$ appears is that
the design variable $X$ is
distributed on a bounded interval of $\bR^{d}$, that is,
the density $f$ is compactly supported.
This will be assumed throughout this section.

Let $\cI\in\bR^{d}$ be a bounded interval,
$\cT=\cX=\cI$, and let $\tau=\nu=\operatorname{mes}$ be the Lebesgue measure.
For the sake of brevity, we write $\alpha_*=\alpha_1^{-1}\alpha
_2^{-1/2}$ where~$\alpha_1$ and $\alpha_2$ appear in Assumption \hyperlink{assW2}{(W2)}.
Define
\begin{eqnarray*}
\mathrm{a}&:=&
\max\bigl(\sigma_\e\sqrt{\operatorname{mes}(\cI)}, c_1(s)[\sigma_\e\mathrm
{f}_\infty^{1/2} + 2e_s\alpha_*]\bigr),\\
\mathrm{c}_n&:=&\tfrac{4}{3}c_*(s)\alpha_*n^{-1/s};
\\
\mathrm{b}^{2}_{n} &:=& \cases{
\bigl[2\sigma_\e^2+\frac{8}{3} +32\sigma_\e n^{1/s-1}\bigr]
\mu_*^{2/s-1}, &\quad
$s<2$,\vspace*{2pt}\cr
2\mathrm{f}^{2}_\infty\mu_*+8n^{-1/2},
&\quad
$s=2$,\vspace*{2pt}\cr
2c_3(s)\mathrm{f}^{2}_\infty[\sigma^{2}_\e\mu_*^{2/s}+
(4\sigma_\e\alpha_*+8e_s\alpha_*^{2}) n^{-1/s}], &\quad
$s> 2$.}
\end{eqnarray*}
%
%
%
\begin{theorem}
\label{t:proc-eta-assW}
Let Assumptions \ref{assW} and \textup{\hyperlink{assE1}{(E1)}} hold. Suppose
$f\in\cF$, and assume that (\ref{eq:L-s>2})
is valid for all $s\geq1$.
Let Assumption \textup{(W4)} be fulfilled with
$\beta<\alpha/(2+\alpha)$, if $s<2$, and with
$\beta<\alpha/(1+\alpha)$ if $s\geq2$.
Then for all $s\geq1$, $q\geq1$ and $y>1$ one has
\begin{eqnarray*}
&&\bE\sup_{w\in\cW}
\bigl[\|\eta_w\|_{s}-\mathrm{a}u_{\ve}
\bigl(1+2\sqrt{y}\mathrm{b}_{n}+2y\mathrm{c}_n\bigr)\sqrt{n}
\|w\|_{2}\bigr]_+^q\\
&&\qquad\leq T_{n,\ve}\bigl[1+2\sqrt{y}\mathrm{b}_{n}+
2y\mathrm{c}_n\bigr]^{q}[g_{\alpha,b}(y)]^{1/4},
\end{eqnarray*}
where $u_{\ve}=2^{\ve}(1+\ve)$,
$g_{\alpha,b}(\cdot)$ is defined in (\ref{eq:G}),
and the explicit expression of the constant
$T_{n,\ve}$ is given in the beginning of the proof of the theorem.
\end{theorem}

The following asymptotic result is an immediate
consequence of Theorem~\ref{t:proc-eta-assW}.
\begin{corollary}
\label{cor:concec-eta}
Let the assumptions of Theorem \ref{t:proc-eta-assW} hold.
For any $\alpha>0$ there exist a universal constant
$\mathrm{c}=\mathrm{c}(\alpha)>0$ such that
if $\mu_*\asymp[\ln n]^{-\mathrm{c}}$
then for all $s\geq1$, $\ve\in(0,1)$
and for all $\ell>0$
\[
\lim_{n\to\infty}n^{\ell}\sup_{f\in\cF} \bE\sup_{w\in\cW}
\bigl[\|\eta_w\|_{s}-
(1+\ve)\mathrm{a}\sqrt{n}\|w\|_{2}\bigr]_+^q=0.
\]
\end{corollary}

The explicit expression for $\mathrm{c}(\alpha)$ is easily
derived from Theorem \ref{t:proc-eta-assW}.

%
\section{\texorpdfstring{Proofs of Propositions \protect\ref{l_uniform} and \protect\ref{l_uniform-2}}
{Proofs of Propositions 1 and 2}}
\label{sec:proofs-prop}
\subsection{\texorpdfstring{Proof of Proposition \protect\ref{l_uniform}}{Proof of Proposition 1}}
Let $\cZ_k$, $k\in\bN$ be an
$\ve2^{-k-3}$-net of $\cZ$, and let $z_k(\zeta)$, $\zeta\in\cZ$
denote the element of $\cZ_k$ closest to
$\zeta$ in the metric $\mathrm{d}$.

The continuity of the mapping $\zeta\mapsto\xi_{\phi[\zeta]}$
guarantees that $\mathrm{P}$-almost
surely the following relation holds for any $\zeta\in\cZ$:
%
%
\begin{equation}
\label{chaining}
\xi_{\phi[\zeta]}=\xi_{\phi[\zeta^{(0)}]}+\sum_{k=0}^{\infty}\bigl[
\xi_{\phi[z_{k+1}(\zeta)]}-\xi_{\phi[z_{k}(\zeta)]}\bigr],
\end{equation}
where $\zeta^{(0)}$ is an arbitrary fixed element of $\cZ$ and
$z_0(\zeta)=\zeta^{(0)}$, $\forall\zeta\in\cZ$.

Note also that independently of $\zeta$ for all $k\geq0$
%
%
\begin{equation}
\label{p2}
\mathrm{d}(z_{k+1}(\zeta),z_{k}(\zeta))\leq\ve2^{-k-2} .
\end{equation}
We get from sub-additivity of $\Psi$, (\ref{chaining}) and (\ref{p2})
that for any $\zeta\in\cZ$
%
%
\begin{eqnarray}\label{p3}\qquad
\Psi\bigl(\xi_{\phi[\zeta]}\bigr)&\leq&
\Psi\bigl(\xi_{\phi[\zeta^{(0)}]}\bigr)+ \frac{\pi^2}{6}\sum_{k=0}^\infty
p_k\Psi\bigl(\xi_{\phi[z_{k+1}(\phi)]}-\xi_{\phi[z_k(\phi)]}\bigr)(k+1)^2
\nonumber\\[-8pt]\\[-8pt]
&\leq&
\Psi\bigl(\xi_{\phi[\zeta^{(0)}]}
\bigr)+
\frac{\pi^2}{6}\sup_{k\geq
0} \mathop{\sup_{(z,z^\prime)\in\cZ_{k+1}\times\cZ_k:}}_{
\mathrm{d}(z,z^\prime)\leq\ve2^{-k-2}}
(k+1)^2\Psi\bigl(\xi_{\phi[z]}-\xi_{\phi[z^\prime]}\bigr),
\nonumber
\end{eqnarray}
where $p_k:=6/(\pi^2 (k+1)^{2})$ and $\sum_{k=0}^\infty p_k=1$.
Since $\xi_{\bullet}$ is linear, $\xi_{\phi[z]}-\xi_{\phi[z^\prime]}=\xi
_{\phi[z]-\phi[z^\prime]}$ for all $z,z^\prime\in\bZ$,
and we obtain from (\ref{p3}) and the triangle inequality for probabilities
that
%
%
\begin{eqnarray}\qquad
\label{p4}
&&\mathrm{P}\Bigl\{\sup_{\zeta\in\cZ}\Psi\bigl(\xi_{\phi[\zeta]}\bigr)\geq(1+\ve
)[\varkappa_U(Z)+C^{*}(y,\cZ)]\Bigr\} \nonumber\\
&&\qquad\leq
\mathrm{P}
\bigl\{\Psi\bigl(\xi_{\phi[\zeta^{(0)}]}\bigr)\geq
\varkappa_U(Z)+C^{*}(y,\cZ)\bigr\} \nonumber\\[-8pt]\\[-8pt]
&&\qquad\quad{}+ \sum_{k=0}^{\infty} \mathop{\sum_{(z,z^\prime)\in\cZ_{k+1}\times
\cZ
_k:}}_{\mathrm{d}(z,z^\prime)\leq\ve2^{-k-2}}
\mathrm{P}\biggl\{\Psi\bigl(\xi_{\phi[z]-\phi[z^\prime]}\bigr)\geq\frac{6\ve[\varkappa
_U(Z)+C^{*}(y,\cZ)]}{\pi^2(k+1)^2}\biggr\}\nonumber\\
&&\qquad=:I_1+I_2.\nonumber
\end{eqnarray}
In view of (\ref{eq2:assuption_parameter}) and because
$\zeta^{(0)}\in\cZ$, we have that
$U(\phi[\zeta^{(0)}])\leq\varkappa_U(\cZ)$.
Therefore,
we get from Assumption \ref{fixed_theta}(i)
and monotonicity of the function $g$ that for any $y> 0$
%
%
\begin{eqnarray}\label{p5}
I_1 &\leq& \mathrm{P}
\bigl\{\Psi\bigl(\xi_{\phi[\zeta^{(0)}]}\bigr)-U\bigl(\phi\bigl[\zeta^{(0)}\bigr]\bigr)
\geq C^{*}(y,\cZ)\bigr\}
\nonumber\\
&\leq& g\biggl(
\frac{[C^{*}(y,\cZ)]^{2}}{A^2(\phi[\zeta^{(0)}])+
B(\phi[\zeta^{(0)}])C^{*}(y,\cZ)}
\biggr)\\
&\leq& g\biggl(
\frac{[C^{*}(y,\cZ)]^{2}}{\Lambda^{2}_A(\cZ)+
\Lambda_B(\cZ)C^{*}(y,\cZ)}
\biggr)
\leq g(y).
\nonumber
\end{eqnarray}
To order to get the last inequality, we have used monotonicity of $g$
and that for any $y>0$
\[
\frac{[C^{*}(y,\cZ)]^{2}}{\Lambda^{2}_A(\cZ)+
\Lambda_B(\cZ)C^{*}(y,\cZ)}=
\frac{[\sqrt{y}\Lambda_A(Z)+y\Lambda_B(Z)]^2}
{\Lambda_A^2(Z)+\Lambda_B(Z)[\sqrt{y}\Lambda_A(Z)+y\Lambda_B(Z)]} \geq
y.\vspace*{3pt}
\]

By (\ref{eq2:assuption_parameter}),
if $z,z^\prime\in\cZ$ and $\mathrm{d}(z,z^\prime)\leq\ve2^{-k-2}$ then\vspace*{3pt}
\[
U(\phi[z]-\phi[z^\prime])\leq\ve2^{-k-2}\varkappa_U(\cZ),\vspace*{3pt}
\]
and, therefore, for any $y\geq0$
\begin{eqnarray*}
&& \mathrm{P}\biggl\{\Psi\bigl(\xi_{\phi[z]-\phi[z^\prime]}\bigr)\geq\frac{6\ve
[\varkappa_U(Z)+C^{*}(y,\cZ)]}
{\pi^2(k+1)^2}\biggr\}
\\[1pt]
&&\qquad \leq\mathrm{P}\biggl\{\Psi\bigl(\xi_{\phi[z]-\phi[z^\prime]}\bigr)-
U(\phi[z]-\phi[z^\prime])\\[1pt]
&&\qquad\hspace*{22.6pt}
\geq\frac{6\ve[\varkappa_U(Z)+C^{*}(y,\cZ)]}{\pi^2(k+1)^2}-\varkappa
_U(\cZ) \ve2^{-k-2}\biggr\}\\[1pt]
&&\qquad \leq\mathrm{P}\biggl\{\Psi\bigl(\xi_{\phi[z]-\phi[z^\prime]}\bigr)-
U(\phi[z]-\phi[z^\prime])
\geq\frac{9\ve C^{*}(y,\cZ)}{16(k+1)^2}\biggr\}.
\end{eqnarray*}
Here we took into account that $\min_{k\geq0}
[6\pi^{-2}(k+1)^{-2} -2^{-k-2}]>0$ and $9/16<(6/\pi^{2})$.
Putting $C_k=\frac{9\ve C^{*}(y,\cZ)}{16(k+1)^2}$ and applying
Assumption \ref{fixed_theta}(i),
we obtain for any $z,z^\prime\in\cZ_{k+1}\times\cZ_k$ satisfying
$\mathrm{d}(z,z^\prime)\leq\ve2^{-k-2}$:
\begin{eqnarray*}
&& \mathrm{P}\biggl\{\Psi\bigl(\xi_{\phi[z]-\phi[z^\prime]}\bigr)\geq\frac{6\ve
[\varkappa_U(Z)+C^{*}(y,\cZ)]}
{\pi^2(k+1)^2}\biggr\}
\\[1pt]
&&\qquad \leq
g\biggl(\frac{C_k^{2}}{A^2(\phi[z]-\phi[z^\prime])+
B(\phi[z]-\phi[z^\prime])C_k}
\biggr)\nonumber\\[1pt]
&&\qquad \leq g\biggl(
\frac{C_k^{2}}{[\Lambda_A(\cZ)\ve2^{-k-2}]^{2}+
[\Lambda_B(\cZ)\ve2^{-k-2}]C_k}
\biggr)\\[1pt]
&&\qquad\leq g\biggl(
\frac{\tilde{C}_k^{2}}{\Lambda_A^2(\cZ)+
\Lambda_B(\cZ)\tilde{C}_k}
\biggr),\nonumber
\end{eqnarray*}
where we denoted
$\tilde{C}_k= C_k 2^{k+2}$.
Taking into account that $9(k+1)^{-2}2^{k-2}\geq1$ for any $k\geq0$,
and by definition of $C^{*}(y,\cZ)$,
we obtain for any $y>0$ that\looseness=1
\[
\frac{\tilde{C}_k^{2}}{\Lambda_A^2(\cZ)+
\Lambda_B(\cZ)\tilde{C}_k} \geq9y(k+1)^{-2}2^{k-2}.\vadjust{\goodbreak}
\]\looseness=0
Hence, for any $z,z^\prime\in\cZ_{k+1}\times\cZ_k$ satisfying
$\mathrm{d}(z,z^\prime)\leq\ve2^{-k-2}$ one has
%
%
\begin{equation}
\label{p1000}
\mathrm{P}\biggl\{\Psi\bigl(\xi_{\phi[z]-\phi[z^\prime]}\bigr)\geq\frac{6\ve
[\varkappa_U(Z)+C^{*}(y,\cZ)]}
{\pi^2(k+1)^2}\biggr\}\leq g\bigl(9y 2^{k-2}(k+1)^{-2}
\bigr).\hspace*{-40pt}\vadjust{\goodbreak}
\end{equation}
Noting that the right-hand side of (\ref{p1000}) does not depend on
$z,z^{\prime}$ we get
%
%
\begin{equation}
\label{p6}
I_2\leq\sum_{k=0}^{\infty}\{N_{\cZ,\mathrm{d}}
(\ve2^{-k-1})\}^{2}
g\bigl(9y 2^{k-2}(k+1)^{-2} \bigr).
\end{equation}
The theorem statement
follows now from (\ref{p4}), (\ref{p5}) and (\ref{p6}).

%
\vspace*{-3pt}\subsection{\texorpdfstring{Proof of Proposition \protect\ref{l_uniform-2}}{Proof of Proposition 2}}
$\!\!\!$Let $Z_l$, $l\,{=}\,1, \ldots, N_{\bZ,\mathrm{d}}(\ve/8)$ be $\mathrm{d}$-balls
of~\mbox{radius} $\ve/8$ 
forming a minimal covering of the set $\bZ$.
For any $0\leq j\leq[\ve^{-1}\log_2(R/r)-1]_+$
[without loss of generality, we assume that
$\ve^{-1}\log_2(R/r)$ is integer], let $\delta_j=r2^{\ve j}$,
and put
\[
\tilde{\bZ}_{\delta_{j+1}} =\{\zeta\in\bZ\dvtx\delta_j <
U(\phi[\zeta])\leq\delta_{j+1} \}.
\]
Note that
$\tilde{\bZ}_{\delta_j}\!\subseteq\!\bZ_{\delta_{j}}$ for all $j$ because $\epsilon\!\in\!(0,1]$; recall
that $\bZ_{a}$ is defined in~(\ref{eq:Z-a}).

We have $Z_l=\bigcup_{j=0}^{[\ve^{-1}\log_2(R/r)-1]_+} \{Z_l\cap\tilde
{\bZ}_{\delta_{j+1}}\}$
for any $l=1, \ldots, N_{\cZ,\mathrm{d}}(\ve/8)$.
Therefore,
for any $y>0$,
%
%
\begin{equation}\hspace*{28pt}
\label{p1}
\Psi_{u_\ve}^*(y, Z_l)\leq
\sup_{j=0,\ldots, [\ve^{-1}\log_2(R/r)-1]_+}
\Bigl[ \sup_{\zeta\in Z_l\cap\tilde{\bZ}_{\delta_{j+1}}}
\Psi\bigl(\xi_{\phi[\zeta]}\bigr)-
u_\ve C^*(y)\delta_{j}\Bigr].
\end{equation}
Let $0\leq j\leq[\ve^{-1}\log_2(R/r)-1]_+$ be fixed;
then using the definition of $\Lambda_A$ and $\Lambda_B$ [see
(\ref{eq:assumption_parameter-2}) and (\ref{eq2:assuption_parameter-3})]
and the fact that $\tilde{\bZ}_{\delta_{j+1}}\subseteq\bZ_{\delta_{j+1}}$
we have that
\begin{eqnarray*}
C^*(y)&\geq& 1+\delta^{-1}_{j+1}\bigl[2\sqrt{y}\Lambda_A(\bZ_{\delta
_{j+1}})+2y\Lambda_B(\bZ_{\delta_{j+1}})\bigr]
\\
&\geq&
1+\delta^{-1}_{j}\bigl[\sqrt{y}\Lambda_A(\bZ_{\delta_{j+1}})+y\Lambda_B(\bZ
_{\delta_{j+1}})\bigr]\\
&\geq& 1+\delta^{-1}_{j}\bigl[\sqrt{y}\Lambda_A(
\tilde{\bZ}_{\delta_{j+1}})+y\Lambda_B(\tilde{\bZ}_{\delta_{j+1}})\bigr].
\end{eqnarray*}
Therefore
\begin{eqnarray*}
C^*(y)\delta_{j}&\geq&\delta_{j}+\bigl[\sqrt{y}\Lambda_A(
\tilde{\bZ}_{\delta_{j+1}})+y\Lambda_B(
\tilde{\bZ}_{\delta_{j+1}})\bigr]\\
&\geq&2^{-\ve}\varkappa_U(
\tilde{\bZ}_{\delta_{j+1}})+
C^{*}(y,\tilde{\bZ}_{\delta_{j+1}}),
\end{eqnarray*}
since by the premise of the proposition
$\delta_j=2^{-\epsilon}\delta_{j+1}\geq
2^{-\epsilon}\varkappa_U(\bZ_{\delta_{j+1}})\geq
2^{-\epsilon}\times\varkappa_U(\tilde{\bZ}_{\delta_{j+1}})$.
Note also that the definition of $C^{*}(\cdot,\cdot)$ implies that
$C^{*}(\cdot,\cZ_1)\leq C^{*}(\cdot,\cZ_2)$ whenever
$\cZ_1\subseteq\cZ_2$.
Thus, we have for
any $0\leq j\leq[\ve^{-1}\log_2(R/r)-1]_+$ and any $l=1, \ldots,
N_{\bZ,\mathrm{d}}(\ve/8)$
%
%
\begin{equation}
\label{eq:p-new1}
u_\ve C^*(y)\delta_{j}
\geq(1+\ve)[\varkappa_U(Z_l\cap\tilde{\bZ}_{\delta_{j+1}})+
C^{*}(y,Z_l\cap\tilde{\bZ}_{\delta_{j+1}})].
\end{equation}
Taking into account (\ref{p1}), we obtain
\begin{eqnarray*}
&&\mathrm{P}\{\Psi_{u_\ve}^*(y, Z_l)\geq0\}\\[-2pt]
&&\qquad\leq
\sum_{j=0}^{[\ve^{-1}\log_2(R/r)-1]_+} \mathrm{P}\Bigl\{\sup_{\zeta\in
Z_l\cap\tilde{\bZ}_{\delta_{j+1}}}
\Psi\bigl(\xi_{\phi[\zeta]}\bigr)\geq(1+\ve)[\varkappa_U(Z_l\cap
\tilde{\bZ}_{\delta_{j+1}})\\[-2pt]
&&\qquad\quad\hspace*{215pt}{}+C^{*}(y,Z_l\cap
\tilde{\bZ}_{\delta_{j+1}})]\Bigr\}.
\end{eqnarray*}
Applying Proposition \ref{l_uniform}
for the sets $Z_l\cap\tilde{\bZ}_{\delta_{j+1}}$, we get for any $y>0$
\begin{eqnarray*}
\mathrm{P}\{\Psi_{u_\ve}^*(y, Z_l)\geq0\} &\leq&\sum_{j=0}^{[\ve
^{-1}\log_2(R/r)-1]_+}
\cL^{(\epsilon)}_{g, \bZ_{\delta_{j+1}}}(y)\\[-2pt]
&=&
\sum_{j=0}^{[\ve^{-1}\log_2(R/r)-1]_+}
\cL^{(\ve)}_{g}\bigl(y, r2^{\ve(j+1)}\bigr).
\end{eqnarray*}
It remains to note that the right-hand side of the last inequality does
not depend on~$l$; thus, we come to the first assertion of the proposition.

Now we derive the bound for the moments of $\Psi^*_{u_\ve}(y, \bZ)$.
We have from~(\ref{p1}) with $y>0$ that for any $q\geq1$
%
%
\begin{eqnarray}
\label{p10}\qquad
&&\mathrm{E} \Bigl(\sup_{\zeta\in\bZ}\bigl\{\Psi\bigl(\xi_{\phi[\zeta]}\bigr)-
u_\ve C^*(y)U(\phi[\zeta])\bigr\}\Bigr)^q_+\nonumber\\
&&\qquad\leq
\sum_{l=1}^{N_{\bZ,\mathrm{d}}(\ve/8)} \sum_{j=0}^{[\ve^{-1}\log_2(R/r)-1]_+}
\mathrm{E}
\Bigl( \sup_{\zeta\in Z_l\cap\tilde{\bZ}_{\delta_{j+1}}}
\bigl\{\Psi\bigl(\xi_{\phi[\zeta]}\bigr)-u_{\ve}C^*(y)\delta_{j}
\bigr\}\Bigr)^q_+\\
&&\qquad=: \sum_{l=1}^{N_{\bZ,\mathrm{d}}
(\ve/8)} \sum_{j=0}^{[\ve^{-1}\log_2(R/r)-1]_+} E_j(l).\nonumber
\end{eqnarray}
For $l=1,\ldots, N_{\bZ,\mathrm{d}}(\ve/8)$ and
$0\leq j\leq[\epsilon^{-1}\log_2(R/r)-1]_+$ we have
%
%
\begin{eqnarray}
\label{p11}
E_j(l)&=&q\int_{u_{\ve}C^*(y)\delta_{j}}^{\infty}
[x-u_{\ve}C^*(y)\delta_{j}]^{q-1}\nonumber\\[-2pt]
&&\hspace*{49.2pt}{}\times
\mathrm{P} \Bigl\{\sup_{\zeta\in Z_l\cap\tilde{\bZ}_{\delta_{j+1}}}
\Psi\bigl(\xi_{\phi[\zeta]}\bigr)\geq x
\Bigr\}\,{\rd x}\nonumber\\[-1pt]
&=&[u_{\ve} C^*(y)]^q\delta^q_{j} q\nonumber\\[-2pt]
&&{}\times\int_{1}^{
\infty}(z-1)^{q-1}
\mathrm{P} \Bigl\{\sup_{\zeta\in Z_l\cap\tilde{\bZ}_{\delta_{j+1}}}
\Psi\bigl(\xi_{\phi[\zeta]}\bigr)\geq z u_\ve C^*(y)\delta_{j} \Bigr\}\,{\rd z}
\\[-2pt]
&\leq&[u_{\ve} C^*(y)]^q\delta^q_{j} q\nonumber\\[-2pt]
&&{}\times\int_{1}^{\infty}(z-1)^{q-1}
\mathrm{P} \Bigl\{\sup_{\zeta\in Z_l\cap\tilde{\bZ}_{\delta_{j+1}}}\Psi\bigl(\xi
_{\phi[\zeta]}\bigr)\geq
u_\ve C^*(yz)\delta_{j} \Bigr\}\,{\rd z}\nonumber\vadjust{\goodbreak}\\[-2pt]
&\leq&[u_{\ve}C^*(y)]^q\delta^q_{j} q\int_{1}^{\infty}(z-1)^{q-1}\cL
^{(\ve)}_g\bigl(yz,r2^{\ve(j+1)}\bigr)
\,\rd z.\nonumber
\end{eqnarray}
Here the third line follows from
$zC^{*}(y)\geq C^{*}(yz)$ for any $z\geq1$, and the last line is a
consequence of (\ref{eq:p-new1}) and the probability bound
established above.

The second statement of the theorem follows now from
(\ref{p10}) and (\ref{p11}) since the right-hand side in (\ref{p11})
does not depend on $l$.

%
\section{\texorpdfstring{Proof of Theorem \protect\ref{fixed_w1}}{Proof of Theorem 1}}
\label{sec:proof-fixed_w}
\subsection{Preliminaries}
For convenience
in this section,
we present
some well-known results that will be repeatedly used in the proofs.
\subsubsection*{Empirical processes}
Let $\cF$ be a countable set of functions $f\dvtx\cX\to\bR$.
Suppose that $\bE f(X)=0$, $\|f\|_\infty\leq b$, $\forall f\in\cF$ and put
\[
Y=\sup_{f\in\cF}\sum_{i=1}^{n}f(X_i),\qquad
\sigma^2=\sup_{f\in\cF}\bE[f(X)]^2.
\]
\begin{lemma}
\label{bousquet}
For any $x\geq0$
\[
\bP\{Y-\bE Y\geq x\}\leq
\exp\biggl\{-\frac{x^2}{2n\sigma^2 + 4b\bE Y + ({2}/{3})b x}\biggr\}.
\]
\end{lemma}

The statement of the lemma
is an immediate consequence of the
the Bennett inequality for empirical processes
[see \citet{bousquet}] and the standard
arguments allowing to derive the Bernstein inequality from the Bennett
inequality.

\subsubsection*{Inequalities for sums of independent random variables}
We recall the well-known
Rosenthal and Bahr--Esseen [see \citet{bahr}]
bounds on the moments of sums of independent random variables.
\begin{lemma}
\label{l_rozen}
Let $Y_1,\ldots, Y_n$ be independent random variables,
$\bE Y_i=0$, $i=1,\ldots, n$.
Then
\begin{eqnarray*}
\bE\Biggl|\sum_{i=1}^n Y_i\Biggr|^p &\leq&[c_1(p)]^p
\Biggl\{ \sum_{i=1}^n \bE|Y_i|^p
+ \Biggl(\sum_{i=1}^n \bE Y_i^2\Biggr)^{p/2}\Biggr\},\qquad p>2;\\
\bE\Biggl|\sum_{i=1}^n Y_i\Biggr|^p &\leq&
2\sum_{i=1}^n \bE|Y_i|^p,\qquad p\in[1,2),
\end{eqnarray*}
where\vadjust{\goodbreak} $c_1(p)=15p/\ln p$.
\end{lemma}

The constant
$c_1(p)=15p/\ln p$ in the Rosenthal inequality
is obtained by symmetrization of the inequality
of Theorem 4.1 in \citet{Johnson}.

\subsubsection*{Norms of integral operators}
The next statement presents inequalities for norms
of integral operators.
\begin{lemma}\label{folland}
Let $(\cT,\mT,\tau)$ and $(\cX,\mX,\chi)$
be $\sigma$-finite spaces, $w$ be a $(\mT\times\mX)$-measurable
function on $\cT\times\cX$, and let
\[
M_{p,\tau,\chi}(w):=
\sup_{x\in\cX} \|w(\cdot, x)\|_{p,\tau} \vee
\sup_{t\in\cT} \|w(t,\cdot)\|_{p,\chi}.
\]
If $R\in\bL_p(\cX,\chi)$ and
$\cI_R(t):=\int w(t,x)R(x)\chi(\rd x)$
then
the following statements hold:
\begin{enumerate}[(a)]
\item[(a)]
For any $p\in[1,\infty]$
%
%
\begin{equation}\label{eq:folland-1}
\|\cI_R\|_{p,\tau} \leq
M_{1,\tau,\chi}(w)\|R\|_{p,\chi}.
\end{equation}
\item[(b)]
For any $1<p< r<\infty$
%
%
\begin{equation}\label{eq:folland-2}
\|\cI_R\|_{r,\tau} \leq c_2(p) M_{q,\tau,\chi}(w)
\|R\|_{p,\chi},
\end{equation}
where
$\frac{1}{q}=1+\frac{1}{r}-\frac{1}{p}$,
and $c_2(p)$ is a numerical constant independent of~$w$.
\end{enumerate}
\end{lemma}

The statements of the lemma can be found in
\citet{folland}, Theorems~6.18 and 6.36.

Note that if
$\chi=\nu^\prime:=f\nu$ then $M_{p,\tau,\chi}(w)=M_p(w)$, $\forall w$ [see
(\ref{eq:M-s})].
If $\cT=\cX=\bR^d$, $\tau$ and $\chi$ are the Lebesgue measures, and if
$w(t,x)$ depends on the difference $t-x$ only,
then $c_2(p)=1$, and (\ref{eq:folland-2}) is the well-known
Young inequality.

%

\subsection{\texorpdfstring{Proof of Theorem \protect\ref{fixed_w1}}{Proof of Theorem 1}}
We begin with two technical lemmas; their
proofs are given in the \hyperref[app]{Appendix}.\vspace*{-1pt}
\begin{lemma}
\label{l_countable_xi}
Let $\bB_{{s}/({s-1})}$ be the unit ball in $\bL_{{s}/({s-1})}(\cT
,\tau)$,
and suppose that
Assumption \ref{assumptionA1} hold.
Then, there exists a
\textit{countable} set $\mL\subset\bB_{{s}/({s-1})}$ such that
\[
\|\xi_w\|_{s,\tau}=\sup_{l\in\mL}\int l(t)\xi_{w}(t)\tau(\rd t).\vspace*{-1pt}
\]
\end{lemma}

\begin{lemma}
\label{lepski}
Let
$\bar{w}(t,x)=w(t,x)-\bE w(t,X)$; then for all $p\geq1$ one has:
\begin{enumerate}[(a)]
\item[(a)]
$
\|\bar{w}(\cdot,x)\|_{p,\tau}\leq2\sup_{x\in\cX}\|w(\cdot,x)\|_{p,\tau}.
$
\item[(b)]
$M_{p}(\bar{w})\leq2 M_{p}(w)$.
\end{enumerate}
\end{lemma}

We break the proof
of Theorem \ref{fixed_w1} into several steps.

\subsubsection*{Step 1: Reduction to empirical process}
We obtain from Lemma \ref{l_countable_xi}
\begin{eqnarray*}
\|\xi_w\|_{s,\tau}&=&\sup_{l\in\mL}\int l(t)\xi_{w}(t)\tau(\rd t)\\
&=&\sup_{l\in\mL}\sum_{i=1}^{n}\int l(t)\bar{w}(t,X_i)\tau(\rd
t)\\
&=&\sup_{\lambda\in\Lambda}\sum_{i=1}^{n}\lambda(X_i),
\end{eqnarray*}
where
\[
\Lambda=\biggl\{\lambda\dvtx\cX\to\bR\dvtx\lambda(x)=\int l(t)\bar{w}(t,x)\tau
(\rd
t), l\in\mL\biggr\}.
\]
Thus,
%
%
\begin{equation}
\label{duality}
\|\xi_w\|_{s,\tau}=\sup_{\lambda\in\Lambda}\sum_{i=1}^{n}\lambda(X_i)=:Y
\end{equation}
and, obviously, $\bE\lambda(X)=0$.
The idea now is to apply Lemma \ref{bousquet} to the random variable
$Y$.

\subsubsection*{Step 2: Some upper bounds}
In order to apply
Lemma \ref{bousquet}, we need to bound from above the following quantities:
(i)
$\bE Y$;
(ii)
$b:=\sup_{\lambda\in\Lambda}\|\lambda\|_\infty$;
and (iii)
$\sigma^2:=\sup_{\lambda\in\Lambda}\bE\lambda^2(X)$.

(i) \textit{Upper bound for
$\bE Y$.}
Applying the H\"older inequality, we get from (\ref{duality})
\[
\bE\Biggl[\sup_{\lambda\in\Lambda}\sum_{i=1}^{n}\lambda(X_i)\Biggr]=\bE\|\xi_w\|
_{s,\tau}\leq
[\bE\|\xi_w\|^s_{s,\tau}]^{{1}/{s}}=\biggl[\int\bE|\xi_w(t)|^s\tau(\rd
t)\biggr]^{{1}/{s}}.
\]

If $s\in[1,2]$, then for all $t\in\cT$
\[
\bE|\xi_w(t)|^s\leq[\bE|\xi_w(t)|^2]^{{s}/{2}}
\leq[n\bE w^2(t,X)]^{{s}/{2}}=\biggl[n\int w^2(t,x)
f(x)\nu(\rd x)\biggr]^{{s}/{2}}.
\]
Thus, we have for all $s\in[1,2]$
%
%
\begin{equation}
\label{s_less_2_1}
\bE Y=\bE\Biggl[\sup_{\lambda\in\Lambda}\sum_{i=1}^{n}\lambda(X_i)\Biggr]\leq\sqrt{n}
\Sigma_s(w,f).
\end{equation}

Note that the same quantity can be bounded from above in a different way.
Indeed, in view of the Barh--Esseen inequality (the
second statement of Lem\-ma~\ref{rosen})
\[
\bE|\xi_w(t)|^s
\leq2n\bE|\bar{w}(t,X)|^s= 2^{1+s}n \bE|w(t,X)|^s
\]
and we obtain for all $s\in[1,2]$
%
%
\begin{equation}
\label{s_less_2_2}
\bE Y=\bE\Biggl[\sup_{\lambda\in\Lambda}\sum_{i=1}^{n}\lambda(X_i)\Biggr]\leq
2^{1+1/s}n^{1/s}M_s(w).
\end{equation}
We get finally from (\ref{s_less_2_1}) and (\ref{s_less_2_2})
%
%
\begin{equation}
\label{s_less_2}
\bE Y\leq\bigl\{\sqrt{n} \Sigma_s(w,f)\bigr\}\wedge\{4n^{1/s}M_s(w)\}.
\end{equation}

If $s=2$, we obtain a bound independent of $f$: indeed, in this case
%
%
\begin{eqnarray}
\label{s=2}
\bE Y&=&\bE\Biggl[\sup_{\lambda\in\Lambda}\sum_{i=1}^{n}\lambda(X_i)\Biggr]\leq\sqrt{n}
\biggl[\int\!\!\int w^2(t,x)f(x)\nu(\rd x)\tau(\rd t)\biggr]^{{1}/{2}}\nonumber\\[-8pt]\\[-8pt]
&\leq&\sqrt{n} M_2(w).\nonumber
\end{eqnarray}

If $s>2$, then applying the Rosenthal inequality
(the first assertion of Lem\-ma~\ref{rosen}) to $\xi_w(t)$,
which is a sum of i.i.d. random variables for any $t\in\cT$,
we get
\[
[\bE(|\xi_w(t)|^s)]^{{1}/{s}}\leq c_1(s)[(n \bE w^2(t,X))^{{s}/{2}}
+n \bE|\bar{w}(t,X)|^s
]^{{1}/{s}}
\]
and, therefore,
%
%
\begin{eqnarray}
\label{rosen}
&&\bE\Biggl[\sup_{\lambda\in\Lambda}\sum_{i=1}^{n}\lambda(X_i)\Biggr]\nonumber\\
&&\qquad\leq c_1(s)\biggl\{
\sqrt{n}
\biggl[\int\biggl(
\int w^2(t,x)f(x)\nu(\rd x)\biggr)^{{s}/{2}}\tau(\rd
t)\biggr]^{{1}/{s}}\\
&&\qquad\quad\hspace*{28.3pt}{}+ 2n^{1/s}\biggl[\int\!\!
\int|w(t,x)|^sf(x)\nu(\rd x)
\tau(\rd t)\biggr]^{{1}/{s}}\biggr\}.\nonumber
\end{eqnarray}
To get the last inequality
we have used that
$\bE|\bar{w}(t,X)|^s\leq2^s \bE|w(t,X)|^s$, for all $s\geq1$.

It is evident that the second integral on the right-hand side of
(\ref{rosen}) does not exceed $M_s(w)$.
Moreover,
since
$(\bE w^2(t,X))^{{s}/{2}}\leq\bE|w(t,X)|^{s}$, $s\geq2$,
the following bound is true
$
\Sigma_s(w,f)\leq M_s(w).
$
We conclude that $\bE Y<\infty$ whenever $M_s(w)<\infty$, and
%
%
\begin{equation}
\label{c1}
\bE Y=\bE\Biggl[\sup_{\lambda\in\Lambda}
\sum_{i=1}^{n}\lambda(X_i)\Biggr]\leq c_1(s)\bigl\{
\sqrt{n} \Sigma_s(w,f) + 2n^{1/s}M_s(w)\bigr\}.
\end{equation}

%

(ii) \textit{Upper bound for $b=\sup_{\lambda\in\Lambda}\|\lambda\|
_\infty$.}
Taking into account that $ l\in\mL\subset\bB_{{s}/({s-1})}$
(Lemma \ref{l_countable_xi})
and applying the
H\"older inequality, we get for any $x\in\cX$
\[
|\lambda(x)|\leq\biggl[\int|w(t,x)-\bE w(t,X)|^s\tau(\rd t)\biggr]^{{1}/{s}}
= \|\bar{w}(\cdot,x)\|_{s,\tau}.
\]
Therefore, in view of Lemma \ref{lepski}(a)
%
%
\begin{equation}
\label{supnorm}
b=\|\lambda\|_\infty\leq2\sup_{x\in\cX}\|w(\cdot,x)\|_{s,\tau}\leq2M_s(w).
\end{equation}

(iii)
\textit{Upper bound on the} ``\textit{dual}'' \textit{variance $\sigma^2$.}
Since $\bE\lambda(X)=0$, we have
\begin{eqnarray*}
\sigma^2 &=&\sup_{\lambda\in\Lambda}\int\lambda^2(x)f(x)\nu(\rd x)\\
&=&
\sup_{ l\in\mL}
\int\biggl[\int\bar{w}(t,x)l(t)\tau(\rd t)\biggr]^2 f(x) \nu(\rd x)\\
&\leq&\sup_{ l\in\bB_{{s}/({s-1})}}
\int\biggl[\int\bar{w}(t,x)l(t)\tau(\rd t)\biggr]^2f(x)\nu(\rd x)\\
&\leq&\sup_{ l\in\bB_{{s}/({s-1})}}
\int\biggl[\int w(t,x)l(t)\tau(\rd t)\biggr]^2f(x)\nu(\rd x).
\end{eqnarray*}
The expression on the right-hand side is bounded differently depending
on the
value of $s$.

If $s\in[1,2)$, then applying the H\"older inequality to the inner integral
in the previous expression we obtain
%
%
\begin{eqnarray}
\label{sigma1}
\sigma^2 &\leq&\int\biggl[\int|w(t,x)|^s\tau(\rd t)\biggr]^{{2}/{s}}f(x)
\nu(\rd x)\nonumber\\[-8pt]\\[-8pt]
&\leq&\sup_{x\in\cX}
\|w(\cdot,x)\|^2_{s,\tau}\leq M^2_{s}(w).\nonumber
\end{eqnarray}
We remark also that the bound given by (\ref{sigma1}) remains true
for all $s\geq1$. This shows, in particular, that
$\sigma$ is always bounded whenever $M_s(w)<\infty$.

If $s=2$, then we apply inequality (\ref{eq:folland-1}) of Lemma \ref{folland}
with $p=2$ and $ \chi(\rd x)=\nu^\prime(\rd x) = f(x)\nu(\rd x)$
to the integral operator $\cI_l(x)=\int w(t,x)l(t)\tau(\rd t)$.
This leads to the following bound
%
%
\begin{equation}
\label{sigma2}
\sigma^2\leq M^2_{1,\tau,\nu^\prime}(w).
\end{equation}

If $s>2$, then we apply inequality (\ref{eq:folland-2})
of Lemma \ref{folland}
with $r=2$, $p=\frac{s}{s-1}$, $q=\frac{2s}{s+2}$ and $\chi=\nu^\prime$
to the integral operator $\cI_l(x)=\int w(t,x)l(t)\tau(\rd t)$. This yields
%
%
\begin{equation}
\label{sigma3}\quad
\sigma^2\leq c_2\bigl(s/(s-1)\bigr) M^2_{q,\tau,\nu^\prime}(w)=c_2\bigl(s/(s-1)\bigr)
M^2_{2s/(s+2),\tau,\nu^\prime}(w).
\end{equation}

\subsubsection*{Step 3: Application of Lemma \protect\ref{bousquet}}

$\!\!\!$1. \textit{Case $s\in[1,2)$.}
Here we
have from~(\ref{s_less_2}), (\ref{supnorm}) and (\ref{sigma1})
\begin{eqnarray*}
\bE Y&\leq&\bigl\{\sqrt{n}
\Sigma_s(w,f)\bigr\}\wedge\{4n^{1/s}M_s(w)\}=:\rho_s(w,f),\\
b&\leq&2M_s(w),\qquad \sigma^2 \leq M^2_{s}(w).
\end{eqnarray*}
Therefore applying Lemma \ref{bousquet}, we have
for all $z>0$
%
%
\begin{eqnarray}\label{eq:PP}
&&\bP\{\|\xi_w\|_{s,\tau}\geq\rho_s(w,f)+z \}\nonumber\\[-8pt]\\[-8pt]
&&\qquad\leq
\exp\biggl\{-\frac{z^2}{2M^{2}_s(w)[n+16n^{1/s}]+[4M_s(w)z/3]}
\biggr\},\nonumber
\end{eqnarray}
where we have used
(\ref{s_less_2_2}) in the denominator of the expression inside of the exponent.

To get the result of the theorem, we note that
the following trivial upper bound follows from
the triangle inequality and the statement (a) of
Lemma~\ref{lepski}:
\[
\|\xi_w\|_{s,\tau}\leq2n M_s(w)\qquad \forall s\geq1.
\]
Thus, the probability in (\ref{eq:PP})
is equal to zero if $z> 2n M_s(w)$; hence,
we can replace $z$
by $ 2n M_s(w)$ in the denominator of the expression
on the right-hand side.
This leads to the statement of the theorem for
$s\in[1,2)$.

2. \textit{Case $s=2$.}
We have from (\ref{s=2}), (\ref{supnorm}) and (\ref{sigma2})
\[
\bE Y\leq\sqrt{n}M_2(w),\qquad
b\leq2M_2(w),\qquad \sigma^2
\leq M^2_{1,\tau,\nu^\prime}(w).
\]
Thus, for all $z>0$
\begin{eqnarray*}
&&\bP\bigl\{\|\xi_w\|_{2,\tau}\geq\sqrt{n} M_2(w)+z \bigr\}\\
&&\qquad\leq
\exp\biggl\{-\frac{z^2}{2[ nM^2_{1,\tau,\nu^\prime}(w)+
4\sqrt{n} M^2_2(w)+({2}/{3})M_2(w)z]}
\biggr\},
\end{eqnarray*}
and the statement of Theorem \ref{fixed_w1} is established for $s=2$.


3. \textit{Case $s>2$.}
We have from (\ref{c1}), (\ref{supnorm}) and (\ref{sigma3})
\begin{eqnarray*}
\bE Y&\leq &c_1(s)\bigl[ \sqrt{n} \Sigma_s(w,f)
+ 2n^{1/s}M_s(w)\bigr],\\
b&\leq&2M_s(w);\qquad
\sigma^2
\leq c_2\bigl(s/(s-1)\bigr) M^2_{2s/(s+2),\tau,\nu^\prime}(w).
\end{eqnarray*}
Thus, for any $z>0$ we get
\begin{eqnarray*}
&&\bP\bigl\{\|\xi_w\|_{s,\tau} \geq c_1(s)\bigl[
\sqrt{n} \Sigma_s(w,f)
+ 2n^{1/s}M_s(w)\bigr]+z\bigr\}\\
&&\qquad\leq
\exp\bigl\{-{z^2 }\bigl(2c_3(s)
\bigl[n M^2_{{2s}/({s+2}),\tau,\nu^\prime}(w)+
4\sqrt{n}\Sigma_s(w,f) M_s(w)\\
&&\hspace*{172.4pt}{}+ 8n^{1/s}M^2_s(w)
+\tfrac{2}{3}M_s(w)z\bigr]\bigr)^{-1}
\bigr\},
\end{eqnarray*}
where $c_3(s)$ is given in (\ref{eq:c-*}). This completes the proof
of the theorem for the case of $s>2$.

We conclude by establishing the inequalities in (\ref{eq:sigg}).
In order to derive the first inequality, we
apply (\ref{eq:folland-1}) of Lemma \ref{folland}
with $p=s/2>1$, $\chi=\nu$ to the
integral operator $\cI_f(t):=\int w^2(t,x)f(x)\nu(\rd x)$. This yields
\[
\biggl[\int\biggl(
\int w^2(t,x)f(x)\nu(\rd x)\biggr)^{s/2}\tau(\rd t)\biggr]^{1/s}\leq M_2(w)\bigl\|\sqrt
{f}\bigr\|_{s,\nu},
\]
as claimed.
The second inequality in (\ref{eq:sigg}) follows straightforwardly
from the definition
of $M_{p,\tau, \nu^\prime}$ and $M_p$.\vadjust{\goodbreak}

%
\section{\texorpdfstring{Proofs of Theorem \protect\ref{t_random} and Corollary \protect\ref{cor:th_random}}
{Proofs of Theorem 3 and Corollary 4}}
%
\subsection{\texorpdfstring{Proof of Theorem \protect\ref{t_random}}{Proof of Theorem 3}}
First, we specify the constants appearing in the statement of the theorem:
\begin{eqnarray*}
T_{1,\ve}&:=&
\biggl(\frac{2^{q(\ve+1)}}{2^{q\epsilon}-1} \Gamma(q+1)+1\biggr)
N_{\bZ,\mathrm{d}}(\ve/8)
(2u_\ve R_\xi)^q
[1\vee\log_2 (R_\xi/r_\xi)]
\bigl[1+L^{(\ve)}_{\exp}\bigr],\\
T_{2,\ve}&:=&[c_1(s)+2]^{q}N_{\bZ,\mathrm{d}}(\ve/8)
[1\vee\log_2(R_\xi/r_\xi)] \bigl[1+L^{(\ve)}_{\exp}\bigr].
\end{eqnarray*}

Recall that in view of (\ref{eq:W}), any $w\in\cW$ is represented as
$w=\phi[\zeta]$ for some \mbox{$\zeta\in\bZ$}.
For every $0\leq j\leq[\log_2(R_\xi/r_\xi)-1]_+$
[without loss of generality, we assume that $\log_2(R_\xi/r_\xi)$
is an integer number], put $\delta_j=r2^{j+1}$,
and define the random events
\[
\cA:=\bigcap_{j=0}^{[\log_2(R_\xi/r_\xi)-1]_+}\cA_j,\qquad \cA_j:=\Bigl\{
\sup_{\zeta\in\bZ_{\delta_j}}\bigl\|\xi_{\phi^2[\zeta]}\bigr\|_{s/2,\tau}\leq
[2(1+\ve)\gamma\delta_j]^{2}\Bigr\}.
\]

(i)
The following trivial inequality holds:
\begin{eqnarray*}
&&\sup_{\zeta\in\bZ}\bigl\{\bigl\|\xi_{\phi[\zeta]}\bigr\|_{s,\tau}-
\bar{u}_\ve(\gamma) C_\xi^*(y)\hat{U}_\xi(\phi[\zeta])\bigr\}_+\\
&&\qquad \leq
\sup_{\zeta\in\bZ}\bigl\{\bigl\|\xi_{\phi[\zeta]}\bigr\|_{s,\tau}-
u_\ve C_\xi^*(y) U_\xi(\phi[\zeta],f)\bigr\}_+\\
&&\qquad\quad{} +
u_\ve C_\xi^*(y)\sup_{\zeta\in\bZ}U_\xi(\phi[\zeta],f).
\end{eqnarray*}
Therefore,
%
%
\begin{eqnarray}
\label{q3}
&&\bE
\sup_{\zeta\in\bZ}\bigl\{\bigl\|\xi_{\phi[\zeta]}\bigr\|_{s,\tau}-
\bar{u}_\ve(\gamma) C_\xi^*(y)\hat{U}_\xi(\phi[\zeta])\bigr\}^q_+
\nonumber
\\
&&\qquad \leq\bE\Bigl[\sup_{\zeta\in\bZ}
\bigl\{\bigl\|\xi_{\phi[\zeta]}\bigr\|_{s,\tau}-\bar{u}_\ve(\gamma)
C_\xi^*(y)\hat{U}_\xi(\phi[\zeta])\bigr\}^q_+ {\mathbf1}(\cA)\Bigr]
\nonumber\\[-8pt]\\[-8pt]
&&\qquad\quad{} + 2^{q-1}\bE
\sup_{\zeta\in\bZ}\bigl\{\bigl\|\xi_{\phi[\zeta]}\bigr\|_{s,\tau}-
u_\ve C_\xi^*(y)U_\xi(\phi[\zeta],f)\bigr\}^q_+\nonumber\\
&&\qquad\quad{}+ 2^{q-1}[u_\ve C_\xi^*(y)R_\xi]^q
\sum_{j=0}^{[\log_2(R_\xi/r_\xi)-1]_+}\bP\{\bar{\cA}_j\},
\nonumber
\end{eqnarray}
where $\bar{\cA}_j$ denotes the event complementary
to $\cA_j$, and ${\mathbf1}(\cA)$ is the indicator of the event $\cA$.
The second term on the right-hand side is bounded using Theorem \ref{t_uniform};
our current goal is to bound the first and the third terms.

Note that, if the event $\cA$ occurs then for every $\zeta\in\bZ$
%
%
\begin{eqnarray}
\label{q2}
&&U_\xi(\phi[\zeta],f)[1+4c_1(s)(1+\ve)\gamma]\nonumber\\[-8pt]\\[-8pt]
&&\qquad \geq
\hat{U}_\xi(\phi[\zeta]) \geq U_\xi(\phi[\zeta],f)
[1-4c_1(s)(1+\ve)\gamma].\nonumber
\end{eqnarray}
Indeed, in view of (\ref{eq:U-xi-f}), (\ref{eq:rho-hat-2}) and
(\ref{eq:r1}) we get
%
%
\begin{eqnarray}
\label{eq:proof_theorem_random_new1}
\hat{U}_\xi(\phi[\zeta])&\geq& U_\xi(\phi[\zeta],f)-
|\hat{U}_\xi(\phi[\zeta])-
U_\xi(\phi[\zeta],f)|\nonumber\\
&=& U_\xi(\phi[\zeta],f)-
c_1(s)\sqrt{n} |
\hat{\Sigma}_s(\phi[\zeta])
- \Sigma_s(\phi[\zeta],f)|
\\
&\geq& U_\xi(\phi[\zeta],f)-
c_1(s)\sqrt{\bigl\|\xi_{\phi^2[\zeta]}\bigr\|_{s/2,\tau}} .\nonumber
\end{eqnarray}
Let $\zeta\in\bZ$ be fixed. Since $\bZ_{\delta_j}$,
$j=0,\ldots, [\log_2(R_\xi/r_\xi)-1]_+$, defined in (\ref{eq:Z-aa}),
form the partition of $\bZ$,
there exists $j_*$ such that
$\zeta\in\bZ_{\delta_{j_*}}$.
Because $\zeta\in\bZ_{\delta_{j_*}}$ implies $U_\xi(\phi[\zeta],f)\geq
\delta_{j_*}/2=\delta_{j_*-1}$,
we obtain from (\ref{eq:proof_theorem_random_new1}) on the event $\cA$
that
%
%
\begin{eqnarray}
\label{eq:proof_theorem_random_new2}
\hat{U}_\xi(\phi[\zeta])&\geq& U_\xi(\phi[\zeta],f)-2c_1(s)(1+\ve)\gamma
\delta_{j_*}\nonumber\\[-8pt]\\[-8pt]
&\geq& U_\xi(\phi[\zeta],f)[1-4c_1(s)(1+\ve)\gamma].\nonumber
\end{eqnarray}
Thus, the right-hand side inequality in (\ref{q2}) is proved.
Similarly,
we have from (\ref{eq:proof_theorem_random_new1}) and (\ref
{eq:proof_theorem_random_new2}) that
\begin{eqnarray*}
\hat{U}_\xi(\phi[\zeta])&\leq& U_\xi(\phi[\zeta],f)+
|\hat{U}_\xi(\phi[\zeta])-U_\xi(\phi[\zeta],f)|
\\
&\leq& U(\phi[\zeta])+
c_1(s)\sqrt{\bigl\|\xi_{\phi^2[\zeta]}\bigr\|_{s/2,\tau}}\\
&\leq& U_\xi(\phi[\zeta],f)[1+4c_1(s)(1+\ve)\gamma].
\end{eqnarray*}
Thus, (\ref{q2}) is proved.

Using the right-hand side inequality in (\ref{q2})
and applying
Theorem \ref{t_uniform}, we obtain
%
%
\begin{eqnarray}\quad
\label{q4}
&&\bE\Bigl[\sup_{\zeta\in\bZ}
\bigl\{\bigl\|\xi_{\phi[\zeta]}\bigr\|_{s,\tau}-
\bar{u}_\ve(\gamma) C_\xi^*(y)\hat{U}_\xi(\phi[\zeta])\bigr\}^q_+
{\mathbf1}(\cA)\Bigr]
\nonumber\\
&&\qquad
\leq\bE\sup_{\zeta\in\bZ}
\bigl\{\bigl\|\xi_{\phi[\zeta]}\bigr\|_{s,\tau}-
u_\ve C_\xi^*(y)U_\xi(\phi[\zeta],f)\bigr\}^q_+
\\
&&\qquad \leq
\frac{2^{q(\ve+1)}u_\ve^q}{2^{q\ve}-1} \Gamma(q+1)
N_{\bZ,\mathrm{d}}(\ve/8) [R_\xi
C_\xi^*(1)]^q
\bigl[1+L_{\exp}^{(\ve)}\bigr]\exp\{-y/2\}.\nonumber
\end{eqnarray}

Now we bound the probability
$\bP\{\bar\mathcal{A}_j\}$.
Let $Z_l$, $l=1, \ldots, N_{\bZ, \mathrm{d}}(\ve/8)$ be
a~minimal covering of
$\bZ$ by balls of radius $\ve/8$ in the metric $\mathrm{d}$.
By definition of $\mathcal{A}_j$, we have
%
%
\begin{equation}
\label{q5}
\bP\{\bar{\cA}_j\}
\leq
\sum_{l=1}^{N_{\bZ,\mathrm{d}}(\ve/8)}
\bP\Bigl\{\sup_{\zeta\in Z_l\cap\bZ_{\delta_j}}
\bigl\|\xi_{\phi^2[\zeta]}\bigr\|_{s/2,\tau} \geq[2(1+\ve)\gamma\delta_j]^{2}
\Bigr\}.
\end{equation}
Note that
\[
[2\gamma\delta_j]^{2} \geq\varkappa_{\tilde{U}}(\bZ_{\delta_j})+
\delta_j^{2}\bigl[\sqrt{y_{\gamma}}\lambda_{\tilde{A}}
+y_{\gamma}\lambda_{\tilde{B}}\bigr]
\geq\varkappa_{\tilde{U}}(Z_l\cap\bZ_{\delta_j})+\delta^{2}_j\bigl[\sqrt
{y}\lambda_{\tilde{A}}
+y\lambda_{\tilde{B}}\bigr];
\]
here the first inequality follows from the condition
$\varkappa_{\tilde{U}}(\bZ_a)=\varkappa_{\tilde{U}}(a)\leq(\gamma a)^2$,
$\forall a\in[r_\xi,R_\xi]$
and from definition
of $y_\gamma$; the second inequality holds by
the inclusion $Z_l\cap\bZ_{\delta_j}\subseteq\bZ_{\delta_j}$ and
because $y\leq y_\gamma$.
Furthermore, by (\ref{eq:definitions-tilde-new}) and
by the above inclusion
\begin{eqnarray*}
\lambda_{\tilde{A}} &\geq&\delta_j^{-2}\Lambda_{\tilde{A}}(\bZ_{\delta_j})
\geq\delta_j^{-2}\Lambda_{\tilde{A}}(\bZ_{\delta_j}\cap Z_l),\\
\lambda_{\tilde{B}} &\geq&\delta_j^{-2}\Lambda_{\tilde{B}}(\bZ_{\delta_j})
\geq\delta_j^{-2}\Lambda_{\tilde{B}}(\bZ_{\delta_j}\cap Z_l),
\end{eqnarray*}
which leads to
\begin{eqnarray*}
[2\gamma\delta_j]^{2} &\geq& \varkappa_{\tilde{U}}
(Z_l\cap\bZ_{\delta_j}) + \sqrt{y}\Lambda_{\tilde{A}}(\bZ_{\delta
_j}\cap Z_l)
+y\Lambda_{\tilde{B}}(\bZ_{\delta_j}\cap Z_l)
\\
&=& \varkappa_{\tilde{U}}
(Z_l\cap\bZ_{\delta_j}) + \tilde{C}_*(y, Z_l\cap\bZ_{\delta_j}),
\end{eqnarray*}
where $\tilde{C}_*(y, \cdot):=\sqrt{y}\Lambda_{\tilde{A}}(\cdot)
+y\Lambda_{\tilde{B}}(\cdot)$ [cf. (\ref{eq:C-*})].

Hence, applying Proposition \ref{l_uniform}, we obtain from (\ref{q5})
that
%
%
\begin{eqnarray} \label{q55}\qquad
\bP\{\bar{\cA}_j\}
&\leq&
\sum_{l=1}^{N_{\bZ,\mathrm{d}}(\ve/8)}
\bP\Bigl\{\sup_{\zeta\in Z_l\cap\bZ_{\delta_j}}
\bigl\|\xi_{\phi^2[\zeta]}\bigr\|_{s/2,\tau}\nonumber\\
&&\hspace*{49.3pt}
\geq(1+\ve)[\varkappa_{\tilde{U}}(Z_l\cap\bZ_{\delta_j})+
\tilde{C}^{*}(y,Z_l\cap\bZ_{\delta_j})]
\Bigr\}\nonumber\\[-8pt]\\[-8pt]
&\leq&
N_{\bZ,\mathrm{d}}(\ve/8)\Biggl[
\exp\{-y\}
+\sum_{k=0}^{\infty}\exp\{2\cE_{\cZ,\mathrm{d}}(\ve2^{-k})
-9y2^{k-3} k^{-2}\}\Biggr]
\nonumber\\
&\leq& N_{\bZ,\mathrm{d}}
(\ve/8)
\bigl[1+L^{(\ve)}_{\exp}\bigr]\exp\{-y/2\},\nonumber
\end{eqnarray}
where we have used that $y\geq1$.

Finally, combining (\ref{q3}), (\ref{q4}), the bound of
Theorem \ref{t_uniform}, and (\ref{q55})
we come to the first assertion of the theorem.
Here we also used that
$C_\xi^{*}(1)\leq C_\xi(y)$ because
$y\geq1$.

(ii)
In order to prove the second statement,
we note first the following nonrandom bound: since
$\hat{\Sigma}_s(w)\leq M_s(w)$
for all $w\in\cW$ and
$s>2$,
\[
\hat{U}_s(w) \leq M_s(w)
\bigl[c_1(s)\sqrt{n}+2n^{1/s}\bigr]
\leq[c_1(s)+2]\sqrt{n}M_s(w)\qquad \forall w\in\cW.
\]
Next, the left-hand side inequality in (\ref{q2}) implies that
for any subset~\mbox{$\cW_0\!\subseteq\!\cW$}
\[
\cA\subseteq
\Bigl\{\sup_{w\in\cW_0}\hat{U}_\xi(w)<
[1+4c_1(s)(1+\ve)\gamma]\sup_{w\in\cW_0}
U_\xi(w,f)\Bigr\} =: \cA_0.
\]
Therefore $\bP(\bar{\cA}_0)\leq\bP(\bar{\cA})$ and
\[
\bE\{[\hat{U}(w)]^q {\mathbf1}(\bar{\cA}_0)\} \leq
[c_1(s)+2]^q \bigl[\sqrt{n}M_s(w)\bigr]^q \bP(\bar{\cA}).
\]
Using (\ref{q55}) with $y=y_\gamma$,
and definition of the event $\cA$, we complete
the proof.

\subsection{\texorpdfstring{Proof of Corollary \protect\ref{cor:th_random}}{Proof of Corollary 4}}
First, as in (\ref{q2}),
we need to bound $\breve{U}_\xi(w):=\max\{\hat{U}_\xi(w), \sqrt
{n}M_2(w)\}$ from above and
from below in terms
of $\bar{U}_\xi(w,f):=\max\{U_\xi(w,f), \sqrt{n}M_2(w)\}$.
Such bounds
are easily derived
from the following trivial fact:
for any positive $A,B$, $C$ and any $\delta\in(0,1)$
\[
A(1+\delta)\geq B\geq A(1-\delta)
\quad\Rightarrow\quad
[A\vee C](1+\delta)\geq[B\vee C]\geq[A\vee C](1-\delta).
\]
Next, (\ref{q4}) remains valid because, by construction,
$U_\xi(w,f)\leq\bar{U}_\xi(w,f)$ and the assumptions, allowing to apply
Theorem \ref{t_uniform} are imposed now on $\bar{U}_\xi(w,f)$ instead of
$U_\xi(w,f)$.
The computations leading to (\ref{q55})
remain also unchanged
if $U_\xi(w,f)$ is replaced
by $\bar{U}_\xi(w,f)$.
Note that now $\lambda_{\tilde{A}}$ and $\lambda_{\tilde{B}}$
are defined via $\bar{U}_\xi(w,f)$.

\section{\texorpdfstring{Proofs of Theorems \protect\ref{t:consec1_new}, \protect\ref{t_random_case}}{Proofs of Theorems 4, 5}}

\subsection{\texorpdfstring{Proof of Theorem \protect\ref{t:consec1_new}}{Proof of Theorem 4}}
The proof is based on an application of Theorem \ref{t_uniform}.

Put
\begin{eqnarray*}
T_{3,\ve}&:=&
\frac{2^{q(\ve+1)}u_\ve^q}{2^{q\ve}-1}
\Gamma(q+1)
N_{\bZ,\mathrm{d}}(\ve/8)
\bigl[1+L^{(\ve)}_{\exp}\bigr][4\bar{\mathrm{w}}_{s}
(1+ 4 n^{1/2-1/s})]^q;\\
T_{4,\ve}&:=&
\frac{2^{q(\ve+1)}u_\ve^q}{2^{q\ve}-1}
\Gamma(q+1)
N_{\bZ,\mathrm{d}}(\ve/8)\bigl[1+L^{(\ve)}_{\exp}\bigr]\\
&&\hspace*{0pt}{}\times
\bar{\mathrm{w}}^q_{2}
\bigl\{1+2\sqrt{2\mu_*\mathrm{f}_\infty^{2}+8n^{-1/2}} +
(8/3)n^{-1/2}\bigr\}^q.
\end{eqnarray*}

We have
$M_p(w)=\|w\|_p$ for all $w\in\cV$ and $p\geq1$, and
(\ref{eq:U-xi-f})
yields
%
%
\begin{equation}\label{eq:rho1}
U_\xi(w,f)=\cases{
4n^{1/s}\|w\|_s, &\quad
$s\in[1,2)$,\cr
\sqrt{n}\|w\|_2, &\quad$s=2$.}
\end{equation}
Therefore, in view of (\ref{eq:r-R-xi})
%
%
\begin{equation}\label{eq:R-xixi}\quad
r_\xi=\cases{
4n^{1/s}\underline{\mathrm{w}}_s, &\quad$s\in[1,2)$,\cr
\sqrt{n} \underline{\mathrm{w}}_2, &\quad$s=2$,}\qquad
R_\xi=\cases{
4n^{1/s} \bar{\mathrm{w}}_s, &\quad$s\in[1,2)$,\cr
\sqrt{n} \bar{\mathrm{w}}_2, &\quad$s=2$.}
\end{equation}
It follows from (\ref{eq:mu-*}), the H\"older inequality
and the formulas
for $A_\xi^2(w)$ and $B_\xi(w)$
immediately after (\ref{eq:U-xi-f})
that
%
%
\begin{eqnarray}\label{eq:A-B}
A_\xi^2(w)&\leq&\cases{
37n\|w\|^{2}_s,&\quad
$s\in[1,2)$,
\vspace*{2pt}\cr
\bigl[2\mathrm{f}_\infty^2 n\mu_* + 8\sqrt{n}\bigr]\|w\|^{2}_2, &\quad
$s=2$,}\nonumber\\[-8pt]\\[-8pt]
B_\xi(w)&=&\cases{
0, &\quad
$s\in[1,2)$,
\cr
\frac{4}{3}\|w\|_2,&\quad$s=2$.}\nonumber
\end{eqnarray}
In order to apply Theorem \ref{t_uniform},
we need to check that $\varkappa_{U_\xi}(a)\leq a$ for all $a\in[r_\xi
,R_\xi]$.

Let $s\,{\in}\,[1,2)$;
here $\bZ_a\,{=}\,\{\zeta\dvtx
a/2\,{<}\,4n^{1/s}\|\phi[\zeta]\|_s\,{=}\,4n^{1/2}\|w\|_s\,{\leq}\,a\}$; see~(\ref{eq:Z-aa}).
By (\ref{eq:rho1}),
Assumption \ref{assL} and because $\bZ_a\subseteq\bZ_{s}(a/4)$ we have
\[
\sup_{\zeta_1,\zeta_2\in\bZ_a}
\frac{U_\xi(\phi[\zeta_1]-\phi[\zeta_2],f)}{\mathrm{d}(\zeta_1,\zeta_2)}
\leq\sup_{\zeta_1,\zeta_2\in\bZ_s(a/4)}
\frac{4n^{1/s}\|\phi[\zeta_1]-\phi[\zeta_2]\|_s}{\mathrm{d}(\zeta
_1,\zeta_2)} \leq a.
\]
If $s=2$, then $\bZ_a=\{\zeta\dvtx a/2<\sqrt{n}\|\phi[\zeta]\|_2=\sqrt
{n} \|
w\|_2\leq a\}$, and
again by Assumption \ref{assL}
$
\sup_{\zeta_1,\zeta_2\in\bZ_a} [\sqrt{n}\|\phi[\zeta_1]-\phi[\zeta_2]\|_2/
\mathrm{d}(\zeta_1,\zeta_2)]\leq a$.
Thus, $\varkappa_{U_\xi}(a)\leq a$ for all $a\in[r_\xi,R_\xi]$,
and Theorem \ref{t_uniform} can be applied.
To this end, we should compute the quantities
$\Lambda_{A_\xi}$ and $\Lambda_{B_\xi}$ [see (\ref{eq:assumption_parameter-2}),
(\ref{eq2:assuption_parameter-3}) and (\ref{eq:definitions-new})].

For $s\in[1,2)$, we have by
(\ref{eq:A-B}), definition of $\bZ_a$ and Assumption \ref{assL} that
\begin{eqnarray*}
\sup_{\zeta\in\bZ_a} A_\xi(\phi[\zeta])&=&
\sup_{\zeta\in\bZ_a} \sqrt{37n} \|\phi[\zeta]\|_s =\frac{\sqrt{37}}{4}
a n^{1/2-1/s},
\\
\sup_{\zeta_1,\zeta_2\in\bZ_a} \frac{A_\xi(\phi[\zeta_1]-\phi[\zeta_2])}
{\mathrm{d}(\zeta_1,\zeta_2)} &\leq&
\sup_{\zeta_1,\zeta_2\in\bZ_s(a/4)}\frac{\sqrt{37} n^{1/2}\|\phi[\zeta_1]-
\phi[\zeta_2]\|_s}{\mathrm{d}(\zeta_1,\zeta_2)}\\
&\leq&\frac{\sqrt{37}}{4}
a n^{1/2-1/s}.
\end{eqnarray*}
Similarly, if $s=2$ then $\sup_{\zeta\in\bZ_a} A_\xi(\phi[\zeta])\leq
a(2\mathrm{f}_\infty^2 \mu_*+8n^{-1/2})^{1/2}$ and
\begin{eqnarray*}
\sup_{\zeta_1,\zeta_2\in\bZ_a} \frac{A_\xi(\phi[\zeta_1]-\phi[\zeta_2])}
{\mathrm{d}(\zeta_1,\zeta_2)} &\leq&
\sup_{\zeta_1,\zeta_2\in\bZ_2(a)} \bigl[2\mathrm{f}_\infty^2n \mu_*+8\sqrt{n}\bigr]^{1/2}
\frac{\|\phi[\zeta_1]-\phi[\zeta_2]\|_2}
{\mathrm{d}(\zeta_1,\zeta_2)}
\\
&\leq& a\biggl(2\mathrm{f}_\infty^2 \mu_*+\frac{8}{\sqrt{n}}\biggr)^{1/2}.
\end{eqnarray*}
These computations and similar computations for $\Lambda_{B_\xi}$
yield
\begin{eqnarray*}
\Lambda_{A_\xi} &\leq&
\cases{
\frac{\sqrt{37}}{4} n^{1/2-1/s}, &\quad$s\in
[1,2)$,\vspace*{2pt}\cr
[2\mathrm{f}_\infty^2 \mu_* + 8n^{-1/2}
]^{1/2}, &\quad$s=2$,}\\
\Lambda_{B_\xi}&=&\cases{
0, &\quad$s\in[1,2)$,\vspace*{2pt}\cr
\frac{4}{3} n^{-1/2}, &\quad$s=2$.}
\end{eqnarray*}
Recall that $C_\xi^*(y)=1+2\sqrt{y}\Lambda_{A_\xi}+2y\Lambda_{B_\xi}$
[see (\ref{eq:definitions-new-new})].
Therefore if for arbitrary $z>0$, we set
\[
y=\cases{
\dfrac{4}{37} n^{(2/s)-1} z^2,&\quad$s\in[1,2)$,\vspace*{2pt}\cr
\dfrac{z^2}{8} [\mathrm{f}_\infty^2 \mu_* + 4n^{-1/2}]^{-1}, &\quad$s=2$,}
\]
then we get
$C^*_\xi(y)=1+z$ if $s\in[1,2)$ and
\[
C^*_\xi(y)= 1+z + \frac{z^2}{3\sqrt{n}[\mathrm{f}_\infty^2\mu
_*+4n^{-1/2}]} \leq
1+ z+ \frac{z^2}{12},
\]
if $s=2$.
Then the statements (i) and (ii) follow by application of
the moment bound of
Theorem \ref{t_uniform}.
Observe that
$C_\xi(1)=1+ \frac{\sqrt{37}}{2}n^{-1/2-1/s}$ for $s\in[1,2)$,
and $C_\xi(1)=1+2[2\mathrm{f}_\infty^2 \mu_* + 8n^{-1/2}
]^{1/2}+\frac{4}{3}n^{-1/2}$ for $s=2$;
$R_\xi$ is given in (\ref{eq:R-xixi}). These expressions
along with the moment bound of Theorem~\ref{t_uniform} lead to the formulas
for $T_{1,\ve}$ and $T_{2,\ve}$ given in the beginning of the proof.

%
\subsection{\texorpdfstring{Proof of Theorem \protect\ref{t_random_case}}{Proof of Theorem 5}}
\label{sec:proof-th-5}
First, we specify the constants appearing in the statement of the theorem.
Put $\alpha_*:=\alpha_1^{-1}\alpha_2^{-1/2}$ where $\alpha_1$ and
$\alpha_2$
appear in Assumption~\hyperlink{assW2}{(W2)}; then
%
%
\begin{equation}\label{eq:vartheta1-2}
\vartheta_1:= [148 \alpha_*^4]^{-1} ,\qquad
\vartheta_2:= 5\sqrt{2} c_1(s/2) \mathrm{f}_\infty\alpha_*^2 C_{s/2}.
\end{equation}
Define also
%
%
\begin{eqnarray}\label{eq:vartheta}
k_*&:=&8\alpha_*^2 c_1(s) [C_s\vee C_{s/2}\vee1],\nonumber\\[-8pt]\\[-8pt]
L^{(\ve)}_*(\beta) &:=& \sum_{k=1}^\infty
\exp\{ 2^{1+k\beta/m} (k_*^{-1}\epsilon)^{-\beta/m} -
(9/16) 2^{k} k^{-2}\},\nonumber
\end{eqnarray}
and note that $L^{(\ve)}_*(\beta)<\infty$ because $\beta<m$.
If we set
$I_\ve(q):= 2^{q(\ve+1)}[2^{q\epsilon}-1]^{-1}\Gamma(q+1)+1$,
then
the constants
$T_{5,\ve}$ and $T_{6,\ve}$ appearing in the statement of the theorem
are given by
%
%
\begin{eqnarray}
\label{eq:T-5}
T_{5,\ve}&:=&
I_\ve(q)
(2u_\ve k_*\bar{\mathrm{w}}_2)^q
N_{\bZ,\mathrm{d}}([k_*^{-1}\ve/8]^{1/m})\nonumber\\[-8pt]\\[-8pt]
&&{}\times
\log_2 \biggl(\frac{k_*\bar{\mathrm{w}}_2}{\underline{\mathrm{w}}_2}\biggr)
\bigl[1+L^{(\ve)}_*(\beta) \exp\{2C_\bZ(\beta)\}\bigr],\nonumber
\\
\label{eq:T-6}
T_{6,\ve}&:=&[c_1(s)+2]^q(\alpha_*\bar{\mathrm{w}}_2)^q
N_{\bZ,\mathrm{d}}([k_*^{-1}\ve/8]^{1/m})\nonumber\\[-8pt]\\[-8pt]
&&{}\times\log_2 \biggl(\frac{k_*\bar{\mathrm{w}}_2}{\underline{\mathrm{w}}_2}\biggr)
\bigl[1+L^{(\ve)}_*(\beta) \exp\{2C_\bZ(\beta)\}\bigr].\nonumber
\end{eqnarray}

The proof is based on application of Theorem \ref{t_random} and
Corollary \ref{cor:th_random}.
These results will be utilized with a distance $\mathrm{d}_*$ on
$\mathfrak{Z}$ which is related
to the original distance $\mathrm{d}$, and specified below.
In order to apply Theorem \ref{t_random}, we need to verify its conditions
and to compute the quantities $\Lambda_{A_\xi}$,
$\Lambda_{B_\xi}$, $\lambda_{\tilde{A}}$, $\lambda_{\tilde{B}}$
and~$y_\gamma$. These computations are routine and tedious.

We break the proof into steps.

0$^0$. \textit{Auxiliary results.}
We begin with preliminary results that will be used in the subsequent proof.
\begin{lemma}
\label{lem:tech_new1}
Let (\ref{eq:W-V}) hold and
Assumptions \textup{\hyperlink{assW2}{(W2)}} and \textup{\hyperlink{assW3}{(W3)}} be
satisfied; then
for all $w\in\cW$ and
$1\leq p < q\leq\infty$
one has
\[
[n^{1/q}M_q(w)]\leq\alpha^{-1}_1\alpha_2^{-1/p}\mu^{1/q-1/p}[n^{1/p}M_p(w)].
\]
\end{lemma}

\begin{pf}
Recall that under (\ref{eq:W-V}), $M_p(w)=\|w\|_p$ for all $p\geq1$.
In view of Assumption \hyperlink{assW2}{(W2)} for any $w\in\cV$, we have
\[
\alpha_1\alpha_2^{1/p}\|w\|_\infty
[\operatorname{mes}\{\operatorname{supp}(w)\} ]^{1/p} \leq\|w\|_p \leq\|
w\|_\infty
[\operatorname{mes}\{\operatorname{supp}(w)\}]^{1/p}.
\]
This inequality, together with Assumption \hyperlink{assW3}{(W3)}, yields
\begin{eqnarray*}
n^{1/q}\|w\|_q &\leq& n^{1/q}\|w\|_\infty
[\operatorname{mes}\{\operatorname{supp}(w)\}]^{1/q}
\\
&=& \frac{n^{1/p}
\|w\|_\infty[\operatorname{mes}\{\operatorname{supp}(w)\}
]^{1/p}}
{[n \operatorname{mes}\{\operatorname{supp}(w)\}]^{1/p-1/q}} \\
&\leq&
\alpha^{-1}_1\alpha_2^{-1/p}\mu^{1/q-1/p}[n^{1/p}\|w\|_p].
\end{eqnarray*}
\upqed\end{pf}

Our next lemma demonstrates that there exists
a real number $m_p\in(0,1]$ such that
(\ref{eq:before-t-random-case}) holds.
\begin{lemma}\label{lem:before_th-random-case}
Let Assumptions \ref{assW} and \ref{assL} hold;
then for any $p\geq2$, the inequality (\ref{eq:before-t-random-case})
is valid
with
$m_p=2/p$ and $C_p=(2\alpha_*)^{1-2/p}\mu^{1/p-1/2}$, that is,
\[
\sup_{b\in[\underline{\mathrm{w}}_{2},\bar{\mathrm{w}}_{2}]}
b^{-1}\sup_{\zeta_1,\zeta_2\in\bZ_2(b)}
\frac{n^{1/p}\|\phi[\zeta_1]-\phi[\zeta_2]\|_p}{
[\mathrm{d}(\zeta_1,\zeta_2)]^{2/p}}
\leq(2\alpha_*)^{1-2/p}\mu^{1/p-1/2}.
\]
\end{lemma}
\begin{pf}
We obviously have for any $p>2$
\[
\|\phi[\zeta_1]-\phi[\zeta_2]\|_p\leq(\|\phi[\zeta_1]\|_\infty+\|
\phi[\zeta_2]\|_\infty
)^{1-2/p}
(\|\phi[\zeta_1]-\phi[\zeta_2]\|_2)^{2/p}.
\]
Applying Lemma \ref{lem:tech_new1} with $q=\infty$ and $p=2$,
we have that
$\sup_{\zeta\in\bZ_2(b)}\|\phi[\zeta]\|_\infty\leq b\alpha_*\mu^{-1/2}$
for all $b\in[\underline{\mathrm{w}}_{2},\bar{\mathrm{w}}_{2}]$.
Then in view of Assumption \ref{assL}
\[
\sup_{\zeta_1,\zeta_1\in\bZ_2(b)}\frac{n^{1/p}
\|\phi[\zeta_1]-\phi[\zeta_2]\|_p}
{[\mathrm{d}(\zeta_1,\zeta_2)]^{2/p}}\leq
b(2\alpha_*)^{1-2/p}\mu^{1/m-1/2}\qquad \forall b\in
[\underline{\mathrm{w}}_2, \bar{\mathrm{w}}_2],
\]
as claimed.
\end{pf}
\begin{lemma}\label{lem:phi-norms}
Let Assumptions \ref{assW} and \ref{assL} hold; then
for any $\zeta\in\bZ_a$
%
%
\begin{eqnarray}
\label{eq:000}
\sqrt{n} \|\phi[\zeta]\|_p &\leq&
\mu_*^{1/p - 1/2}a\qquad \forall p\in[1,2),
\\
\label{eq:001}
n^{1/p} \|\phi[\zeta]\|_p &\leq&
\alpha_*\mu^{1/p-1/2} a\qquad \forall p>2,
\\
\label{eq:002}
\sqrt{n} \|\phi^2[\zeta]\|_p
&\leq&\alpha_*\mu^{-1/2}\mu_*^{1/p-1/2}a^2\qquad
\forall p\in[1,2).
\end{eqnarray}
\end{lemma}
\begin{pf}
By the H\"older inequality
$\|\phi[\zeta]\|_p \leq\mu_*^{1/p-1/2} \|\phi[\zeta]\|_2$; then (\ref{eq:000})
holds by definition of $\bZ_a$.
Inequality (\ref{eq:001}) follows
Lemma \ref{lem:tech_new1}.
In order to prove (\ref{eq:002}), we write
$\|\phi^2[\zeta]\|_p \leq\|\phi[\zeta]\|_\infty\|\phi[\zeta]\|_p$,
note that
by Lemma \ref{lem:tech_new1}
$\|\phi[\zeta]\|_\infty\leq\alpha_* \mu^{-1/2} \sqrt{n}\|\phi[\zeta]\|_2$
and use (\ref{eq:001}).
\end{pf}

1$^0$. \textit{Notation.} Now we establish some notation.
Recall that
$U_\xi(w,f)=c_1(s)[\sqrt{n} \Sigma_s(w,f) + 2n^{1/s}M_s(w)]$ and
$\bar{U}_\xi(w,f)$ is given by (\ref{eq:U-bar}).
It follows from the definition of $\bar{U}_\xi(\cdot,f)$, (\ref{eq:U-xi-f})
and (\ref{eq:sigg})
that
$\bar{U}_\xi(w,f)\geq\sqrt{n}\|w\|_2$ and
\[
\bar{U}_\xi(w,f) \leq c_1(s)
\bigl[
\mathrm{f}_\infty^{1/2-1/s}\sqrt{n}\|w\|_2 + 2n^{1/s}
\|w\|_s\bigr]
\leq c_1(s) \alpha_*[\mathrm{f}_\infty^{1/2}+2]\sqrt{n}\|w\|_2,
\]
where the last inequality is a consequence of Lemma \ref{lem:tech_new1}.
Therefore, we put
\[
r_\xi= \sqrt{n} \underline{\mathrm{w}}_2,\qquad R_\xi=c_1(s) \alpha_*[\mathrm
{f}_\infty^{1/2}+2] \sqrt{n}\bar{\mathrm{w}}_2,
\]
where $\underline{\mathrm{w}}_p$ and $\bar{\mathrm{w}}_p$ are defined
in (\ref{eq:w-p}). Recall also that
$\bZ_a=\{\zeta\dvtx a/2< \bar{U}_\xi(\phi[\zeta],f)\leq a\}$.
By definition of $\bar{U}_\xi(w,f)$ and by
the fact that $M_p(w)=\|w\|_p$ for all $p\geq1$,
we have that
$\bZ_a\subseteq\bZ_2(a)$ for all $a\in[r_\xi, R_\xi]$;
see (\ref{eq:Z-ss}). Define the distance
%
%
\begin{equation}
\label{eq:distance}
\mathrm{d}_*(\zeta_1,\zeta_2)= k_*\times
\cases{
\mathrm{d}(\zeta_1,\zeta_2) \vee[\mathrm{d}(\zeta_1,\zeta_2)]^{m_s},
&\quad
$s\in[1,4)$,\cr
\mathrm{d}(\zeta_1,\zeta_2) \vee[\mathrm{d}(\zeta_1,\zeta_2)]^{m_s}\cr
\qquad\vee\,
[\mathrm{d}(\zeta_1,\zeta_2)]^{m_{s/2}}, &\quad$s>4$,}
\end{equation}
where $k_*$ is given in (\ref{eq:vartheta}).
Note that $\mathrm{d}_*(\cdot,\cdot)$ is indeed a distance
because by definition $m_p\leq1$ for all $p\geq2$.

2$^0$. \textit{Verification of condition} (\ref{eq:condition-kappa}).
It follows from definition of $\bar{U}_\xi(\cdot,f)$, (\ref{eq:U-xi-f})
and (\ref{eq:sigg})
that
%
%
\begin{eqnarray}
\label{eq:new4}
\sqrt{n}\|\phi[\zeta]\|_2 &\leq&\bar{U}_\xi(\phi[\zeta],f) \nonumber\\[-8pt]\\[-8pt]
&\leq& c_1(s)\bigl[
\mathrm{f}_\infty^{1/2-1/s}\sqrt{n}\|\phi[\zeta]\|_2 + 2n^{1/s}
\|\phi[\zeta]\|_s\bigr].\nonumber
\end{eqnarray}
Therefore, by (\ref{eq:new4}), Assumption \ref{assL}
and (\ref{eq:before-t-random-case}) for any $\zeta_1,\zeta_2\in\bZ_a$
\begin{eqnarray*}
&&\bar{U}_\xi(\phi[\zeta_1]-\phi[\zeta_2],f) \\
&&\qquad\leq c_1(s)
\bigl[\mathrm{f}_\infty^{1/2-1/s}\sqrt{n}\|\phi[\zeta_1]-\phi[\zeta_2]\|_2 +
2n^{1/s}
\|\phi[\zeta_1]-\phi[\zeta_2]\|_s\bigr]
\\
&&\qquad\leq c_1(s)[\mathrm{f}_\infty^{1/2} + 2C_s] a \{\mathrm{d}(\zeta
_1,\zeta_2)
\vee[\mathrm{d}(\zeta_1,\zeta_2)]^{m_s}\}.
\end{eqnarray*}
Thus
\[
\sup_{\zeta_1, \zeta_2}
\frac{\bar{U}_\xi(\phi[\zeta_1]-\phi[\zeta_2],f)}{\mathrm{d}_*(\zeta_1,
\zeta_2)} \leq a\qquad \forall a\in[r_\xi,R_\xi],
\]
and (\ref{eq:condition-kappa}) is valid,
because
$k_* \geq
c_1(s)[\mathrm{f}_\infty^{1/2} + 2C_s]$;
see (\ref{eq:vartheta}) and (\ref{eq:distance}).

3$^0$.
\textit{Computation of $\varkappa_{\tilde{U}}$ and verification of}
(\ref{eq:condition-kappa-tild}).

We start with bounds on $\sup_{\zeta\in\bZ_a}\tilde{U}(\phi^2[\zeta])$.
Recall that
%
%
\begin{equation}\label{eq:U-tilde}
\tilde{U}(\phi^{2}[\zeta])
=
\cases{
4n^{2/s}M_{s/2}(\phi^{2}[\zeta]), &\quad$s\in(2,4)$,
\cr
c_1(s/2)\bigl[ \mathrm{f}_{\infty}^{1/2} \sqrt{n}
M_2(\phi^{2}[\zeta])\cr
\hspace*{37.13pt}+\, 2n^{2/s}M_{s/2}(\phi^{2}[\zeta])\bigr],
&\quad$s\geq4$.}\vadjust{\goodbreak}
\end{equation}
By (\ref{eq:001}), for any $\zeta\in\bZ_a$,
%
%
\begin{eqnarray}\label{eq:0022}\hspace*{28pt}
n^{2/s}\|\phi^2[\zeta]\|_{s/2} &=& (n^{1/s}\|\phi[\zeta]\|_s)^2 \leq
\alpha_*^{2}\mu^{(2/s)-1}
n \|\phi[\zeta]\|_2^2 \nonumber\\[-8pt]\\[-8pt]
&\leq&\alpha_*^{2}\mu^{(2/s)-1}a^2\qquad \forall s>2,\nonumber
\\
\label{eq:0023}
\sqrt{n} \|\phi^2[\zeta]\|_2 &=& (n^{1/4} \|\phi[\zeta]\|_4)^2 \leq
\alpha_*^{2}\mu^{-1/2}n\|\phi[\zeta]\|_2^2 \leq\alpha_*^{2}\mu^{-1/2}a^2.
\end{eqnarray}
Substituting these bounds in the expression for $\tilde{U}(\phi^2[\zeta])$
and taking into account that $\mu\geq1$ in view of \hyperlink
{assW3}{(W3)} we obtain
for all $s>2$
%
%
\begin{equation}\label{eq:vartheta-2}\quad
\sup_{\zeta\in\bZ_a} \tilde{U}(\phi^2[\zeta])
\leq k_1 \mu^{{2}/({s\wedge4})-1} a^2,\qquad
k_1:=4\alpha_*^2 c_1(s/2) [\mathrm{f}_\infty^{1/2}+2].
\end{equation}

Now we establish bounds on $\tilde{U}(\phi^2[\zeta_1]-\phi^2[\zeta
_2])$, $\zeta_1,\zeta_2\in\bZ_a$.

(a) First, we consider the case $s\in(2,4)$.
By the H\"older and triangle inequalities, we have
\begin{eqnarray*}
&&n^{2/s}\|\phi^2[\zeta_1]-\phi^2[\zeta_2]\|_{s/2} \\
&&\qquad\leq
n^{2/s-1/2}\bigl[\|\phi[\zeta_1]\|_{2s/(4-s)} + \|\phi[\zeta_2]\|_{2s/(4-s)}\bigr]
\sqrt{n}\|\phi[\zeta_1]-\phi[\zeta_2]\|_2.
\end{eqnarray*}
Noting that $2s/(4-s)>2$ and
applying (\ref{eq:001}),
we have
\[
n^{2/s-1/2}\|\phi[\zeta]\|_{2s/(4-s)} \leq\alpha_*\mu^{2/s-1} \sqrt{n}
\|\phi[\zeta]\|_2 \leq\alpha_*\mu^{2/s-1} a\qquad
\forall\zeta\in\bZ_a.
\]
Then using Assumption \ref{assL} we get
%
%
\begin{equation}\label{eq:phi2-2<s<4}\qquad
n^{2/s}\|\phi^2[\zeta_1]-\phi^2[\zeta_2]\|_{s/2} \leq
2\alpha_*\mu^{2/s-1} a^2 \mathrm{d}(\zeta_1,\zeta_2)\qquad
\forall\zeta_1,\zeta_2\in\bZ_a.
\end{equation}
This along with (\ref{eq:U-tilde}) implies that
for $s\in(2,4)$
%
%
\begin{equation}\label{eq:U-tilde-s<4}
\tilde{U}(\phi^2[\zeta_1]-\phi^2[\zeta_2]) \leq8
\alpha_*
\mu^{2/s-1} a^2
\mathrm{d}(\zeta_1,\zeta_2)\qquad \forall\zeta_1,\zeta_2\in\bZ_a.
\end{equation}

(b) Now assume that $s\geq4$.
We have for $\zeta_1,\zeta_2\in\bZ_a$
%
%
\begin{eqnarray}\label{eq:phi2-phi2-2}\hspace*{38pt}
\sqrt{n}\|\phi^2[\zeta_1]-\phi^2[\zeta_2]\|_2 &\leq&
\bigl[\|\phi[\zeta_1]\|_\infty+\|\phi[\zeta_2]\|_\infty\bigr]
\sqrt{n}\|\phi[\zeta_1]-\phi[\zeta_2]\|_2
\nonumber\\[-8pt]\\[-8pt]
&\leq& 2\alpha_*\mu^{-1/2} a^2 \mathrm{d}(\zeta_1,\zeta_2),
\nonumber
\end{eqnarray}
where we used Assumption \ref{assL}, and (\ref{eq:001}) with $p=\infty$.
Furthermore, we have
for all $\zeta_1,\zeta_2\in\bZ$
%
%
\begin{eqnarray}\label{eq:phi2-phi2-s/2}
&&
n^{2/s} \|\phi^2[\zeta_1]-\phi^2[\zeta_2]\|_{s/2}\nonumber\\
&&\qquad\leq \bigl[\|\phi[\zeta_1]\|_\infty+ \|\phi[\zeta_2]\|_\infty\bigr]
n^{2/s} \|\phi[\zeta_1]-\phi[\zeta_2]\|_{s/2}
\\
&&\qquad\leq 2C_{s/2}\alpha_*\mu^{-1/2} a^2 \{\mathrm{d}(\zeta_1,\zeta_2)\}^{m_{s/2}},
\nonumber
\end{eqnarray}
where we have used
(\ref{eq:001}) with $p=\infty$ and the definition of $m_p$ [see
(\ref{eq:before-t-random-case})].
These inequalities lead to the following bound: for all $s>4$
\begin{eqnarray*}
\tilde{U}(\phi^2[\zeta_1]-\phi^2[\zeta_2])
&\leq&2\alpha_*
c_1(s/2) [\mathrm{f}^{1/2}_\infty+2C_{s/2}]\\
&&{}\times \mu^{-1/2} a^2
\{ \mathrm{d}(\zeta_1,\zeta_2) \vee[
\mathrm{d}(\zeta_1,\zeta_2)]^{m_{s/2}}\}.
\end{eqnarray*}
Combining this with (\ref{eq:U-tilde-s<4}),
we obtain that for all $s>2$
%
%
\begin{equation}\label{eq:U-tild-phi-phi}
\tilde{U}(\phi^2[\zeta_1]-\phi^2[\zeta_2])
\leq k_2
\mu^{{2}/({s\wedge4})-1} a^2 \{ \mathrm{d}(\zeta_1,\zeta_2) \vee
\mathrm{d}^{m_{s/2}}(\zeta_1,\zeta_2)\},
\end{equation}
where $k_2:= 8\alpha_* c_1(s/2)[\mathrm{f}_\infty^{1/2}+2C_{s/2}]$.
Now using (\ref{eq:vartheta-2}) and (\ref{eq:U-tild-phi-phi}),
we obtain
\[
\varkappa_{\tilde{U}}(a) = \sup_{\zeta\in\bZ_a}
\frac{\tilde{U}(\phi^2[\zeta_1]-\phi^2[\zeta_2])}{\mathrm{d}_*(\zeta
_1,\zeta_2)}
\vee\sup_{\zeta\in\bZ_a} \tilde{U}(\phi^2[\zeta])
\leq\mu^{{2}/({s\wedge4})-1} a^2,
\]
and the last bound holds because $k_*\geq k_1\vee k_2$
[see (\ref{eq:vartheta})].
Thus, the condition~(\ref{eq:condition-kappa-tild}) is valid with
%
%
\begin{equation}\label{eq:gamma}
\gamma=\mu^{{1}/({s\wedge4})- {1}/{2}}.
\end{equation}
Note that condition of the theorem $\mu>[64c^2_1(s)]^{s\wedge
4/(s\wedge4-2)}$
ensures that $\gamma<[4c_1(1+\epsilon)]^{-1}$ for any $\epsilon\in
(0,1)$ as required
in Theorem \ref{t_random}.\vspace*{1pt}

4$^0$.
\textit{Bounding $\Lambda_{A_\xi}$ and $\Lambda_{B_\xi}$.}
By the formula for $A_\xi^2(w)$ given
immediately after
(\ref{eq:U-xi-f}), and by (\ref{eq:000}) and (\ref{eq:001}),
we have for $\zeta\in\bZ_a$
\begin{eqnarray*}
A_\xi^2(\phi[\zeta]) &\leq& 2c_1(s)\mathrm{f}_\infty^2
\bigl[ n \|\phi[\zeta]\|_{2s/(s+2)}^2+4\sqrt{n}\|\phi[\zeta]\|_2
\|\phi[\zeta]\|_s+8n^{1/s}\|\phi[\zeta]\|_s^2\bigr]
\\
&\leq& 2c_1(s)\mathrm{f}_\infty^2 a^2[\mu_*^{2/s}
+ 12\alpha_*^2 n^{-1/s}] \leq24 \alpha_*^2
c_1(s)\mathrm{f}_\infty^2 a^2 [\mu_*^{2/s}+n^{-1/s}].
\end{eqnarray*}
Here we have used that $\mu\geq1$, $\alpha_*\geq1$ and
we write $c_1(s)$ instead of $c_3(s)$
in the definition of $A^2(\cdot)$ because for functions $w(t,x)$
depending on $t-x$ only
the constant $c_2(s)$ equals one
[see (\ref{eq:c-*}) and remark after Lemma \ref{folland} in Section
\ref{sec:proof-fixed_w}]. Thus,
\[
\sup_{\zeta\in\bZ_a} A_\xi(\phi[\zeta]) \leq5\sqrt{c_1(s)}
\alpha_* \mathrm{f}_\infty
a \bigl[ \mu_*^{1/s} + n^{-1/(2s)}\bigr].
\]
In order to bound $A_\xi^2(\phi[\zeta_1]-\phi[\zeta_2])$, we note that
for all $\zeta_1,\zeta_2\in\bZ_a$:
\begin{itemize}
\item by the H\"older inequality and by Assumption \ref{assL},
$\sqrt{n}\|\phi[\zeta_1]\!-\!\phi[\zeta_2]\|_{2s/(s+2)}\leq a \mu_*^{1/s}
\mathrm{d}(\zeta_1,\zeta_2)$;
\item
by Assumption \ref{assL}, $\sqrt{n}\|\phi[\zeta_1]-\phi[\zeta_2]\|
_2\leq a
\mathrm{d}(\zeta_1,\zeta_2)$;
\item by (\ref{eq:before-t-random-case}),
$n^{1/s}\|\phi[\zeta_1]-\phi[\zeta_2]\|_s\leq C_s a [\mathrm{d}(\zeta_1,
\zeta_2)]^{m_s}$.
\end{itemize}
Therefore,
\[
\sup_{\zeta_1,\zeta_2\in\bZ_a}
\frac{A_\xi(\phi[\zeta_1]-\phi[\zeta_2])}{\mathrm{d_*}(\zeta_1,\zeta_2)}
\leq5\sqrt{c_1(s)}\mathrm{f}_\infty(C_s\vee1) a
\bigl[\mu_*^{1/s} + n^{-1/(2s)}\bigr]
\]
and
$
\Lambda_{A_\xi}\leq5\sqrt{c_1(s)}\alpha_* \mathrm{f}_\infty
(C_s\vee1)
[\mu_*^{1/s} + n^{-1/(2s)}]$.
Similarly, since $B_\xi(\phi[\zeta])=\frac{4}{3}c_1(s)\|\phi[\zeta]\|
_s$, we have
by (\ref{eq:001})
that
$\Lambda_{B_\xi} \leq\frac{4}{3}c_1(s) (C_s\vee1) \alpha_* n^{-1/s}$.
Thus, we have shown that
\[
\Lambda_{A_\xi} \leq k_3 \bigl[\mu_*^{1/s} + n^{-1/(2s)}\bigr],\qquad
\Lambda_{B_\xi} \leq k_3 n^{-1/s},\qquad
k_3:= 5c_1(s)\alpha_* \mathrm{f}_\infty(C_s\vee1).
\]
These bounds on $\Lambda_{A_\xi}$ and $\Lambda_{B_\xi}$ lead to the definition
of $C_\xi^*(y)$ in (\ref{eq:C-*-def}) [see also (\ref{eq:definitions-new-new})].
Note that $\vartheta_0$ in (\ref{eq:C-*-def}) satisfies $\vartheta_0=k_3$.

5$^0$. \textit{Computation of $\lambda_{\tilde{A}}$, $\lambda_{\tilde{B}}$
and $y_\gamma$.}

(i) First,\vspace*{1pt} consider the case $s\in(2,4)$. Recall that in this case
$\tilde{A}^2(\phi^2[\zeta])=37 n\|\phi^2[\zeta]\|^2_{s/2}=37 n \|\phi
[\zeta]\|_s^4$
and $\tilde{B}(\phi^2[\zeta])=0$.
Hence, by (\ref{eq:001})
\[
\sup_{\zeta\in\bZ_a} \tilde{A}(\phi^2[\zeta]) = \sup_{\zeta\in\bZ_a}
\sqrt{37n}\|\phi[\zeta]\|^2_s \leq\sqrt{37}\alpha_*^2\mu^{2/s-1}
n^{1/2-2/s} a^2.
\]
It follows from (\ref{eq:phi2-2<s<4}) that for any $\zeta_1,\zeta_2\in
\bZ_a$
\[
\sqrt{37n}\|\phi^2[\zeta_1]-\phi^2[\zeta_2]\|_{s/2} \leq
2\sqrt{37} \alpha_* \mu^{2/s-1} n^{1/2-2/s} a^2 \mathrm{d}(\zeta_1,\zeta_2).
\]
Combining these results, we obtain that
$\lambda_{\tilde{A}} \leq
2\sqrt{37} \alpha_*^2\mu^{2/s-1} n^{1/2-2/s}$ and $\lambda_{\tilde{B}}=0$
which, in turn, by (\ref{eq:y-gamma}) and (\ref{eq:gamma}) implies that
\[
y_\gamma=\gamma^4 \lambda^{-2}_{\tilde{A}} \geq
\bigl(2\sqrt{37}\alpha_*^2\bigr)^{-2} n^{4/s-1}=:y_*.
\]
This explains the definition of the
constant $\vartheta_1$ in (\ref{eq:vartheta1-2}).

(ii) Now let $s\geq4$; here recall that
\begin{eqnarray*}
\tilde{A}^2(\phi^2[\zeta])&=&
2c_1(s/2)\mathrm{f}_\infty^2 \bigl[n\|\phi^2[\zeta]\|_{2s/(s+4)}^2
+ 4\sqrt{n} \|\phi^2[\zeta]\|_2\|\phi^2[\zeta]\|_{s/2}\\
&&\hspace*{169.19pt}{} + 8 n^{2/s}\|\phi^2[\zeta]\|_{s/2}^2\bigr].
\end{eqnarray*}
Observing that for $\zeta\in\bZ_a$:
\begin{enumerate}[(a)]
\item[(a)] $\sqrt{n}\|\phi^2[\zeta]\|_{2s/(s+4)} \leq\alpha_*\mu_*^{2/s}a^2$
by (\ref{eq:002}) and $\mu\geq1$;
\item[(b)] $n^{1/s}\|\phi^2[\zeta]\|_{s/2}\leq\alpha_*^2 n^{-1/s} a^2$
by (\ref{eq:0022});
\item[(c)] $\sqrt{n}\|\phi^2[\zeta]\|_2\|\phi^2[\zeta]\|_{s/2}
\leq\alpha_*^2 \mu^{-1/2}n^{-2/s} a^4$
by (\ref{eq:0023}) and (b),
\end{enumerate}
we obtain
%
%
\begin{equation}\label{eq:A-t-1}
\sup_{\zeta\in\bZ_a} \tilde{A}(\phi[\zeta]) \leq
5\sqrt{c_1(s/2)} \mathrm{f}_\infty\alpha_* a [\mu_*^{2/s} + n^{-1/s}].
\end{equation}
Similarly, for $\zeta_1,\zeta_2\in\bZ_a$ we have
\begin{eqnarray*}
\sqrt{n}\|\phi^2[\zeta_1]-\phi^2[\zeta_2]\|_{2s/(s+4)} &\leq&\mu_*^{2/s}
\sqrt{n}\|\phi^2[\zeta_1]-\phi^2[\zeta_2]\|_2 \\
&\leq&
2\alpha_*\mu_*^{2/s} a^2 \mathrm{d}(\zeta_1,\zeta_2),
\\
n^{1/s}\|\phi^2[\zeta_1]-\phi^2[\zeta_2]\|_{s/2} &\leq& 2C_{s/2}\alpha_*
n^{-1/s} a^2\{\mathrm{d}(\zeta_1,\zeta_2)\}^{m_{s/2}},
\end{eqnarray*}
\begin{eqnarray*}
&&\bigl[\sqrt{n} \|\phi^2[\zeta_1]-\phi^2[\zeta_2]\|_2
\|\phi^2[\zeta_1]-\phi^2[\zeta_2]\|_{s/2}\bigr]^{1/2} \\
&&\qquad\leq
2\sqrt{C_{s/2}}\alpha_* n^{-1/s} a^2 \{\mathrm{d}(\zeta_1,\zeta_2)
\vee[\mathrm{d}(\zeta_1,\zeta_2)]^{m_{s/2}}\},
\end{eqnarray*}
where the first line follows from
the H\"older inequality and (\ref{eq:phi2-phi2-2}); the second one follows
from
(\ref{eq:phi2-phi2-s/2}); and the third line follows from the two previous
inequalities.
This yields
\[
\sup_{\zeta_1,\zeta_2\in\bZ_a}
\frac{\tilde{A}(\phi^2[\zeta_1]-\phi^2[\zeta_2])}{\mathrm{d}_*(\zeta
_1,\zeta_2)}
\leq5\mathrm{f}_\infty\sqrt{2 c_1(s/2)} \alpha_*C_{s/2}
a^2 [\mu_*^{2/s} + n^{-1/s}].
\]
Combining the last inequality with (\ref{eq:A-t-1}), we obtain
\[
\lambda_{\tilde{A}} \leq k_4 [\mu_*^{2/s}+n^{-1/s}],\qquad
k_4:= 5\mathrm{f}_\infty\sqrt{2c_1(s/2)} \alpha_* C_{s/2}.
\]

Now in order to bound $\lambda_{\tilde{B}}$ we recall
that $\tilde{B}(\phi^2[\zeta])=\frac{4}{3}c_1(s/2)\|\phi^2[\zeta]\|_{s/2}$.
Then
(\ref{eq:0022}) gives
$\sup_{\zeta\in\bZ_a}
\tilde{B}(\phi^2[\zeta])\leq\frac{4}{3}c_1(s/2)\alpha_*^2
n^{-2/s}a^2$. This along\break with~(\ref{eq:phi2-phi2-s/2}) leads to
\[
\lambda_{\tilde{B}} \leq k_5 n^{-2/s},\qquad
k_5:= \tfrac{8}{3} c_1(s/2)\alpha_*^2 C_{s/2}.
\]
Combining these results with (\ref{eq:y-gamma}) and taking into account
that, by (\ref{eq:gamma}), $\gamma=\mu^{-1/4}\leq1$ for $s\geq4$, we have
\[
\mu^{-1/4}=\sqrt{y_\gamma} \lambda_{\tilde{A}} + y_\gamma\lambda_{\tilde{B}}
\leq
\bigl[\sqrt{y_\gamma}+y_\gamma\bigr] (k_4\vee k_5)[\mu_*^{2/s}+n^{-1/s}],
\]
and an elementary calculation shows that
\[
y_\gamma\geq\mu^{-1/2} (k_4\vee k_5)^{-2}
[\mu_*^{2/s}+n^{-1/s}]^{-2}=:y_*.
\]
This inequality yields the constant $\vartheta_2$ appearing
in (\ref{eq:vartheta1-2}).

6$^0$. \textit{Application of Theorem} \ref{t_random}.
In order to apply Theorem \ref{t_random} with the distance
$\mathrm{d}_*(\cdot,\cdot)$ given in (\ref{eq:distance}), we need to
compute the quantity
\[
L^{(\ve)}_{\exp}=\sum_{k=1}^{\infty}\exp\{2\cE_{\bZ,\mathrm{d}_*}
(\ve2^{-k})-(9/16) 2^{k} k^{-2}\}.
\]
Note that the entropy number $\cE_{\bZ,\mathrm{d}_*}(\cdot)=\ln\{N_{\bZ
,\mathrm{d}_*}(\cdot)\}$ is computed with respect
to the distance $\mathrm{d}_*$.
Therefore, we first express
the entropy $\cE_{\bZ,\mathrm{d}_*}(\cdot)$ in terms of
the original distance $\mathrm{d}$
and then, using Assumption \hyperlink{assW4}{(W4)}, we derive a bound for
$L^{(\ve)}_{\exp}$.

By the definition of the distance $\mathrm{d}_*$,
for all $\delta\in(0,1)$ and $\zeta_1,\zeta_2\in\bZ$,
\[
\mathrm{d}(\zeta_1,\zeta_2) \leq[k_*^{-1}\delta]^{1/m}
\quad\Rightarrow\quad
\mathrm{d}_*(\zeta_1,\zeta_2)\leq\delta,
\]
where $m:=1\wedge m_{s}$ if $s\in(2,4)$ and $m:=1\wedge m_s\wedge m_{s/2}$
if $s\geq4$.
Therefore,
$N_{\bZ,\mathrm{d}_*}(\delta) \leq N_{\bZ,\mathrm{d}}
([k_*^{-1}\delta]^{1/m})$.
In view of Assumption \hyperlink{assW4}{(W4)}, this yields
\begin{eqnarray*}
\sup_{\delta\in(0,1)}\{\cE_{\bZ,\mathrm{d}_*}(\delta)- [k_*^{-1}
\delta]^{-\beta/m}\}
&\leq& \sup_{\delta\in(0,1)}
\{\cE_{\bZ,\mathrm{d}}
([k_*^{-1}\delta]^{1/m})
-[k_*^{-1}\delta]^{-\beta/m}
\}\\
&\leq&\sup_{x\in(0,1)}\{\cE_{\bZ,\mathrm{d}}(x)
-x^{-\beta}\}=C_{\bZ}(\beta).
\end{eqnarray*}
Thus, we obtain that
\begin{eqnarray*}
L^{(\ve)}_{\exp} &\leq&
\exp\{ 2C_{\bZ}(\beta)\}
\sum_{k=1}^\infty
\exp\{ 2^{1+k\beta/m} (k_*^{-1}\epsilon)^{-\beta/m} -
(9/16) 2^{k} k^{-2}\}\\
&=&\exp\{ 2C_{\bZ}(\beta)\}
L_*^{(\ve)}(\beta).
\end{eqnarray*}

Now the result of the theorem follows from
the bounds of Theorem \ref{t_random}.
The constants $T_{5,\ve}$ and $T_{6,\ve}$ given in the
beginning of the proof are obtained from the expressions for
$T_{1,\ve}$ and $T_{2,\ve}$ and bounds of Theorem \ref{t_random}.
In particular, we used that in view of Lemma \ref{lem:tech_new1}
$\sqrt{n}\sup_{w\in\cV_0}\|w\|_s \leq n^{({s-2})/({2s})} \alpha_*\mu^{1/s-1/2}
\bar{\mathrm{w}}_2$.

%
\section{\texorpdfstring{Proof of Theorem \protect\ref{t_example_1}}{Proof of Theorem 7}}
\label{proof:theorem7}
The proof is based on verification of conditions
and application of Theorem \ref{t_random_case}.
First, we establish auxiliary results
that provide the basis for verification of Assumptions \ref{assL} and
\ref{assW}.
Then, based on these results, we show that all conditions of
Theorem \ref{t_random_case} are fulfilled.
This will yield the required result.

Let $K$ and $K^\prime$ be any functions satisfying Assumptions
\hyperlink{assK1}{(K1)}
and \hyperlink{assK2}{(K2)},
and let $h,h^\prime$ be given vectors from $\cH$. Let
$\zeta=(K,h)$, $\zeta^{\prime}=(K^{\prime},h^{\prime})$,
and recall that $\phi_1[\zeta]$ is the mapping $(K,h)\mapsto n^{-1}K_h$.

Similarly,
if $K,Q,K^\prime,Q^\prime$
are any functions satisfying Assumptions \hyperlink{assK1}{(K1)} and
\hyperlink{assK2}{(K2)}, and if
$h,h^\prime,\mh,\mh^\prime$ are vectors from $\cH$ then
$z=[(K,h),(Q,\mh)]$,
$z^{\prime}=[(K^{\prime},h^{\prime}),
(Q^{\prime},\mh^{\prime})]$, and
$\phi_2[z]$ is the mapping
$[(K,h),(Q,\mh)] \mapsto n^{-1}(K_h\ast Q_{\mh})$.

1$^0$. \textit{Auxiliary results.}
We begin with
auxiliary results about properties of the mappings
$\phi_1[\zeta]$ and $\phi_2[z]$. The proofs of these results
are given in the \hyperref[app]{Appendix}.

Define
the function
%
%
\begin{equation}\label{eq:D}
D(x):= e^{dx}\bigl[x+
\tfrac{1}{2}L_\cK\sqrt{d}
(e^x-1)+\mathrm{k}_\infty(e^{dx}-1)\bigr],\qquad
x\geq0,
\end{equation}
and put
%
%
\begin{equation}\label{eq:theta-1-2}
\theta_1:=[\mathrm{k}_\infty/\mathrm{k}_1]D^{\prime}(2),\qquad \theta
_2:=2^{2d+2}\mathrm{k}^4_\infty\mathrm{k}^{-2}_1 D^{\prime}(4),
\end{equation}
where $D^{\prime}$ is the first derivative of the function $D$.

The next lemma states that Assumption \ref{assL} is fulfilled for
the mappings $\zeta\mapsto\phi_1[\zeta]$ and $z\mapsto\phi_2[z]$.
\begin{lemma}\label{lem:as-L}
Let Assumption \ref{assK} hold, and $s\geq1$. If the sets $\bZ^{(i)}$, $i=1,2$,
are equipped with the distances $\mathrm{d}^{(i)}_{\theta_i}(\cdot,\cdot
)$ then
Assumption \ref{assL} is valid for the mappings $\zeta\mapsto\phi
_1[\zeta]$
and $z\mapsto\phi_2[z]$.
\end{lemma}

The next three statements provide a basis for verification of
Assumption~\ref{assW}.
For any $h, h^\prime\in\cH$,
let
$h\vee h^\prime=(h_1\vee h_1^\prime, \ldots, h_d\vee h_d^\prime)$
and $h\wedge h^\prime=
(h_1\wedge h_1^\prime,\ldots, h_d\wedge h_d^\prime)$.
\begin{lemma}
\label{lem:tech1}
Let Assumptions \textup{\hyperlink{assK1}{(K1)}} and
\textup{\hyperlink{assK2}{(K2)}}
hold; then for any \mbox{$p\!\in\![1,\infty]$}
%
%
\begin{eqnarray}
\label{eq:phi1-p-norm}
\|\phi_1[\zeta]\|_p &=& n^{-1}
V_h^{-1+1/p}\|K\|_p\qquad \forall\zeta\in\bZ^{(1)},
\\
\label{eq:phi1-1}
\|\phi_1[\zeta]-\phi_1[\zeta^\prime]\|_p
&\leq&
n^{-1} (V_{h \vee h^\prime})^{-1+1/p}
D\bigl(\mathrm{d}^{(1)}_1(\zeta,\zeta^\prime)\bigr)\qquad \forall
\zeta,\zeta^\prime\in\bZ^{(1)},\hspace*{-25pt}
\\
\label{eq:phi2-2}
\|\phi_2[z]-\phi_2[z^\prime]\|_p
&\leq&
2n^{-1}\mathrm{k}_\infty
[(V_{h\vee
h^\prime})\vee(V_{\mh\vee\mh^\prime})]^{-1+1/p}\nonumber\\[-8pt]\\[-8pt]
&&{}\times D
\bigl(2\mathrm{d}^{(2)}_1(z,z^\prime)\bigr)\qquad
\forall z,z^\prime\in\bZ^{(2)}.\nonumber
\end{eqnarray}
\end{lemma}

Observe that Lemma \ref{lem:tech1} implies that Assumption \ref
{assumptionA2} of Section \ref{sec:uniform} is fulfilled for the mappings
$\zeta\mapsto\phi_1[\zeta]$ and $z\mapsto\phi_2[z]$.
\begin{lemma}
\label{lem:tech2}
Let $w\in\bH_d(1,P)$ with some $P>0$,
and let $\tilde{x}\in\bR^d$ be a point such that
$w(\tilde{x})=\|w\|_\infty>0$; then
\[
\biggl\{x\in\bR^{d}\dvtx|w(x)|\geq\frac{1}{2} \|w\|_\infty\biggr\}
\supseteq
\bigotimes_{i=1}^{d}\biggl[\tilde{x}_i-\frac{\|w\|_\infty}{2P\sqrt{d}},
\tilde{x}_i+
\frac{\|w\|_\infty}{2P\sqrt{d}}\biggr].
\]
\end{lemma}
\begin{lemma}
\label{lem:tech204}
Under Assumption \ref{assK} for any $p\geq1$:
\begin{eqnarray*}
&&\mbox{\hphantom{ii}\textup{(i)}\quad} \|\phi_2[z]\|_p\leq2^{d/p}
\mathrm{k}^{2}_{\infty} n^{-1}
(V_{h\vee\mh})^{-1+1/p},\\
&&\mbox{\hphantom{i}(\textup{ii})\quad} \|\phi_2[z]\|_p\geq2^{{d(1-p)}/{p}}
\mathrm{k}^{2}_{1}
n^{-1}(V_{h\vee\mh})^{-1+1/p},\\
&&\mbox{(\textup{iii})\quad} \operatorname{mes}\{\operatorname{supp}
(\phi_2[z])\}\geq(V_{h\vee\mh}) \biggl[\frac{\mathrm{k}^{2}_{1}}{2^{d+1}\sqrt
{d}L_\cK
\mathrm{k}_{\infty}}\biggr]^{d},
\\
&&\mbox{\hspace*{1.2pt}(\textup{iv})\quad} \operatorname{mes}
\biggl\{t\dvtx\phi_2[z](t)\geq\frac{1}{2}\|\phi_2[z]
\|_\infty\biggr\}\geq\biggl[\frac{\mathrm{k}^{2}_{1}}{2^{d+2}\sqrt{d}
L_\cK
\mathrm{k}_{\infty}}\biggr]^{d}
\operatorname{mes}\{\operatorname{supp}(\phi_2[z])\}.
\end{eqnarray*}
\end{lemma}

2$^0$. \textit{Verification of conditions of Theorem} \ref{t_random_case}.
We check Assumption \ref{assW} for the classes of weights
$\cW^{(1)}$ and $\cW^{(2)}$ given
by the parametrization $\phi_1[\zeta]$
and $\phi_2[z]$.

First, we note that \hyperlink{assW1}{(W1)} is fulfilled
both for $\phi_1[\zeta]$ and $\phi_2[z]$ in view of Assumption
\hyperlink{assK1}{(K1)}.
Furthermore, Assumptions \hyperlink{assK1}{(K1)} and \hyperlink{assK2}{(K2)}
together with Lem\-ma~\ref{lem:tech2}
imply \hyperlink{assW2}{(W2)} for $\phi_1[\zeta]$ with
%
%
\begin{equation}\label{eq:alpha-001}
\alpha_1=\frac{1}{2},\qquad \alpha_2=\alpha_{2,1}:=
\biggl[\frac{\mathrm{k}_1}{L_\cK\sqrt{d}}\biggr]^d,
\end{equation}
while the statement (iv) of Lemma \ref{lem:tech204}
yield \hyperlink{assW2}{(W2)} for $\phi_2[z]$ with the constants
%
%
\begin{equation}\label{eq:alpha-002}
\alpha_1 =\frac{1}{2},\qquad
\alpha_2=\alpha_{2,2}:=\biggl[\frac{\mathrm{k}^{2}_{1}}{2^{d+2}\sqrt{d}L_\cK
\mathrm{k}_\infty}\biggr]^d.
\end{equation}
Clearly, $\operatorname{mes}\{\operatorname{supp}(\phi_1[\zeta])\}\geq
V_{h^{\min}}$; hence
the condition
\[
nV_{h^{\min}} > [64 c^2_1(s)]^{({s\wedge4})/({s\wedge4-2})}\vadjust{\goodbreak}
\]
implies \hyperlink{assW3}{(W3)} for $\phi_1[\zeta]$ with $\mu
=nV_{h^{\min}}$.
It follows from the statement (iii) of Lemma \ref{lem:tech204}
that Assumption \hyperlink{assW3}{(W3)} holds
for $\phi_2[z]$ with $\mu=nV_{h^{\min}}$ if
\[
nV_{h^{\min}} > \alpha_{2,2}^{-1}
[64 c^2_1(s)]^{({s\wedge4})/({s\wedge4-2})}.
\]

Finally, a standard calculation shows that
if $\cE_{\cH}(\cdot)$ is the entropy number of the set
$\cH$ measured in the distance $\Delta_\cH$ [see (\ref{eq:distance-Delta})]
then for any $\delta\in(0,1]$
%
%
\begin{equation}\label{eq:entropy}
\cE_{\cH}(\delta) \leq d\ln(3/\delta)+
\sum_{i=1}^{d}
(\ln\ln[h^{\max}_i/h^{\min}_i])_+.
\end{equation}
This result together with \hyperlink{assK3}{(K3)} guarantees
that Assumption \hyperlink{assW4}{(W4)} is fulfilled for the both
parametrizations.

Now we compute the quantities $m_p$ and $C_p$ appearing in (\ref
{eq:before-t-random-case}). Although Lemma \ref{lem:before_th-random-case}
shows that we always can set $m_p=2/p$, it turns out that
under Assumption \ref{assK} we can put $m_p=1$ for all $p\geq2$
both for $\phi_1[\zeta]$
and for $\phi_2[z]$. This leads to weaker conditions on the entropy $\cE
_\cK(\cdot)$
(see formulation of Theorem \ref{t_random_case}).

First, consider the mapping $\phi_1[\zeta]$; here following
(\ref{eq:Z-ss}), we set
\[
\bZ_2^{(1)}(b):=
\{\zeta=(K,h)\dvtx n^{1/2} \|\phi_1[\zeta]\|_2 \leq b\}=
\{\zeta=(K,h)\dvtx(nV_h)^{-1/2} \|K\|_2\leq b\}
\]
for $b\in[\underline{\mathrm{w}}^{(1)}_2,\bar{\mathrm{w}}^{(1)}_2]$
where by (\ref{eq:w-p})
%
%
\begin{equation}\label{eq:w-21}
\underline{\mathrm{w}}^{(1)}_{2}
\geq\mathrm{k}_1(nV_{h^{\max}})^{-1/2},\qquad
\bar{\mathrm{w}}^{(1)}_{2} \leq\mathrm{k}_\infty
(nV_{h^{\min}})^{-1/2}.
\end{equation}
By (\ref{eq:phi1-1})
of Lemma \ref{lem:tech1}, we have for any $p\geq2$
and $\zeta_1=(K,h)$, $\zeta_2=(K^\prime, h^\prime)$
\[
n^{1/p} \|\phi_1[\zeta_1]-\phi_1[\zeta_2]\|_p \leq
(n V_{h\vee h^\prime})^{-1+1/p} D\bigl(\mathrm{d}_1^{(1)}(\zeta_1,\zeta_2)\bigr),
\]
and if $\zeta_1, \zeta_2 \in\bZ_2^{(1)}(b)$ are such that
$\mathrm{d}_1^{(1)}(\zeta_1,\zeta_2)\leq2$ then by definition of~$\bZ_2^{(1)}(b)$
\[
n^{1/p} \|\phi_1[\zeta_1]-\phi_1[\zeta_2]\|_p \leq
[b/\mathrm{k}_1]^{2-2/p} D^\prime(2) \mathrm{d}_1^{(1)}(\zeta_1,\zeta_2),
\]
where we have used that
$\|K\|_2\geq\|K\|_1\geq\mathrm{k}_1$ for all $K\in\cK$,
$D(0)=0$, and~$D$ is monotone
increasing. If
$\zeta_1,\zeta_2\in\bZ_2^{(1)}(b)$ and
$\mathrm{d}_1^{(1)}(\zeta_1,\zeta_2)>2$, then by the triangle
inequality,
and (\ref{eq:phi1-p-norm}) of Lemma \ref{lem:tech1}
\begin{eqnarray*}
n^{1/p}\|\phi_1[\zeta_1] - \phi_1[\zeta_2]\|_p
&\leq& n^{1/p}\|\phi_1[\zeta_1]\|_p + n^{1/p}\|\phi_1[\zeta_2]\|_p
\\
&\leq& 2 \mathrm{k}_\infty(nV_h)^{-1+1/p} \leq
\mathrm{k}_\infty[b/\mathrm{k}_1]^{2-2/p}
\mathrm{d}_1^{(1)}(\zeta_1,\zeta_2).
\end{eqnarray*}
These inequalities show that
if $\bZ^{(1)}$ is equipped with
the distance
$\mathrm{d}^{(1)}_{\theta_1}(\cdot,\cdot)$ [see (\ref{eq:theta-1-2}) for
definition of $\theta_1$]
then
(\ref{eq:before-t-random-case}) holds with
%
%
\begin{equation}\label{eq:m-p-C-p-1}
m_p=1,\qquad C_p=\theta_1^{-1}
[\mathrm{k}_\infty/\mathrm{k}_1]^{2-2/p}D^\prime(2)\leq1,
\end{equation}
because $nV_{h^{\min}}\geq1$ (which implies $b\leq\mathrm{k}_\infty$).

Now consider the mapping $\phi_2[z]$; following (\ref{eq:Z-ss})
we have here
\[
\bZ_2^{(2)}(b):=\{z=[(K,h),(Q,\mh)]\dvtx n^{1/2}\|\phi_2[z]\|_2\leq b\},\qquad
b\in\bigl[\underline{\mathrm{w}}_2^{(2)},
\bar{\mathrm{w}}_2^{(2)}\bigr],\vadjust{\goodbreak}
\]
where
by the statements (i) and (ii) of
Lemma \ref{lem:tech204}
%
%
\begin{equation}
\label{eq:w-22}
2^{-d/2} \mathrm{k}^{2}_{1}
(nV_{h^{\max}})^{-1/2} \leq
\underline{\mathrm{w}}^{(2)}_{2},\qquad
\bar{\mathrm{w}}^{(2)}_{2} \leq2^{d/2}
\mathrm{k}^{2}_{\infty}
(nV_{h^{\min}})^{-1/2}.
\end{equation}
Note that
if $z=[(K,h),(Q,\mh)]\in\bZ^{(2)}_2(b)$ then by the statement (ii)
of Lem\-ma~\ref{lem:tech204} we have
$(nV_{h\vee\mh})^{-1}\leq2^{d}
\mathrm{k}^{-4}_{1}b^2$. By this fact and
by (\ref{eq:phi2-2}) of Lemma~\ref{lem:tech1}, we have
for $z_1=[(K,h),(Q,\mh)], z_2=[(K^\prime,h^\prime), (Q^\prime,\mh^\prime
)]\in
\bZ_2^{(2)}$ such that $\mathrm{d}_1^{(2)}(z_1,\allowbreak z_2)\leq2$
\begin{eqnarray*}
n^{1/p} \|\phi_2[z_1]-\phi_2[z_2]\|_p
&\leq&
2\mathrm{k}_\infty
[(nV_{h\vee h^\prime})\vee(nV_{\mh\vee\mh^\prime})]^{-1+1/p}D
\bigl(2\mathrm{d}^{(2)}_1(z_1,z_2)\bigr)
\\
&\leq& 2^{d+2-d/p} \mathrm{k}_\infty[b/\mathrm{k}_1^{2}]^{2-2/p}
D^\prime(4) \mathrm{d}_1^{(2)}(z_1,z_2).
\end{eqnarray*}
If $\mathrm{d}_1^{(2)}(z_1,z_2)>2$, then
using the triangle inequality and Lemma \ref{lem:tech204}(i) we have
\[
n^{1/p} \|\phi_2[z_1]-\phi_2[z_2]\|_p \leq
2^{d+1}\mathrm{k}_\infty^2 [b/\mathrm{k}_1^2]^{2-2/p}
\leq2^{d}\mathrm{k}_\infty^2 [b/\mathrm{k}_1^2]^{2-2/p}
\mathrm{d}_1^{(2)}(z_1,z_2).
\]
Combining these inequalities, we observe that
if $\bZ^{(2)}$ is equipped with the distance
$\mathrm{d}_{\theta_2}^{(2)}(\cdot,\cdot)$ [see (\ref{eq:theta-1-2})]
then
(\ref{eq:before-t-random-case}) holds with
%
%
\begin{equation}\label{eq:m-p-C-p-2}
m_p=1,\qquad C_p=\theta_2^{-1} 2^{2d+2-{3d}/{p}} \mathrm{k}_\infty^{2}
[\mathrm{k}_\infty^2/\mathrm{k}_1]^{2-2/p}D^\prime(4) \leq1.
\end{equation}
We have used that $b \leq2^{d/2}\mathrm{k}_\infty^2$ because
$nV_{h^{\min}}\geq1$.
Thus (\ref{eq:m-p-C-p-1}) and (\ref{eq:m-p-C-p-2}) show that
$m=1$ and the condition $\beta<m$ of Theorem \ref{t_random_case}
holds if in Assumption \hyperlink{assK3}{(K3)} $\beta_{\cK}<1$.

3$^0$. \textit{Application of Theorem} \ref{t_random_case}.
First, note that $\vartheta_0^{(i)}$, $i=1,2$, defined in~(\ref{eq:vartheta0-1-2})
satisfy
\[
\vartheta_{0}^{(i)}:=5c_1(s) \mathrm{f}_\infty\alpha_{*,i} ,\qquad
\alpha_{*,i}:=2/\sqrt{\alpha_{2,i}},\qquad
i=1,2,
\]
where $\alpha_{2,i}$, $i=1,2$, are given in (\ref{eq:alpha-001})
and (\ref{eq:alpha-002}).
This is in accordance with
the definition of the constant $\vartheta_0$ in
(\ref{eq:C-*-def}) for the parametrizations $\phi_1[\zeta]$ and $\phi_2[z]$.
Then the definition of $C_{\xi,i}^*(y)$ in (\ref{eq:C-*-i})
corresponds to (\ref{eq:C-*-def}).
Following (\ref{eq:vartheta1-2}), we put
\[
\vartheta_{1}^{(i)}:= \alpha_{*,i}^{-4}/148,\qquad
\vartheta_{2}^{(i)}:=5\sqrt{2} c_1(s/2) \mathrm{f}_\infty
\alpha_{*,i}^2,\qquad i=1,2.
\]
Then the formula for $y_*^{(i)}$ appearing in the statement of the
theorem is a~version of (\ref{eq:y-**}).

Now we need to specify the constants $T_{5,\ve}$ and $T_{6,\ve}$;
see (\ref{eq:T-5}), (\ref{eq:T-6}).

Following (\ref{eq:vartheta}), we set
for $i=1,2$
\[
k_{*,i}:=8c_1(s) \alpha_{*,i}^2,\qquad
L_{*,i}^{(\ve)}(\beta):=\sum_{k=1}^\infty
\exp\{ 2^{1+k\beta} (k_{*,i}^{-1}\epsilon)^{-\beta} -
(9/16) 2^{k} k^{-2}\}.
\]

In view of (\ref{eq:entropy})
and Assumption \hyperlink{assK3}{(K3)}, we obtain for any $\beta\in
(\beta_\cK, 1)$
that
\begin{eqnarray*}
C_{\bZ^{(1)}}(\beta) &=& \sup_{\delta\in(0,1)}\bigl\{\cE_{\bZ^{(1)},
\mathrm{d}_{\theta_1}^{(1)}} (\delta)
-\delta^{-\beta}\bigr\}\\
&\leq& C_\cK+ C_{\beta,d}+
\sum_{i=1}^{d}(\ln\ln[h^{\max}_i/h^{\min}_i])_+ ,\\
C_{\bZ^{(2)}}(\beta)&=&\sup_{\delta\in(0,1)}
\bigl\{\cE_{\bZ^{(2)},\mathrm{d}_{\theta_2}^{(2)}}(\delta)
-\delta^{-\beta}\bigr\} \\
&\leq& 2C_\cK+2C_{\beta,d}+
2\sum_{i=1}^{d}
(\ln\ln[h^{\max}_i/h^{\min}_i])_+,
\end{eqnarray*}
where we have taken into account that $\theta_1\geq1$, $\theta_2\geq1$
and denoted
\[
C_{\beta,d}:=
\sup_{\delta\in(0,1]}[d\ln(3/\delta)+
\delta^{-\beta_\cK}-\delta^{-\beta}],\qquad \beta\in(\beta_\cK,1).
\]
Therefore for $i=1,2$
\[
L_{*,i}^{(\ve)}(\beta)\exp\{2 C_{\bZ^{(i)}}(\beta)\} \leq
[1+A_\cH]^i
\exp\{2iC_\cK\} \inf_{\beta\in(\beta_\cK,1)}
\bigl[L_{*,i}^{(\ve)}(\beta) \exp\{2iC_{\beta,d}\}\bigr],
\]
and, by Assumption \hyperlink{assK3}{(K3)}, (\ref{eq:entropy}) and (\ref
{eq:AH-BH})
\[
N_{\bZ^{(i)}, \mathrm{d}_{\theta_i}^{(i)}}(k_{*,i}^{-1}\ve/8)
\leq
[1+A_\cH]^{i}
[24k_{*,i}\theta_i/\ve]^{di}
\exp\biggl\{i\biggl(\frac{8k_{*,i}\theta_i}{\ve}\biggr)^{\beta_\cK}\biggr\}
\exp\{i C_{\cK}\}.
\]
Finally,
substituting these bounds
in (\ref{eq:T-5}) and
using (\ref{eq:w-21}), (\ref{eq:w-22}) and (\ref{eq:AH-BH})
we have that
\[
T^{(i)}_{5,\ve} \leq(1+A_\cH)^{2i}(1+B_\cH) \tilde{T}_{1}^{(i)},
\]
where
\begin{eqnarray*}
\tilde{T}_{1,\ve}^{(i)} &:=&
I_\ve(q) (2^{1+d/2}u_\ve k_{*,i}\mathrm{k}_\infty^2)^q
[24k_{*,i}\theta_i\ve^{-1}]^{di}
\exp\biggl\{i\biggl(\frac{8k_{*,i}\theta_i}{\ve}\biggr)^{\beta_\cK}\biggr\}
\exp\{3iC_\cK\}
\\
&&\hspace*{0pt}{} \times
\log_2\biggl(\frac{2^d\mathrm{k}_\infty^2 k_{*,i}}{\mathrm{k}_1^2}
\biggr)
\Bigl\{1+\inf_{\beta\in(\beta_\cK,1)}
\bigl[L_{*,i}^{(\ve)}(\beta) \exp\{2iC_{\beta,d}\}\bigr]\Bigr\},\qquad
i=1,2.
\end{eqnarray*}
This leads to the first statement of the theorem.
The second statement of the theorem follows
substitution of the above bounds in (\ref{eq:T-6}) which gives
$T_{6,\ve}^{(i)}\leq(1+A_\cH)^{2i}(1+B_\cH) \tilde{T}_{2}^{(i)}$, where
\begin{eqnarray*}
\tilde{T}_{2,\ve}^{(i)} &:=&
[c_1(s)+2]^q
(
2^{d/2}\alpha_{*,i}\mathrm{k}_\infty^2)^q
[24k_{*,i}\theta_i\ve^{-1}]^{di}
\exp\biggl\{i\biggl(\frac{8k_{*,i}\theta_i}{\ve}\biggr)^{\beta_\cK}\biggr\}
\exp\{3iC_\cK\}
\\
&&{} \times
\log_2\biggl(\frac{2^d\mathrm{k}_\infty^2 k_{*,i}}{\mathrm{k}_1^2}
\biggr)
\Bigl\{1+\inf_{\beta\in(\beta_\cK,1)}
\bigl[L_{*,i}^{(\ve)}(\beta) \exp\{2iC_{\beta,d}\}\bigr]\Bigr\},\qquad
i=1,2.
\end{eqnarray*}

%
\section{\texorpdfstring{Proofs of Theorems \protect\ref{th:eta_fixed_w} and \protect\ref{t:proc-eta-assW}}
{Proofs of Theorems 8 and 9}}
\label{sec:proofs-eta}
\subsection{\texorpdfstring{Proof of Theorem \protect\ref{th:eta_fixed_w}}{Proof of Theorem 8}}
Let $X'=(X,\e)$, and let $X'_i$, $i=1,\ldots, n$
be independent copies of $X'$.
For any $l>0$, $x'=(x,u)\in\cX\times\bR$ and $t\in\cT$ define
the function
\[
w^{(l)}(t,x')=w(t,x)u {\mathbf1}_{[-l,l]}(u).
\]
With this notation, we note that on the event
$\{{\max_{i=1,\ldots,n}}|\e_i|\leq l\}$
\[
\eta_w(t)= \sum_{i=1}^n w(t, X_i)\e_i= \sum_{i=1}^n w^{(l)}(t,
X_i^\prime)
=\xi_{w^{(l)}}(t),
\]
and the last equality holds because
$\bE w^{(l)}(t,X')=0$, for all $t\in\cT$ and $l>0$ because
the distribution of $\e$ is symmetric.
Therefore for any $z>0$,
\[
\bP\{\|\eta_w\|_{s,\tau} \geq z \}
\leq\bP\bigl\{\bigl\|\xi_{w^{(l)}}\bigr\|_{s,\tau}\geq z \bigr\} + n\bP\{|\e|> l \}.
\]
If Assumption \hyperlink{assE1}{(E1)} is fulfilled, then for any $z>0$
%
%
\begin{equation}
\label{b1}
\bP\{\|\eta_w\|_{s,\tau}\geq z \} \leq\bP\bigl\{\bigl\|\xi_{w^{(l)}}\bigr\|_{s,\tau
}\geq z \bigr\} +
n v\exp\{- b l^\alpha\}.
\end{equation}
If Assumption \hyperlink{assE2}{(E2)} is fulfilled, then for any $z>0$
%
%
\begin{equation}
\label{b2}
\bP\{\|\eta_w\|_{s,\tau}\geq z \} \leq\bP\bigl\{\bigl\|\xi_{w^{(l)}}\bigr\|_{s,\tau
}\geq z \bigr\}+ n P l^{-p}.
\end{equation}
In order to bound
the first term on the right-hand side of (\ref{b1}) and (\ref{b2}),
we repeat the steps in the proof
of Theorem \ref{fixed_w1} with $w$ replaced by $w^{(l)}$
and optimize with respect to the truncation level $l$.

For any $z>0$, we define
\[
\Upsilon_s(w,f,z)=
\frac{z^2}{({1}/{3})\varpi_s^2(w, f) + ({4}/{3})c_*(s) M_s(w) z} ,
\]
where $c_*(s)$ is given in (\ref{eq:c-*}).

First, consider the case $s\geq2$.
Using the same reasoning as in the proof of
Theorem \ref{fixed_w1}, we have the following upper bound:
for all $z>0$
%
%
\begin{equation}
\label{b3}
\bP\bigl\{\bigl\|\xi_{w^{(l)}}\bigr\|_{s,\tau}\geq\varrho_s(w,f)+z\bigr\}
\leq\exp\{-[1\vee l]^{-1}\Upsilon_s(w,f,z)\}.
\end{equation}
Under Assumption \hyperlink{assE1}{(E1)}, if we set
\[
l=\cases{
[b^{-1}\Upsilon_s(w,f,z)]^{{1}/{\alpha}}, &\quad
$b^{-1}\Upsilon_s(w,f,z)<1$,\cr
[b^{-1}\Upsilon_s(w,f,z)]^{{1}/({1+\alpha})}, &\quad$b^{-1}\Upsilon
_s(w,f,z)\geq
1$,}
\]
then it follows
from (\ref{b1}) and (\ref{b3}) that
\[
\bP\{\|\eta_w\|_{s,\tau}\geq\varrho_s(w,f)+z \}\leq G^{(1)}(\Upsilon_s(w,f,z)).
\]
Thus, the first statement of the theorem is proved if $s\geq2$.

If Assumption \hyperlink{assE2}{(E2)} is fulfilled then we choose
\[
l=\frac{\Upsilon_s(w,f,z)}{p\ln(1+p^{-1}\Upsilon_s(w,f,z))}
\]
and note that $l\geq1$ for any value of $\Upsilon_s(w,f,z)$.
Then
(\ref{b2}) and (\ref{b3})
imply that
\begin{eqnarray*}
\bP\{\|\eta_w\|_{s,\tau}\geq\varrho_s(w,f)+z \}&\leq&
\biggl[\frac{1}{(1+p^{-1}\Upsilon_s(w,f,z))}\biggr]^p\\
&&{}
+ nP \biggl[\frac{p\ln(1+p^{-1}\Upsilon_s(w,f,z))}{\Upsilon_s(w,f,z)}\biggr]^p.
\end{eqnarray*}
Using the trivial inequality $(1+u)^{-1}\leq u^{-1}\ln(1+u), u\geq0$
we get
\[
\bP\{\|\eta_w\|_{s,\tau}\geq\varrho_s(w,f)+z \}\leq
[1+nP]\biggl[
\frac{p\ln(1+p^{-1}\Upsilon_s(w,f,z))}{\Upsilon_s(w,f,z)}\biggr]^p
\]
and, therefore, the second statement of the theorem is proved for the
case $s\geq2$.

If $s<2$, then
we have similarly to (\ref{b3}) that
for all $z>0$
\[
\bP\bigl\{\bigl\|\xi_{w^{(l)}}\bigr\|_{s,\tau}\geq\varrho_s(w,f)+z\bigr\}\leq\exp\{
-[1\vee l]^{-2}\Upsilon_s(w,f,z)\}.
\]
The same computations as in the case $s\geq2$ lead to
the statement of the theorem when $s<2$.

%
\subsection{\texorpdfstring{Proof of Theorem \protect\ref{t:proc-eta-assW}}{Proof of Theorem 9}}
Put
\begin{eqnarray*}
L^{(\ve)}_{\alpha,b}&:=&
\sum_{k=1}^{\infty}\exp\{\ve^{-\beta} 2^{\beta k+1}\}
\sqrt{g_{\alpha,b}(9\cdot2^{k-3} k^{-2})},\\
J^{(\ve)}_{\alpha,b}&:=&q\int
_{1}^\infty(x-1)^{q-1}[g_{\alpha,b}(x)]^{{1}/{4}}\,
\rd x,\\
T_{n,\ve}&:=&(1+nv)[2^{2\ve}(1+\ve)\mathrm{a}\bar{\mathrm{w}}_2
]^q
[2^{q\ve}-1]^{-1}\exp\{C_\bZ(\beta)+(8/\ve)^{\beta}\}\\
&&{}\times
\bigl(1+\exp\{2C_\bZ(\beta)\}
L^{(\ve)}_{\alpha,b}\bigr)J^{(\ve)}_{\alpha,b}.
\end{eqnarray*}
We note
that $L^{(\ve)}_{\alpha,b}<\infty$ since
$\beta<\alpha/(2+\alpha)$ if $s<2$,
and $\beta<\alpha/(1+\alpha)$ if \mbox{$s\geq2$}.
Note also that
the quantity $J^{(\ve)}_{g}(\cdot)$ in the second inequality of
Corollary~\ref{cor:prop2_new} admits the following bound if $g=G_1$:
\[
J^{(\ve)}_{g}(z)\leq(1+nv)[g_{\alpha,b}(z)]^{1/4}
\bigl(1+L^{(\ve)}_{\alpha,b}\bigr)J^{(\ve)}_{\alpha,b},\qquad z>0.
\]
If for any $\zeta\in\bZ$, we let
\begin{eqnarray*}
U_{\eta}(\phi[\zeta])&=&\mathrm{a}
\sqrt{n}\|\phi[\zeta]\|_2,\qquad
A_{\eta}(\phi[\zeta])=\mathrm{b}_{n}\sqrt{n}
\|\phi[\zeta]\|_2,\\
B_\eta(\phi[\zeta])&=&\mathrm{c}_n\sqrt{n}\|\phi[\zeta]\|_2,
\end{eqnarray*}
then we have for $f\in\cF$
\begin{eqnarray*}
\varrho_s(\phi[\zeta],f)&\leq& U_{\eta}(\phi[\zeta]),\qquad
\tfrac{1}{3}\varpi_s^2(\phi[\zeta],f)\leq
A^2_{\eta}(\phi[\zeta]),\\
\tfrac{4}{3}c_*(s)M_s(\phi[\zeta])
&\leq& B_\eta(\phi[\zeta]).
\end{eqnarray*}
Thus, in view of Theorem \ref{th:eta_fixed_w},
Assumption \ref{fixed_theta} holds with
$U=U_\eta$, $A=A_\eta$, $B=B_\eta$ and $g=G^{(1)}$.
Then standard computations show that
$\Lambda_{A_\eta}= \mathrm{b}_{n}$ and
$\Lambda_{B_\eta}= \mathrm{c}_n$.
The assertion of the theorem follows now from
Corollary \ref{cor:prop2_new}.

\begin{appendix}\label{app}
\section*{Appendix}

\subsection*{\texorpdfstring{Proof of Lemma \protect\ref{l_countable_xi}}{Proof of Lemma 4}}
Let
\[
\cX^{(n)}={\underbrace{\cX\times\cdots\times\cX}_{n\mbox{-}\mathrm{times}}},\qquad
\bar{\cX}{}^{(n)}={\underbrace{\bar{\cX}\times\cdots\times
\bar{\cX}}_{n\mbox{-}\mathrm{times}}}.
\]
Obviously, $\bar{\cX}{}^{(n)}$ is a countable dense subset of
$\cX^{(n)}$.
For any $x^{(n)}\in\cX^{(n)}$ and $t\in\cT$,
put
\[
\xi\bigl(t,x^{(n)}\bigr)=\sum_{i=1}^{n}[w(t,x_i)-\bE w(t,X)],
\]
and let
\[
\mL=\biggl\{
l_{\bar{x}^{(n)}}\dvtx\cT\to\bR\dvtx l_{\bar{x}^{(n)}}(t)=
\frac{|\xi(t,\bar{x}^{(n)})|^{s-1}
\sign{[\xi(t,\bar{x}^{(n)})]}}
{\|\xi(\cdot,\bar{x}^{(n)})\|_{s,\tau}^{s-1}}, \bar{x}^{(n)}\in\bar{\cX}{}^{(n)}
\biggr\}.
\]
Note that $\mL$ is countable and $\mL\subset\bB_{{s}/({s-1})}$
since, obviously
\[
\|l_{\bar{x}^{(n)}}\|_{{s}/({s-1}),\tau}=1\qquad
\forall\bar{x}^{(n)}\in\bar{\cX}{}^{(n)}.
\]
Note that $\xi_w(\cdot)=\xi(\cdot,X^{(n)}), X^{(n)}=(X_1,\ldots, X_n)$,
and therefore, in order to prove the assertion of the lemma
it is sufficient to show that
%
%
\setcounter{equation}{0}
\begin{equation}
\label{eq:app1}
\bigl\|\xi\bigl(\cdot,x^{(n)}\bigr)\bigr\|_{s,\tau}=\sup_{l\in\mL}\int l(t)\xi
\bigl(t,x^{(n)}\bigr)\tau(\rd t)\qquad \forall x^{(n)}\in\cX^{(n)}.
\end{equation}
First, let us note that Assumption \ref{assumptionA1} implies that
for every
$\e>0$ and every $x^{(n)}\in\cX^{(n)}$
there exists $\bar{x}^{(n)}\in\bar{\cX}{}^{(n)}$ such that
%
%
\begin{equation}
\label{eq:app2}
\bigl\|\xi\bigl(\cdot,x^{(n)}\bigr)-\xi\bigl(\cdot,\bar{x}^{(n)}\bigr)
\bigr\|_{s,\tau}\leq\e.
\end{equation}
Taking into account that $\mL\subset\bB_{{s}/({s-1})}$ and using the
H\"older inequality, we obtain from (\ref{eq:app2}) that
%
%
\begin{equation}
\label{eq:app3}
\biggl|\sup_{l\in\mL}\int l(t)\xi\bigl(t,{x}^{(n)}\bigr)\tau(\rd t)-\sup_{l\in\mL}\int
l(t)\xi\bigl(t,\bar{x}^{(n)}\bigr)\tau(\rd t)\biggr|\leq\e.
\end{equation}
Obviously
\[
\bigl\|\xi\bigl(\cdot,\bar{x}^{(n)}\bigr)\bigr\|_{s,\tau}=\int l_{\bar{x}^{(n)}}(t)\xi
\bigl(t,\bar{x}^{(n)}\bigr)\tau(\rd t).\vadjust{\goodbreak}
\]
It implies in view of the duality argument that
%
%
\begin{equation}
\label{eq:app4}
\bigl\|\xi\bigl(\cdot,\bar{x}^{(n)}\bigr)\bigr\|_{s,\tau}=\sup_{l\in\mL}\int l(t)\xi\bigl(t,\bar
{x}^{(n)}\bigr)\tau(\rd t).
\end{equation}
Using the triangle inequality, we obtain from (\ref{eq:app2}), (\ref
{eq:app3}) and (\ref{eq:app4}) that
for every $\e>0$ and every $x^{(n)}\in\cX^{(n)}$
\[
\biggl|\bigl\|\xi\bigl(\cdot,x^{(n)}\bigr)\bigr\|_{s,\tau}-\sup_{l\in\mL}\int l(t)\xi
\bigl(t,{x}^{(n)}\bigr)\tau(\rd t)\biggr|\leq2\e,
\]
which
completes the proof of (\ref{eq:app1}) because
$\e>0$ can be chosen arbitrary small.

\subsection*{Proof of Lemma \protect\ref{lepski}}
First, note that for any $p\geq1$ and $x\in\cX$
\begin{eqnarray*}
\|\bar{w}(\cdot,x)\|_{p,\tau}&\leq&2^{1-{1}/{p}}\biggl[\int|w(t,x)|^p\tau
(\rd t)
+\int\bE|w(t,X)|^p\tau(\rd t)\biggr]^{{1}/{p}}\\
&\leq&2\sup_{x\in\cX}\|
w(\cdot,x)\|_{p,\tau}.
\end{eqnarray*}
Here, we have used the triangle inequality.
Next, for any $p\geq1$ and $t\in\cT$,
\begin{eqnarray*}
&&\biggl[\int|\bar{w}(t,x)|^p f(x) \nu(\rd x)\biggr]^{1/p} \\
&&\qquad:=
[\bE|\bar{w}(t,X)|^p]^{1/p}
\leq2[\bE|w(t,X)|^p]^{1/p}\\
&&\qquad=:2
\biggl[\int|w(t,x)|^p f(x) \nu(\rd x)\biggr]^{1/p} .
\end{eqnarray*}
Here, we used that $\bE|\eta-\bE\eta|^p\leq2^p \bE|\eta|^p$.
Combining both inequalities, we have
\[
M_{p}(\bar{w})\leq2 M_{p}(w),
\]
and the second statement of the lemma is proved.

%
%

\subsection*{Proof of Lemma \protect\ref{lem:as-L}}

1$^0$.
First, we establish statement of the lemma for the mapping
$\zeta\mapsto\phi_1[\zeta]$.
For any $s\geq1$ let $\mathfrak{s}:=s\wedge2$.
Following
(\ref{eq:w-p}) and~(\ref{eq:Z-ss})
and in view of (\ref{eq:phi1-p-norm}), we have
\begin{eqnarray*}
\underline{\mathrm{w}}^{(1)}_{\mathfrak{s}}
&\geq&\mathrm{k}_1(nV_{h^{\max}})^{1/\mathfrak{s} -1},\qquad
\bar{\mathrm{w}}^{(1)}_{\mathfrak{s}} \leq\mathrm{k}_\infty
(nV_{h^{\min}})^{1/\mathfrak{s} - 1},
\\
\bZ^{(1)}_\mathfrak{s}(b):\!&=&\bigl\{\zeta=(K,h)\in\bZ^{(1)}\dvtx
(nV_h)^{1/\mathfrak{s}-1}\|K\|_{\mathfrak{s}}\leq b\bigr\},\qquad
b\in\bigl[\underline{\mathrm{w}}^{(1)}_{\mathfrak{s}},
\bar{\mathrm{w}}^{(1)}_{\mathfrak{s}}\bigr].
\end{eqnarray*}
We note that if $\zeta=(K,h)\in\bZ^{(1)}_\mathfrak{s}(b)$ then
%
%
\begin{equation}
\label{eq:xxxx1}
(nV_h)^{1/\mathfrak{s}-1}\leq\mathrm{k}^{-1}_1 b.
\end{equation}

Let $\zeta_1,\zeta_2\in\bZ^{(1)}_\mathfrak{s}(b)$ be such that $\mathrm
{d}^{(1)}_1(\zeta_1,\zeta_2)\leq2$.
Applying (\ref{eq:phi1-1})
with $p=\mathfrak{s}$ and using (\ref{eq:xxxx1}), we get
%
%
\begin{equation}
\label{eq1:new-t8}
n^{1/\ms}\|\phi_1[\zeta_1]-\phi_1[\zeta_2]\|_{\ms}\leq\mathrm
{k}^{-1}_1 b D^{\prime}(2)\mathrm{d}^{(1)}_1(\zeta_1,\zeta_2) =
b \mathrm{d}^{(1)}_{\theta_1}(\zeta_1,\zeta_2).
\end{equation}
Here we have taken into account that
$D^{\prime}(2)={\sup_{x\in[0,2]}} |D^\prime(x)|$, where
the function $D(\cdot)$ is given in (\ref{eq:D}).
If $\zeta_1,\zeta_2\in\bZ^{(1)}_\ms(b)$ are such that $\mathrm
{d}^{(1)}_1(\zeta_1,\zeta_2)>2$, then by the triangle inequality
%
%
\begin{equation}
\label{eq2:new-t8}
n^{1/\ms}\|\phi_1[\zeta_1]-\phi_1[\zeta_2]\|_{\ms}\leq2b\leq b\mathrm
{d}^{(1)}_1(\zeta_1,\zeta_2)\leq b\mathrm{d}^{(1)}_{\theta_1}(\zeta
_1,\zeta_2).
\end{equation}
Thus,
(\ref{eq1:new-t8}) and (\ref{eq2:new-t8}) imply that
that Assumption \ref{assL} holds if $\bZ^{(1)}$ is equipped with the distance
$\mathrm{d}^{(1)}_{\theta_1}$, where we recall that
$\theta_1=\mathrm{k}_\infty\mathrm{k}^{-1}_1D^{\prime}(2)\geq1$
[see (\ref{eq:theta-1-2})].

2$^0$. Now we prove the statement of the lemma for the mapping
$z\mapsto\phi_2[z]$.
By the statements (i) and (ii) of
Lemma \ref{lem:tech204} applied with $p=\ms$ we have
\[
2^{{d(1-\ms)}/{\ms}}\mathrm{k}^{2}_{1}
(nV_{h^{\max}})^{1/\mathfrak{s}-1} \leq
\underline{\mathrm{w}}^{(2)}_{\ms},\qquad
\bar{\mathrm{w}}^{(2)}_{\ms} \leq2^{d/\ms}
\mathrm{k}^{2}_{\infty}
(nV_{h^{\min}})^{1/\mathfrak{s} -1}.
\]
Recall that
\[
\bZ^{(2)}_\ms(b):=\bigl\{z=[(K,h),(Q,\mh)]\in\bZ^{(2)}\dvtx n^{1/\ms}\|\phi
_2[z]\|_\ms\leq b\bigr\},\qquad
b\in\bigl[\underline{\mathrm{w}}^{(2)}_{\ms},\bar{\mathrm{w}}^{(2)}_{\ms}\bigr].
\]
If $z=[(K,h),(Q,\mh)]\in\bZ^{(2)}_\ms(b)$ then by the statement (ii)
of Lemma \ref{lem:tech204}
%
%
\begin{equation}
\label{eq0:new-t9}
(n V_{h\vee\mh})^{1/\ms-1} \leq2^{{d(\ms-1)}/{\ms}} \mathrm
{k}^{-2}_{1}b \leq2^{d/2}
\mathrm{k}^{-2}_{1}b.
\end{equation}

Let $z_1, z_2\in\bZ^{(2)}_\ms(b)$ be such that $\mathrm
{d}^{(2)}_1(z_1,z_2)\leq2$.
Applying (\ref{eq:phi2-2})
with $p=\ms$ and using (\ref{eq0:new-t9}), we obtain
%
%
\begin{equation}
\label{eq1:new-t9}
n^{1/\ms}\|\phi_2[z_1]-\phi_2[z_2]\|_{\ms}
\leq b 2^{2+d/2}\mathrm{k}_\infty\mathrm{k}^{-2}_1
D^{\prime}(4) \mathrm{d}^{(2)}_1(z_1,z_2).
\end{equation}
If $z_1,z_2\in\bZ^{(2)}_\ms(b)$ are such that
$\mathrm{d}^{(2)}_1(z_1,z_2)>2$, then we have by the triangle inequality
%
%
\begin{equation}
\label{eq2:new-t9}
n^{1/\ms}
\|
\phi_2[z_1]-\phi_2[z_2]\|_{\ms}\leq2b\leq b \mathrm{d}^{(2)}_1(z_1,z_2).
\end{equation}
Thus, (\ref{eq1:new-t9}) and (\ref{eq2:new-t9})
imply that Assumption \ref{assL} is valid provided
that~$\bZ^{(2)}$ is equipped with the distance
$\mathrm{d}^{(2)}_{\theta_2}(\cdot,\cdot)$, where
$\theta_2=2^{2d+2}\mathrm{k}_\infty^4 \mathrm{k}_1^{-2}
D^\prime(4)\geq\break
2^{({d+4})/{2}}\times\mathrm{k}_\infty
\mathrm{k}^{-2}_1 D^{\prime}(4)\geq1$
[see (\ref{eq:theta-1-2})].\vspace*{-3pt}

\subsection*{Proof of Lemma \protect\ref{lem:tech1}}

1$^0$.
Inequality (\ref{eq:phi1-p-norm}) is immediate.
We start with the proof of (\ref{eq:phi1-1}).

Since the required bound is symmetric in $h$
and $h^\prime$, without loss of generality
we will assume that $V_h\geq V_{h^\prime}$.
By the triangle inequality in view of Assumption~\hyperlink
{assK1}{(K1)}, we get
\begin{eqnarray}
\label{eq1:proof_lemma9}\quad
\|K_h- K^\prime_{h^\prime}\|_p &\leq&\|K_h- K^\prime_h\|_p +
\|K^\prime_h - K^\prime_{h^\prime}\|_p\nonumber\\
&\leq& V_h^{-1+1/p}
\|K- K^\prime\|_p+\|K^\prime_h - K^\prime_{h^\prime}\|_p \nonumber\\
&\leq&V_h^{-1+1/p}\biggl[ \| K- K^\prime\|_\infty
+\mathrm{k}_\infty\biggl(\frac{V_h}{V_{h^\prime}}-1\biggr)\biggr]\nonumber\\
&&{} +
V^{-1}_{h^{\prime}}\|K^\prime(\cdot/h) - K^\prime(\cdot/h^\prime)\|
_p\\
&\leq&(V_{h\vee h^{\prime}})^{-1+1/p}\biggl[\frac{V_{h\vee h^{\prime
}}}{V_{h\wedge h^{\prime}}}\biggr]\biggl[ \| K- K^\prime\|_\infty
+\mathrm{k}_\infty\biggl(\frac{V_{h\vee h^{\prime}}}{V_{h\wedge h^{\prime
}}}-1\biggr)\biggr]\nonumber\\
&&{} +
(V_{h\vee h^{\prime}})^{-1+1/p}\biggl[\frac{V_{h\vee h^{\prime}}}{V_{h\wedge
h^{\prime}}}\biggr]\nonumber\\
&&\hspace*{11pt}{}\times\|K^\prime(\cdot[h\vee h^{\prime}]/h) - K^\prime(\cdot
[h\vee h^{\prime}]/ h^\prime)\|_p,\nonumber
\end{eqnarray}
where $h\wedge h^{\prime}=(h_1\wedge h_1^{\prime},\ldots, h_d\wedge
h_d^{\prime})$.
The second term of the last inequality is obtained using the evident
change-of-variables $t\mapsto t/[h\vee h^\prime]$ (the division
is understood in the coordinate-wise sense).

Note that all coordinates of the vectors
$[h\vee h^\prime]/h$ and $[h\vee h^\prime]/h^{\prime}$ are greater or equal
to $1$. Therefore,
in view of Assumption \hyperlink{assK1}{(K1)}
the integration (or supremum if $p=\infty$) over the whole $\bR^d$ in
$\|K^\prime(\cdot[h\vee h^{\prime}]/h) - K^\prime(\cdot[h\vee h^{\prime
}]/ h^\prime)\|_p$ can be
replaced by the integration (supremum) over the support of $K^\prime$.
Together with Assumption \hyperlink{assK1}{(K1)}, this yields
%
%
\begin{eqnarray}
\label{eq4:proof_lemma9}\qquad
&&\|K^\prime(\cdot[h\vee h^{\prime}]/h) - K^\prime(\cdot[h\vee h^{\prime
}]/ h^\prime)\|_p\nonumber\\[-8pt]\\[-8pt]
&&\qquad\leq L_\cK\sqrt{\frac{1}{4}
\sum_{j=1}^{d}
\biggl[\frac{h_j\vee h_j^\prime}{h_j\wedge h_j^\prime}-1\biggr]^2}
\leq 2^{-1}L_\cK\sqrt{d}
\bigl(\exp\{\Delta_\cH(h,h^\prime)\}-1\bigr).\nonumber
\end{eqnarray}
Noting that
$V_{h\vee h^{\prime}}/V_{h\wedge h^{\prime}}\leq\exp\{d \Delta_\cH
(h,h^\prime)\}$
we obtain from (\ref{eq1:proof_lemma9}) and
(\ref{eq4:proof_lemma9})
that
%
%
\begin{eqnarray}\label{eq6:proof_lemma9}\qquad
&&\|K_{h}-K^\prime_{h^\prime}\|_p
\nonumber\\
&&\qquad \leq
(V_{h\vee{h^\prime}})^{({1-p})/{p}}
e^{d\Delta_\cH(h,h^\prime)}\biggl[
\|K-K^\prime\|_\infty +
\mathrm{k}_\infty\bigl(e^{d\Delta_\cH(h,h^\prime)}-1\bigr)\\
&&\qquad\quad\hspace*{157.6pt}{}
+ \frac{L_\cK\sqrt{d}}{2}
\bigl(e^{\Delta_\cH(h,h^\prime)}-1\bigr)\biggr].
\nonumber
\end{eqnarray}
Then (\ref{eq:phi1-1}) follows
from the last inequality and the
monotonicity of the function $D(\cdot)$.

2$^0$. Now we turn to the proof of (\ref{eq:phi2-2}).
Recall that $z=[(K,h),(Q,\mh)]$ and $z^\prime=[(K^\prime, h^\prime),
(Q^\prime, \mh^\prime)]$.
For brevity, we also write $\zeta_K=(K,h)$ and $\zeta_Q=(Q,\mh)$ with evident
changes in notation for $\zeta^\prime_K$ and $\zeta^\prime_Q$.

By the triangle inequality, we have
%
%
\begin{eqnarray}\label{eq:ineq-00}
\|K_h* Q_{\mh} - K^\prime_{h^\prime}*Q^\prime_{\mh^\prime}\|_p &\leq&
\|K_h*Q_\mh- K_h*Q^\prime_{\mh^\prime}\|_p \nonumber\\[-8pt]\\[-8pt]
&&{}+ \| K_h*Q^\prime_{\mh
^\prime}
- K^\prime_{h^\prime}*Q^\prime_{\mh^\prime}\|_p.\nonumber
\end{eqnarray}
Using the Young inequality (the first statement of Lemma \ref
{folland}), Assumption~\hyperlink{assK1}{(K1)} and (\ref
{eq6:proof_lemma9}) we obtain
\begin{eqnarray*}
\|K_h*Q_\mh- K_h*Q^\prime_{\mh^\prime}\|_p &\leq&\|K_h\|_1
\|Q_\mh- Q^\prime_{\mh^\prime}\|_p\\
&\leq&\mathrm{k}_\infty(V_{\mh\vee\mh^\prime})^{-1+1/p}
D\bigl(\mathrm{d}_1^{(1)}(\zeta_Q,\zeta^\prime_Q)\bigr).
\end{eqnarray*}
On the other hand, applying the Young inequality and
(\ref{eq6:proof_lemma9}) with $p=1$, we have
\begin{eqnarray*}
\| K_h* Q_\mh- K_h*Q^\prime_{\mh^\prime}\|_p &\leq& \|K_h\|_p
\|Q_\mh- Q^\prime_{\mh^\prime}\|_1
\leq\mathrm{k}_\infty V_h^{-1+1/p}
D\bigl(\mathrm{d}_1^{(1)}(\zeta_Q,\zeta^\prime_Q)\bigr)
\\
&\leq& \mathrm{k}_\infty(V_{h\vee h^\prime})^{-1+1/p}
\exp\{d \Delta_\cH(h,h^\prime)\}
D\bigl(\mathrm{d}_1^{(1)}(\zeta_Q,\zeta_Q^\prime) \bigr)
\\
&\leq& \mathrm{k}_\infty(V_{h\vee h^\prime})^{-1+1/p}
D\bigl(2\mathrm{d}_1^{(2)}(z,z^\prime) \bigr),
\end{eqnarray*}
where we have used the definition of $\Delta_\cH(\cdot,\cdot)$ and
monotonicity of the function $D(\cdot)$.
Combining the last two inequalities, we have
\[
\|K_h*Q_\mh- K_h*Q^\prime_{\mh^\prime}\|_p
\leq\mathrm{k}_\infty
[(V_{h\vee h^\prime})\vee(V_{\mh\vee\mh^\prime})]^{-1+1/p}D\bigl(2\mathrm{
d}^{(2)}_1(z,z^\prime)\bigr).
\]

Repeating the previous computations, we obtain the same bound for the
second term on the right-hand side of (\ref{eq:ineq-00}), namely,
\[
\|K_h*Q^\prime_{\mh^\prime} - K^\prime_{h^\prime}*
Q^\prime_{\mh^\prime}\|_p
\leq\mathrm{k}_\infty
[(V_{h\vee h^\prime})\vee(V_{\mh\vee\mh^\prime})]^{-1+1/p}
D\bigl(2\mathrm{d}^{(2)}_1(z,z^\prime)\bigr).
\]
Thus, we finally get
\[
\|K_h*Q_{\mh} - K^\prime_{h^\prime}*Q^\prime_{\mh^\prime}\|_p\leq
2\mathrm{k}_\infty
[(V_{h\vee h^\prime})\vee(V_{\mh\vee\mh^\prime})]^{-1+1/p}D\bigl(2
\mathrm{d}^{(2)}_1(z,z^\prime)\bigr),
\]
as claimed.

%

\subsection*{Proof of Lemma \protect\ref{lem:tech2}}
If $w\in\bH_d(1,P)$,
then for any
\[
x\in\bigotimes_{i=1}^{d}\biggl[\tilde{x}_i- \frac{\|w\|_\infty}{2P\sqrt{d}},
\tilde{x}_i+
\frac{\|w\|_\infty}{2P\sqrt{d}}\biggr]
\]
we have by the triangle inequality
\[
|w(x)| \geq|w(\tilde{x})|-
|w(x)-w(\tilde{x})| \geq\|w\|_\infty
- P|x-\tilde{x}| \geq
\tfrac{1}{2}\|w\|_\infty.
\]
This completes the proof.

\subsection*{Proof of Lemma \protect\ref{lem:tech204}}
Recall that
\[
\phi_2[z](t)= (K_h*Q_\mh)(t)=\int K_h(t-y)Q_\mh(y)\,\rd y,\qquad
t\in\bR^{d}.
\]

1$^0$.
Let $\cJ$ denote the set of indexes $j\in\{1,\ldots,d\}$
such that $h_j\leq\mh_j$:
\[
\cJ:=\{j\in(1,\ldots, d)\dvtx h_j\leq\mh_j\}.\vadjust{\goodbreak}
\]
Given two arbitrary vectors $u, v\in\bR^d$, let
$\Delta[u,v]$ and $\delta[u,v]$ denote the vectors in $\bR^d$
with the coordinates
\[
\Delta_j[u,v]=\cases{
u_j, &\quad$j\in\cJ$,\cr
v_j, &\quad$j\notin\cJ$,}\qquad
\delta_j[u,v]=\cases{
u_j, &\quad$j\notin\cJ$,\cr
v_j, &\quad$j\in\cJ$.}
\]
With this notation, we can write
\[
(K_h*Q_\mh)(t) = \frac{1}{V_hV_\mh}\int K\biggl(\Delta\biggl[\frac{t-v}{h},\frac{v}{h}\biggr]\biggr)
Q\biggl(\delta\biggl[\frac{t-v}{\mh},\frac{v}{\mh}\biggr]\biggr)\,\rd
v,\qquad
t\in\bR^{d}.
\]
Then changing the variables
$v\mapsto u=(t-v)/(h\wedge h^\prime)$
and setting for brevity $\eta=(h\wedge\mh)/(h\vee\mh)$ (as usual,
all operations are understood in the coordinate-wise sense),
we come
to the formula
%
%
\begin{eqnarray}\label{eq1:proof_lem-new}\qquad
&&(K_h*Q_\mh)(t) \nonumber\\
&&\qquad= \frac{V_{h\wedge\mh}}{V_hV_\mh}
\int K\bigl(\Delta[u, t/(h\vee\mh) - \eta u]\bigr)
Q\bigl(\delta[t/(h\vee\mh) -\eta u, u]\bigr)\,\rd u
\\
&&\qquad= \frac{1}{V_{h\vee\mh}} F\biggl(\frac{t}{h\vee\mh}\biggr),
\nonumber
\end{eqnarray}
where we have denoted
%
%
\begin{equation}
\label{eq100:proof_lem-new}
F(t):=\int K(\Delta[u, t- \eta u])
Q(\delta[t-\eta u, u])\,\rd u,\qquad
t\in\bR^d.
\end{equation}

Now we note some
properties of the function $F$ that will be useful in the sequel.
First, Assumption~\hyperlink{assK1}{(K1)} implies that
the integration over $\bR^d$ in~(\ref{eq100:proof_lem-new})
can be replaced by the integration over $[-1/2,1/2]^d$.
Indeed, if at least one of the coordinates of $u$ lies outside
the interval $[-1/2,1/2]$ then, in view of~\hyperlink{assK1}{(K1)}, one
of the functions
$K$ or $Q$ vanishes.
This fact along with Assumption~\hyperlink{assK2}{(K2)} and~(\ref
{eq100:proof_lem-new})
imply that
$\|F\|_\infty\leq\mathrm{k}^{2}_{\infty}$;
in addition,
%
%
\begin{equation}\label{eq4:proof_lem-new}
\operatorname{supp}(F)\subseteq[-1,1]^{d}.
\end{equation}
Taking into account these facts and using
(\ref{eq1:proof_lem-new}),
we obtain
\[
\|K_h*Q_\mh\|_p \leq(V_{h\vee\mh})^{-1+1/p}\|F\|_p \leq
2^{d/p} \mathrm{k}^{2}_{\infty}(V_{h\vee\mh})^{-1+1/p},
\]
and the statement (i) of the lemma is proved.

To get the assertion (ii) of the lemma, we note that
\begin{eqnarray*}
\biggl|\int F(t) \,\rd t \biggr| &=& \biggl|\int\!\!
\int K(\Delta[u, t- \eta u])
Q(\delta[t-\eta u, u])\,\rd u
\,\rd t\biggr|
\\
&=& \biggl| \int K(x) \,\rd x\biggr| \biggl|\int Q(x)\,\rd x\biggr| \geq
\mathrm{k}_1^2.
\end{eqnarray*}
The second equality follows from the fact that functions $K$ and $Q$
are integrated over $t$
and over $u$ over disjoint sets of components; and\vadjust{\goodbreak} the last inequality
is a consequence of \hyperlink{assK2}{(K2)}.
Therefore, invoking (\ref{eq4:proof_lem-new}) we have
\begin{eqnarray*}
\|G\|_p&=&(V_{h\vee\mh})^{-1+1/p}\|F\|_p\geq(2^{d}V_{h\vee\mh
})^{-1+1/p}\|F\|_1\\
&\geq&2^{{d(1-p)}/{p}} \mathrm{k}^{2}_{1}(V_{h\vee
\mh})^{-1+1/p},
\end{eqnarray*}
as claimed in the statement (ii) of the lemma.

2$^0$. Now
we turn to the proof of the statements (iii) and (iv) of the lemma.
The idea in the proof of these statements is to show that
$F$ satisfies the Lipschitz condition and then to apply
Lemma \ref{lem:tech2}.

By (\ref{eq100:proof_lem-new})
for any $x,y\in\bR^{d}$, we have
%
%
\begin{eqnarray}\label{eq6:proof_lem-new}\qquad
|F(x)-F(y)|&\leq& \mathrm{k}_{\infty}
\sup_{u\in[-{1}/{2},{1}/{2}]^{d}}
|K(\Delta[u, x- \eta u])
-
K(\Delta[u, y- \eta u])
|
\nonumber
\\
&&{}
+ \mathrm{k}_\infty\sup_{u\in[-{1}/{2},{1}/{2}]^{d}}
|Q(\delta[x-\eta u, u])- Q(\delta[y-\eta u, u])|
\nonumber\\[-8pt]\\[-8pt]
&\leq&L_\cK\mathrm{k}_{\infty}
\Biggl\{\sqrt{\sum_{j\notin\cJ}(x_j-y_j)^{2}}+
\sqrt{\sum_{j\in\cJ}(x_j-y_j)^{2}}\Biggr\} \nonumber\\
&\leq&
2L_\cK\mathrm{k}_{\infty}|x-y|.
\nonumber
\end{eqnarray}
The obtained inequality means that $F\in\bH_d(1,P)$ with
$P=2L_\cK\mathrm{k}_{\infty}$;
moreover, (\ref{eq4:proof_lem-new}) implies that
%
%
\begin{equation}
\label{eq7:proof_lem-new}
\|F\|_\infty\geq2^{-d} \mathrm{k}^{2}_{1}.
\end{equation}
Applying Lemma \ref{lem:tech2} and using (\ref{eq6:proof_lem-new}),
we obtain
\[
\biggl\{x\in\bR^{d}\dvtx F(x)\geq\frac{1}{2}\|F\|_\infty\biggr\} \supseteq
\bigotimes_{i=1}^{d}\biggl[\tilde{x}_i-\frac{\|F\|_\infty}{2P\sqrt{d}}, \tilde
{x}_i+\frac{\|F\|_\infty}{2P\sqrt{d}}\biggr],
\]
where, recall, $F(\tilde{x})=\|F\|_\infty$.
Using (\ref{eq7:proof_lem-new}), we obviously deduce from
(\ref{eq1:proof_lem-new})
that
\begin{eqnarray*}
&&\biggl\{x\dvtx(K_h*Q_\mh)(x)\geq\frac{1}{2}
\|K_h*Q_\mh\|_\infty\biggr\} \\
&&\qquad\supseteq
\bigotimes_{i=1}^{d}
\biggl[\tilde{x}_i (h\vee\mh)_i -\frac{\mathrm{k}^{2}_{1}
(h\vee\mh)_i }{2^{d+1}P\sqrt{d}},
\tilde{x}_i(h\vee\mh)_i+
\frac{\mathrm{k}^{2}_{1}(h\vee\mh)_i}{2^{d+1}P\sqrt{d}}\biggr],
\end{eqnarray*}
which implies that
%
%
\begin{equation}
\label{eq8:proof_lem-new}\hspace*{28pt}
\operatorname{mes}\biggl\{x\dvtx(K_h*Q_\mh)(x)\geq\frac{1}{2}
\|K_h*Q_\mh\|_\infty\biggr\}
\geq V_{h\vee\mh} \biggl[\frac{\mathrm{k}^{2}_{1}}{2^{d+1}\sqrt{d}
L_\cK\mathrm{k}_{\infty}}\biggr]^{d}.
\end{equation}
Then the statement (iii)
of the lemma follows because
\[
\operatorname{mes}\{\operatorname{supp}(K_h*Q_\mh)\}
\geq\operatorname{mes}\bigl\{x\dvtx(K_h*Q_\mh)(x)\geq
\tfrac{1}{2}\|K_h*Q_\mh\|_\infty\bigr\}.
\]

It remains to note that (\ref{eq4:proof_lem-new}) implies that $\mathrm
{mes}\{\operatorname{supp}(K_h*Q_\mh)\}\leq2^{d} V_{h\vee\mh}$.
Therefore by (\ref{eq8:proof_lem-new}),
\begin{eqnarray*}
&&\operatorname{mes}\biggl\{x\dvtx(K_h*Q_\mh)(x)\geq\frac{1}{2}
\|K_h*Q_\mh\|_\infty\biggr\}\\
&&\qquad\geq\biggl[\frac{\mathrm{k}^{2}_{1}}{2^{d+2}\sqrt{d}L_\cK
\mathrm{k}_{\infty}}\biggr]^{d}\operatorname{mes}\{\operatorname
{supp}(K_h*Q_\mh)
\}.
\end{eqnarray*}
This completes the proof of the lemma.
\end{appendix}

\section*{Acknowledgments}
We thank two anonymous referees for useful comments that
led to significant improvements in the presentation.

%

%
\printaddresses

\end{document}